\documentclass[sort&compress,3p]{elsarticle}

\usepackage{lineno,hyperref}
\journal{Journal of \LaTeX\ Templates}

\usepackage{numcompress}

\usepackage{comment}

\usepackage{lipsum}
\usepackage{amsfonts}
\usepackage{amsmath,amssymb}
\usepackage{amsthm}
\usepackage{fourier}
\usepackage{xcolor}
\usepackage{mathrsfs}
\usepackage{graphicx}
\usepackage{epstopdf}
\usepackage[ntheorem]{empheq} % this loads amsmath as well
\empheqset{box=fbox}
\usepackage[final]{pdfpages}
\usepackage{subcaption}
\usepackage{amsopn}

\DeclareMathOperator*{\argmin}{arg\,min}

\def\tred{\textcolor{black}}

\def\tpur{\textcolor{black}}
\newcommand{\Vn}{\mathbf{n}}
\newcommand{\Vb}{\mathbf{b}}

\newcommand{\Vx}{\mathbf{x}}
\newcommand{\R}{\mathbb{R}}
\newcommand{\Pol}{\mathbb{P}}

\newcommand{\Par}{\mathcal{P}}
\newcommand{\Sk}{\mathcal{S}}
\renewcommand{\div}{\mathrm{div}}

\newcommand{\udg}{u_h^\text{dG}}
\newcommand{\vdg}{v_h^{\text{dG},\ast}}
\newtheorem{ssmptn}{Assumption}
\newtheorem{lmm}{Lemma}
\newtheorem{prpstn}{Proposition}

\newtheorem{thrm}{Theorem}
\newtheorem{rmrk}{Remark}
\newtheorem{crllr}{Corollary}
\newcommand{\llbrace}{\lbrace}
\newcommand{\rrbrace}{\rbrace}

\begin{document}

\begin{frontmatter}
%\title{Stabilized Finite Element Method with Discontinuous Galerkin based residual minimization}
\title{\tred{Goal-oriented adaptivity for a \tpur{conforming residual minimization method in a dual discontinuous Galerkin norm}}}

\author[one,two]{Sergio Rojas\corref{mycorrespondingauthor}}
\cortext[mycorrespondingauthor]{Corresponding author}
\ead{srojash@gmail.com}
%
%\author[four,five]{Alexandre Ern}
%
\author[seven,eight,nine]{David Pardo}
\author[three]{Pouria Behnoudfar}
\author[one,two,ten]{Victor M. Calo}
\address[one]{Curtin Institute for Computation and School of Electrical Engineering, 
Curtin University, P.O. Box U1987, Perth, WA 6845, Australia
}
\address[two]{Computing and Mathematical Sciences,
Curtin University, P.O. Box U1987, Perth, WA 6845, Australia
}
\address[seven]{University of the Basque Country (UPV/EHU), Leioa, Spain}
\address[eight]{BCAM - Basque Center for Applied Mathematics, Bilbao, Spain}
\address[nine]{IKERBASQUE, Basque Foundation for Science, Bilbao, Spain}
\address[three]{School of Earth and Planetary Sciences, Curtin University, Kent Street, Bentley, Perth, WA 6102, Australia}
\address[ten]{Mineral Resources, Commonwealth Scientific and Industrial Research Organisation (CSIRO), Kensington, Perth, WA 6152, Australia}

\begin{abstract}

\tpur{
We propose a goal-oriented mesh-adaptive algorithm for a finite element method stabilized via residual minimization on dual discontinuous-Galerkin norms.  By solving a saddle-point problem,  this residual minimization delivers a stable continuous approximation to the solution on each mesh instance and a residual projection onto a broken polynomial space, which is a robust error estimator to minimize the discrete energy norm via automatic mesh refinement. In this work, we propose and analyze a goal-oriented adaptive algorithm for this stable residual minimization. We solve the primal and adjoint problems considering the same saddle-point formulation and different right-hand sides. By solving a third stable problem, we obtain two efficient error estimates to guide goal-oriented adaptivity. We illustrate the performance of this goal-oriented adaptive strategy on advection-diffusion-reaction problems.}

\end{abstract}


\end{frontmatter}

%\linenumbers

\section{Introduction}
\tpur{Adaptive mesh refinements minimize the computational cost of solving boundary value problems when using grid-based numerical methods. In this class of methods, one performs refinements to reduce the error in an energy norm of the problem (see e.g.,~\cite{  babuska1986accuracy}). This energy-norm-driven mesh adaptivity reduces the computational cost of the simulations significantly. Nevertheless, this adaptive strategy is often suboptimal since many engineering problems seek only to approximate a quantity of interest (QoI), which may be only loosely related to the energy norm.  As a consequence,  in the late 90's, the goal-oriented adaptivity (GoA) methodology arose to tackle this problem (see, e.g.,~\cite{  becker1996weighted, prudhomme1999goal, oden2001goal, romkes2006multi, feischl2016abstract, darrigrand2018goal}). In GoA, we first construct an adjoint problem (see~\cite{  oden2001goal}).  We then, represent the error in the quantity of interest as an integral over the entire computational domain that depends upon the solutions of both direct and adjoint problems.} Finally, we use the direct and adjoint solutions to build adequate a posteriori error estimators.

\tpur{A crucial limitation of standard} finite element methods is that they can suffer from instability on coarse meshes \tred{(e.g., advection-dominated problems or first-order partial differential equations)}. Thus, adaptivity with standard finite elements is not feasible in several scenarios since stability is critical for a posteriori error estimation. Several alternative Galerkin methods exist that are stable on coarse meshes (see, e.g.,~\cite{ ern2013theory, Ern_Guermond_16, HUGHES2017} and references therein). 

\tpur{In~\cite{  rojas2019adaptive}, the authors present an adaptive-stabilized finite element method (FEM), which combines the idea of residual minimization (a core idea of several stabilization methods since the '60s; e.g., least-squares FEM (LS-FEM)~\cite{ Diska:68, Lucka:69, BrSc70, Jiang99} and Galerkin/least-squares (GaLS) method~\cite{  HFH89}) with the discontinuous Galerkin (dG) mathematical framework, where stability derives from the enlargement of the continuous trial space with discontinuous functions and the addition of penalization terms~\cite{  ReeHi:73, LesRa:74, JohPi:86, CoKaS:00, BrMaS:04, Ern2006, di2011mathematical}.  The adaptive-stabilized finite element method (AS-FEM)~\cite{ rojas2019adaptive} minimizes a discrete residual in a dG norm.  That is, starting from an inf-sup stable dG formulation, the method minimizes the residual in an adjoint norm (to the dG test functions) over a conforming trial space (e.g., a standard finite element space). This residual minimization problem is equivalent to a saddle-point problem that inherits the dG inf-sup stability. Thus, the method delivers solutions of the same quality as those associated with the underlying dG formulation. From a practical point of view, the resulting mixed formulation delivers two significant benefits: A stable approximation of the solution in the trial space of continuous (conforming) functions, and a projection of the residual onto the discontinuous dG test space (error estimate). This method has similarities with the Discontinuous Petrov-Galerkin (DPG) methods since both technologies minimize the residual in a non-standard norm (see, e.g.,~\cite{  demkowicz2010class, demkowicz2012class, demkowicz2011class, zitelli2011class, DemHeuSINUM2013, ChaHeuBuiDemCAMWA2014, DemGopBOOK-CH2014}).  Nevertheless, AS-FEM builds on non-conforming (dG) formulations, which allows us to use stronger norms than those of DPG when the trial space contains continuous functions; application examples include diffusive-advective-reactive problems~\cite{  cier2020automatic}, incompressible Stokes flows~\cite{  los2020stable, kyburg2020incompressible}, continuation analysis of compaction bandings in geomaterials~\cite{   cier2020adaptive}, and weak constraint enforcement for advection-dominated diffusion problems~\cite{  cier2020nonlinear}.}

\tpur{Herein, we extend and analyze the method proposed in~\cite{ rojas2019adaptive} to goal-oriented adaptivity (GoA). We describe a general theory that applies to any problem where a well-posed discontinuous Galerkin formulation for the primal problem is available and demonstrate its numerical performance for advection-diffusion-reaction problems. We define the corresponding discrete adjoint system as a saddle-point problem, where its solution is constrained by the conforming space and is discontinuous across element faces. The same dG inf-sup arguments guarantee the well-posedness of this adjoint saddle-point problem. Solving both the primal and the adjoint problem requires the solution of a single saddle-point problem with two right-hand sides. Unfortunately, the discontinuous adjoint solution satisfies extra constraints imposed by the conforming variable, which does not estimate the residual.  We propose two alternative stable discrete problems that allow us to measure the error of the adjoint discrete problem. As a result, we obtain an automatic, easy-to-implement, and stable GoA strategy that complements the formulation of~\cite{ rojas2019adaptive}. Moreover, we show through numerical experimentation that these two strategies deliver the optimal GoA convergence rates for diffusion problems~\cite{ feischl2016abstract}.}
 
\tpur{Our GoA strategy is similar to a recent DPG theory~\cite{  keith2019goal}, where they solve the adjoint problem in terms of the original saddle-point formulation with a different right-hand side.  For advection-diffusion-reaction, several works explore the use of conforming FEM stabilization schemes~(see~\cite{  formaggia2004anisotropic, formaggia2001anisotropic, burman2017error, kuzmin2010goal, bruchhauser2017numerical}). Typically, in conforming schemes, the Dual Weighted Residual (DWR) method~\cite{  becker1996feed, becker2001optimal} allows for an efficient post-processing strategy. Recently,~\cite{  dolejvsi2017goal, bartovs2019goal} proposed a DWR method for the discontinuous Galerkin SIP method~\cite{  arnold1982interior}. Finally,~\cite{  mozolevski2015goal} considered a conforming approximation of the primal problem and a dG approximation of the adjoint problem, where the GoA estimates combine the DWR and equilibrated-flux reconstruction methods.}

\tpur{The remainder of the paper continues as follows. Section~\ref{sec:model} introduces the advection-diffusion-reaction model problem, together with its discontinuous Galerkin formulation to facilitate the understanding of the framework we introduce in the next sections.  Section~\ref{sec:mixed_primal} describes the adaptive stabilized finite element method that defines the discrete primal problem.  Section~\ref{sec:adjoint_problems} introduces the adjoint problems and Section~\ref{sec:goal_estimator} details the a-posteriori error analysis for the proposed strategy. Section~\ref{sec:GoA_strategy} describes the GoA algorithm,  while Section~\ref{sec:numerical_examples}, shows the performance of the method in advection-diffusion-reaction problems. Finally, Section~\ref{sec:conclusion} details our contributions and describes future lines of work.}

\section{Model problem: A advection-diffusion-reaction problem}\label{sec:model}
\subsection{Continuous setting}
Let $\Omega \subset \R^d$, with $d = 2,3$, be an open and bounded Lipschitz domain with boundary $\Gamma :=\partial \Omega$, and denote by $\Vn$ its outward unit normal vector. Using the standard notation for Sobolev and Lebesgue spaces and their norms, let $\kappa > 0$ be a diffusion coefficient in $\Omega$, $\Vb \in \left[W^{1,\infty}(\Omega)\right]^d$ a divergence-free advective velocity field, and $\gamma \geq 0$ a reaction coefficient. As model problem, we consider the following advection-diffusion-reaction problem:
\begin{align}\label{eq:scalar}
\left\{\begin{array}{l}
\text{Find } u \text{ such that:} \smallskip \\
\begin{array}{rl}
-\div \big( \kappa \nabla u \big) + \Vb \cdot \nabla u + \gamma \, u = f, & \text{ in } \Omega, \smallskip\\
%\left\llbracket u\right\rrbracket_e = 0, & \text{ on } e \in \I^{\kappa}, \smallskip\\
%\left\llbracket\kappa \nabla u \right\rrbracket_e = 0, & \text{ on } e \in \I^{\kappa}, \smallskip\\
u=g_D, & \text{ on } \Gamma, % \smallskip\\
%\kappa \nabla u \cdot \Vn + \alpha \, u = g_N, & \text{ on } \Gamma_N,
\end{array}
\end{array}\right.
\end{align}
where $f \in L^2(\Omega)$ is a spatial source and $g_D \in H^{1/2}(\Gamma)$ defines the boundary data. If $g_D \equiv 0$, the weak variational formulation of problem~\eqref{eq:scalar} reads:
\begin{equation}\label{eq:cont_dar}
\left\{\begin{array}{l}
\text{Find } u \in U,  \text{ such that:}\\
\displaystyle b(u, v) := \int_{\Omega} \big(\kappa \nabla u \cdot \nabla v - u \, (\Vb \cdot \nabla v) + \gamma u \, v \big) = l(v):=\int_\Omega f v, \quad \forall \, v \in V,
\end{array}
\right.
\end{equation}
with energy space $U = V := H_0^1(\Omega)= \left\{v \in H^1(\Omega) \, : \, v|_{\Gamma} = 0 \right\}$.
%, and bilinear form:
%
%\begin{equation}\label{eq:bilinear_dar_cont}
%b(u,v) = \int_{\Omega} \big(\kappa \nabla u \cdot \nabla v - u \, (\Vb \cdot \nabla v) + \gamma u \, v \big).
%\end{equation}
%
Denoting by $\displaystyle \|v\|^2_{\Omega} = \int_{\Omega} v^2$ the standard $L^2$-norm, and considering the $H^1$-norm:
\begin{equation}
\|v\|^2_{1,\Omega} := \|\nabla v\|^2_{\Omega} + \|v\|^2_{\Omega},
\end{equation}
 a straightforward verification shows that the bilinear form $b(\cdot,\cdot)$ in~\eqref{eq:cont_dar} is bounded in $V \times V$, that is, there exists a constant $M > 0$, such that:
\begin{equation}
b(u,v)\leq M \|u\|_{1,\Omega}\|v\|_{1,\Omega}, \quad \forall \, u,v \in V.
\end{equation}
This bilinear form is also coercive in $V$, that is, there exists a constant $\alpha>0$, such that:
\begin{equation}\label{eq:coercivity_cont}
b(v,v) \geq \alpha \|v\|^2_{1,\Omega}, \quad \forall v\in V.
\end{equation}
Since $\Vb$ is divergence-free, it holds:
\begin{equation}\label{eq:int_by_parts_adv}
\int_\Omega w  \, (\Vb \cdot \nabla v) = -\int_\Omega (\Vb \cdot \nabla w)  \,  v, \quad \forall w,v \in V,
\end{equation}
and thus $\displaystyle \int_{\Omega} v \, (\Vb \cdot \nabla v) = 0$, for all $v \in V$. Therefore, well-posedness for the weak problem~\eqref{eq:bilinear_dar} is proved by evoking the Lax-Milgram theorem (cf.~\cite{ evans10}). Finally, when $g_D$ does not vanish in $\Gamma$, the weak variational formulation of problem~\eqref{eq:scalar} is equivalently obtained in terms of the auxiliary variable $\widetilde{u} := u - w_D$, with $w_D \in H^1(\Omega)$ a function satisfying $w_D = g_D$ in $H^{1/2}(\Gamma)$. The existence of $w_D$ is guaranteed as consequence of the surjectivity of the Dirichlet map from $H^1(\Omega)$ to $H^{1/2}(\Gamma)$.  
\subsection{Discrete setting}\label{sec:disc_sett}
\begin{figure}[t!]
\centering
\includegraphics[scale=0.9]{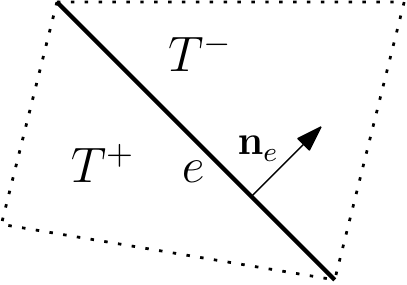}
\caption{Skeleton orientation over the internal face $e = \partial T^+ \cap \partial T^-$.}
\label{fig:skeleton}
\end{figure}
Let $\Par_h$ be a conforming partition of the domain $\Omega$ into open elements $T \subset \Omega$, such that:
\begin{equation}
\Omega_h := \bigcup_{T \in \Par_h} T, \quad \text{ satisfies } \quad \Omega = \text{int}\big(\overline{\Omega_h}\big).
\end{equation}
For any $T\in \Par_h$, we denote by $h_T$ its diameter and by $\partial T$ its boundary. Let $\Sk^0_h$ be the set of interior edges/faces obtained from the intersection of two adjoint elements, as shown in Figure~\ref{fig:skeleton}. Also, $\Sk^\partial_h$ are the edges/faces that belong to the boundary $\partial \Omega_h$, and $\Sk_h := \Sk^0_h \cup \Sk^\partial_h$ form the skeleton of $\Omega_h$.
%
%Over each $T \in \Par_h$, and $e \in \Sk_h$, we define the following inner products:
%\begin{eqnarray}
%(u,v)_T := \int_T u \, v,  \quad \forall u,v \in L^2(T), & \displaystyle <\phi,\psi>_e := \int_e \phi \, \psi, \quad \forall \phi, \psi \in L^2(e),
%\end{eqnarray}
%
For $k\geq1$, we define the following standard broken Hilbert spaces over $\Omega_h$:
\begin{align}
H^k(\Omega_h) &:= \left\{ v \in L^2(\Omega) \, : \, v \in H^k(T), \, \forall \, T \in \Par_h \right\}.
\end{align}
For any $v_h \in H^1(\Omega_h)$, we define the jump function $\llbracket v_h \rrbracket(\Vx)$ and the average function $\llbrace v_h \rrbrace(\Vx)$ respectively, restricted to an interior face/edge $ e = \partial K^+ \cap \partial K^- \in \Sk^0_h$, as:
\begin{align}
\llbracket v_h \rrbracket(\Vx)\vert_e := v^{+}_h(\Vx) - v^{-}_h(\Vx), \quad
\llbrace v_h \rrbrace(\Vx)\vert_e := \dfrac{1}{2} (v^+_h(\Vx) + v^-_h(\Vx)), \quad \forall e \in \Sk_h^0, \label{eq:jump_and_average}
\end{align} 
where $v^+_h$, $v^-_h$ denote the traces over $e = \partial T^+ \cap \partial T^-$ with respect to a predefined normal ${\bf n}_e$, as shown in Figure~\ref{fig:skeleton}. We extend the definitions of~\eqref{eq:jump_and_average} to $e \in \Sk_h^\partial$ by setting:
\begin{align}
\Vn_e(\Vx) := \Vn(\Vx), \quad \llbracket v_h \rrbracket(\Vx)\vert_e := \llbrace v_h \rrbrace(\Vx)\vert_e := v_h(\Vx)|_e \quad \forall e \in \Sk_h^\partial, \label{eq:jump_and_average_border}
\end{align} 
where $\Vn$ denotes the outward normal to $\partial \Omega$.  Let $\Pol^{p_t}(T)$ be the set of polynomials of degree $p_t\geq 1$ over the element $T$. We consider the following broken polynomial space:
\begin{align}
\Pol^{p_t}(\Omega_h) &:= \left\{ v \in L^2(\Omega) \, : \, v\vert_{T} \in \Pol^{p_t}(T), \, \forall \, T \in \Par_h \right\}.
\end{align}
For a given ${p_t} \geq 1$, we denote $V_h := \Pol^{p_t}(\Omega_h)$ endowed with a discrete norm $\|\cdot \|_{V_h}$.
\subsection{DG variational formulations}\label{sec:model_dg}
We now introduce two possible dG formulations (in primal form) for advection-diffusion-reaction problems. We analyze their inf-sup stability in terms of a discrete norm we generalize from~\cite{ shahbazi2005explicit,ayuso2009discontinuous, di2008discontinuous, riviere2008discontinuous}.  For detailed discussions of dG schemes for elliptic problems, see for instance~\cite{ arnold2002unified,di2011mathematical,riviere2008discontinuous}.
\tred{For a given polynomial degree $p_t\geq1$ and $\epsilon =\pm 1$, we consider a dG formulation for  problem~\eqref{eq:cont_dar} of the form:}
\begin{equation}\label{eq:dg_load}
\left\{\begin{array}{l}
\text{Find } \udg \in V_h,  \text{ such that:} \\
b_h(\udg, v_h) = l_h (v_h), \quad \forall \, v_h \in V_h,
\end{array}
\right.
\end{equation}
where the discrete linear form $l_h(\cdot)$, and bilinear form $b_h(\cdot,\cdot)$ are defined as:
\begin{align}\label{eq:linear_dar}
l_h(v_h) = \sum_{T\in \Par_h}\int_T f \, v_h + \sum_{e \in \Sk^\partial_h} \Big(\epsilon \int_e \kappa \nabla v_h \cdot n_e \, g_D + \tred{\int_e \big(( \Vb \cdot \Vn_e)^\ominus + \kappa \, \eta_e \big) \, g_D \, v_h \Big)},
\end{align}
and
\begin{equation}\label{eq:bilinear_dar}
b_h(w_h,v_h) := b_h^{\epsilon}(w_h,v_h) + b_h^{\textrm{up}}(w_h,v_h),
\end{equation} 
with
\begin{align}\label{eq:sip_bh}
b_h^{\epsilon}(w_h,v_h) &:= \sum_{T \in \Par_h} \int_T \kappa \nabla w_h \,  \nabla v_h
+ \sum_{e \in \Sk_h} \int_e \kappa \, \eta_e \, \llbracket w_h \rrbracket \,  \llbracket v_h \rrbracket \nonumber\\
&  - \sum_{e \in \Sk_h}\int_e \llbrace \kappa \nabla w_h \rrbrace \cdot n_e  \,  \llbracket v_h \rrbracket + \epsilon \int_e \llbracket w_h \rrbracket \,   \llbrace \kappa \nabla v_h \rrbrace \cdot n_e ,\\
\label{eq:up_bh}
b_h^{up}(w_h,v_h)  &:=   \sum_{T\in\Par_h}\int_T(\Vb \cdot \nabla  w_h + \gamma \,  w_h) \, v_h + \sum_{e \in \Sk^\partial_h}\int_e ( \Vb \cdot \Vn_e)^\ominus w_h \,  v_h  
\nonumber\\ & 
- \sum_{e \in \Sk_h^0} \int_e(\Vb \cdot \Vn_e) \, \llbracket w_h \rrbracket \, \llbrace v_h \rrbrace + \dfrac{1}{2} \sum_{e \in \Sk_h^0} \int_e \big|\Vb \cdot \Vn_e\big| \, \llbracket w_h \rrbracket \, \llbracket v_h \rrbracket. 
\end{align}
\tred{Here, $x^\ominus = \dfrac{1}{2}(|x| - x) = \max\{-x,0\}$ denotes the negative part of $x$, and $\eta_e>0$ is a user-defined stabilization parameter that, following~\cite{ shahbazi2005explicit}, we define explicitly as
\begin{equation}
\eta_e  := \frac{(p_t+1)(p_t+d)}{d}\left\{\begin{array}{rl} 
\dfrac{1}{2}\big(\dfrac{\mathcal{A}(\partial T^+)}{\mathcal{V}(T^+)} + \dfrac{\mathcal{A}(\partial T^-)}{\mathcal{V}(T^-)}\big), & \text{ if } e = \partial T^+ \cap \partial T^-, \smallskip\\
\dfrac{\mathcal{A}(\partial T)}{\mathcal{V}(T)}, & \text{ if } e = \partial T \cap \partial \Omega_h,
\end{array} \right.
\end{equation}
where $\mathcal{A}, \mathcal{V}$ denote the area and volume, respectively, for $d=3$, and the length and area, respectively, for $d = 2$.}
\begin{rmrk}[Advection-dominated case]\label{rem:adv_dom}
When $\kappa << \|\Vb\|_{\infty,\Omega}$, the bilinear form $b_h$ behaves as the bilinear form $b_h^\textrm{up}$. Thus, the discrete solution is close to the discrete approximation of the following continuous problem:
\begin{align}\label{eq:advection_reaction}
\left\{\begin{array}{l}
\text{Find } u \text{ such that:} \smallskip \\
\begin{array}{rl}
\Vb \cdot \nabla u + \gamma \, u = f, & \text{ in } \Omega, \smallskip\\
u=g^-, & \text{ on } \Gamma^-, 
\end{array}
\end{array}\right.
\end{align}
where $\Gamma^-:=\{ \Vx \in \Gamma\, : \Vb(\Vx) \cdot n(\Vx) <0 \}$ is
the inflow boundary of $\Omega$ and $g^-=g_D |_{\Gamma^-}$ is the
Dirichlet boundary data restricted to $\Gamma^-$ with $(\Vb \cdot \Vn)^\ominus \neq 0$ on $\Gamma^-$ only. 
\end{rmrk}
The discrete formulation couples two independent inf-sup stable schemes. First, the upwind scheme (up) handles the advective-reaction part of the equation (see~\cite{ brezzi2004discontinuous} and Remark~\ref{rem:adv_dom}). Second, an $\epsilon$-dependent formulation encompasses two alternative schemes for the diffusion part: a) the Symmetric Interior Penalization (SIP) when $\varepsilon = -1$ (see~\cite{ arnold1982interior}), and b) the Non-symmetric Interior Penalization (NIP) when $\epsilon=1$ (see~\cite{ riviere1999improved}). \tred{The discrete formulation is also consistent as it satisfies the following Lemma:
  \begin{lmm}\label{lmm:dar_consistency}
    If the analytical solution $u$ of problem~\eqref{eq:cont_dar} belongs to the subspace $U_\# := U \cap H^2(\Omega)$, then:
    \begin{equation}\label{eq:dar_consistency}
      b_h(u\, , v_h) = l_h(v_h), \, \forall \, v_h \in V_h.
    \end{equation}
  \end{lmm}
 }\tpur{\begin{rmrk}[Conformity \& consistency]
 Conforming formulations satisfy the identity~\eqref{eq:dar_consistency}; however, the non-conformity of the discrete space $V_h$ implies that~\eqref{eq:dar_consistency} is not satisfied in general for all dG formulations (see e.g.,~\cite{ arnold2002unified}). 
\end{rmrk}
}
%However, the only one being also adjoint consistent is the one with $\varepsilon = -1$, since the NIP formulation is not adjoint consistent (cf.~\cite{ arnold2002unified}).
%
\begin{rmrk}[Generalizations]
We assume $\kappa$ to be constant, and $\Vb$ to be a divergence-free velocity field, and the solution $u$ to satisfy Dirichlet boundary conditions to simplify the discussion. The use of heterogeneous and non-isotropic diffusion coefficients, non-solenoidal advective fields, and non-homogeneous Robin-type boundary conditions require slight modifications of the bilinear and linear forms~\cite{ di2011mathematical}.
\end{rmrk}
\subsubsection{$V_h$-norm, inf-sup stability, boundedness, and a priori error estimates}\label{sec:inf_sup_dar}
For the discrete space $V_h$, we consider the following induced norm:
\begin{equation}\label{eq:vh_norm}
\|\cdot\|_{V_h}^2 = (\cdot\, ,\, \cdot)_{V_h}:= (\cdot\, ,\, \cdot)_{\epsilon} + (\cdot\, ,\, \cdot)_{\textrm{up}},
\end{equation}
 where, for any $w_h, v_h \in V_h$, we define: 
\begin{align}
  (w_h \, , \, v_h)_{\epsilon} := & \sum_{T \in \Par_h} \int_T \kappa \, 
                                    \nabla w_h \cdot \nabla v_h +
                                    \sum_{e \in \Sk_h} \int_e \kappa \,
                                    \eta_e \,
                                    \llbracket w_h \rrbracket \,
                                    \llbracket v_h \rrbracket, \\
  (w_h \, , \, v_h)_{\textrm{up}} := & \sum_{T \in \Par_h} \int_T
                                       (\gamma + \beta L^{-1} ) w_h \,
                                       v_h +\displaystyle 
                                       \sum_{T \in \Par_h} \beta_l \, h_T \int_T
                                       (\Vb \cdot \nabla w_h )( \Vb
                                       \cdot \nabla v_h)
  + \sum_{e \in \Sk_h} \dfrac{1}{2} \int_e |\Vb \cdot \Vn_e|
     w_h \,  v_h   ,   
\end{align}
with $\beta := \|\Vb\|_{\left[L^\infty(\Omega)\right]^d}$, $L$ is the diameter of the domain $\Omega$ (i.e., the diameter of the largest circumference contained in $\Omega$), and $\beta_l$ is defined as:
\begin{equation}
  \beta_l := \left\{ \begin{array}{rl}
                       \beta^{-1}, & \text{ if } \beta > 0, \smallskip\\
                       0, & \text{ if } \beta = 0.
\end{array}
\right.
\end{equation}
\begin{rmrk}[Vanishing advection consistency]
The norm definition~\eqref{eq:vh_norm} is consistent in the limit case $\beta \rightarrow 0^+$ since, for all $v_h \in V_h$, it holds:
\begin{equation*}
\beta_l \sum_{T \in \Par_h} h_T \int_T (\Vb \cdot \nabla v_h )^2 \leq \beta_l \, \beta^2 \sum_{T \in \Par_h} h_T \int_T |\nabla v_h|^2 \rightarrow 0^+, \, \text{ when } \, \beta \rightarrow 0^+.
\end{equation*}
\end{rmrk}
Before discussing  the proof of the inf-sup stability of the formulations, we recall the sufficient conditions to ensure the coercivity of bilinear forms~\cite{  riviere2008discontinuous}:
\begin{lmm}[Coercivity for pure diffusion]\label{lmm:dar_coercivity}
%
%Consider the bilinear form $b_\epsilon(\cdot,\cdot)$ of~\eqref{eq:sip_bh}.
%, and assume that one of the following cases holds true:
%%
%\begin{enumerate}
%\item[i)] NIP: $\epsilon = 1$, $p\geq1$ and $\eta_{-1} > 0$ for all $e \in \Sk_h$.
%\item[ii)] SIP: $\epsilon = -1$, $p\geq1$ and $\eta_{1}$ is bounded below by a sufficiently large constant for all $e \in \Sk_h$.
%%\item[iii)] IIP: $\epsilon = 0$, $p\geq1$ and $\eta_{0}$ is bounded below by a sufficiently large constant for all $e \in \Sk_h$.
%%\item[iii)] OBB: $\epsilon = 1$, $p\geq2$ and $\eta_{1} = 0$ for all $e \in \Sk_h$.
%\end{enumerate}
%
\tred{For $\epsilon \in \{-1,1\}$ and $p_t\geq 1$, the bilinear form $b_\epsilon(\cdot, \cdot)$ in~\eqref{eq:sip_bh} is coercive in $V_h$ with respect to the norm $\|\cdot\|_{\epsilon}^2 := (\cdot\, ,\cdot)_{\epsilon}$ (cf.~\eqref{eq:coercivity_cont}), with stability constant equal to $1$ if \, $\epsilon = 1$, and $1/2$ if \, $\epsilon = -1$.}
\end{lmm}
Combining this result with a slight variation of the arguments in~\cite[Chap.~4.6.3]{di2011mathematical}, we obtain these results:
\begin{lmm}[Inf-sup stability]\label{lmm:dar_inf_sup}
In the above framework, consider $\epsilon \in \{-1,1\}$ and $p_t\geq 1$. Then,
\tred{
\begin{equation}\sup_{0\neq v_h \in V_h} \dfrac{b_h(w_h,v_h)}{\| v_h \|_{V_h}}  \geq C_{\textrm{sta}} \, \| w_h \|_{V_h}, \quad \forall \, w_h \in V_h. 
\end{equation}}
\end{lmm}
\begin{lmm}[Boundedness]\label{lmm:dar_boundedness}
Under the same hypotheses of Lemma~\ref{lmm:dar_inf_sup}. \tred{Defining $V_{h,\#} =: U_\# \cup V_h$, it holds:
\begin{equation}
 b_h(w,v_h) \leq C_{\textrm{bnd}} \, \|w\|_{V_h,\#} \| v_h \|_{V_h}, \quad \forall \, (w, v_h) \in V_{h,\#} \times V_{h},
 \end{equation} 
where the norm $\|w\|_{V_h,\#}$ is explicitly defined as:}
\begin{equation}\label{eq:norm_dar_ext}
\|v\|_{V_h,\#}^2:= \|v\|_{V_h}^2 +  \beta \sum_{T \in \Par_h} \int_T v^2 + \sum_{T\in \Par_h} h_T \int_{\partial T} \kappa (\nabla v \cdot \Vn_T)^2. 
\end{equation}
\end{lmm}
\tred{The following result is a consequence of Lemmas~\ref{lmm:dar_consistency},~\ref{lmm:dar_inf_sup} and~\ref{lmm:dar_boundedness}, and the polynomial approximation property of the discrete space $V_h$ (cf.~\cite{di2008discontinuous}):
\begin{prpstn}[A priori error estimate]\label{prop:dar_apriori}
 The solution $\udg \in V_h$ of problem~\eqref{eq:dg_load} is unique and the following a priori error estimate holds:
\begin{equation}
\left\|u-\udg \right\|_{V_h} \leq \big(1 + \dfrac{C_{\textrm{bnd}}}{C_{\textrm{sta}}}\big) \, \inf_{v_h \in V_h} \| u - v_h\|_{V_h,\#}. 
\end{equation}
For $u \in H^{1+p_t}(\Omega)$, there exists a mesh independent constant $0 < C$, such that 
\begin{equation}\label{eq:dar_slope}
\left\|u-\udg \right\|_{V_h} \leq C \, \|u\|_{H^{1+p_t}(\Omega)} \big( \kappa^{1/2} + \beta^{1/2}h^{1/2} + (\gamma + \beta L^{-1})^{1/2} h \big)  \, h^{p_t}, \text{ when } h \rightarrow 0^+.
\end{equation}
\end{prpstn}
From~\eqref{eq:dar_slope} we infere that  the optimal convergence rate is $h^{p_t + 1}$,  $h^{p_t + 1/2}$, and $h^{p_t }$ for the reaction-dominated,  advection-dominated, and diffusion-dominated regime,  respectively.} 
%
%\begin{rmrk}[Norm modification]
%%
%The discrete norm~\eqref{eq:vh_norm} is a slight modification of the norm in~\cite{ ayuso2009discontinuous} such that it does depend on the stabilization parameter $\eta_e$. We modify the norm because it plays a fundamental role in the \tred{quality of the residual representative as error estimator in} the method~\cite{ rojas2019adaptive}. In particular, the discrete norm~\eqref{eq:vh_norm} lets us tune $\eta_\epsilon>0$ to improve the discrete solution while keeping the stability constant unchanged. This is not the case for the $\eta_\epsilon$-independent norm (see~\cite[Lemma 4.12]{ di2008discontinuous}).
%%
%\end{rmrk}
%\subsubsection{Localization of the inner product}\label{sec:localization_dar}
%
%
%
\section{Primal saddle-point problem (AS-FEM)}\label{sec:mixed_primal}
\tpur{We now introduce a conforming automatically adaptive stabilized finite element method (AS-FEM) (see~\cite{ rojas2019adaptive}).} \tred{Let $U, V$ be real Banach (typically Hilbert) spaces --with $V$ being reflexive--, $U_h$ be a conforming subspace of $U$ (e.g., space of globally continuous piecewise polynomial functions of degree $p \geq 1$),  and $V_h$ be a discrete space containing $U_h$ but not necessarily being conforming to either $U$ or $V$ (e.g., space of discontinuous piecewise polynomial functions of degree $p_t = p + \Delta_p$,  with $\Delta_p\geq0$). }\tpur{ We seek a discrete approximation $u_h \in U_h$ of the solution $u \in U$ of a well-posed variational formulation of the form~\eqref{eq:cont_dar}. }
\tpur{The AS-FEM method builds on a stable dG formulation of problem~\eqref{eq:cont_dar} set in the discrete space $V_h$; we obtain $u_h \in U_h$ by minimizing a residual in the dual space of $V_h$. In the following, we describe the main aspects of this method.}
\subsection{Discontinuous Galerkin assumptions}
\tred{We assume that the discrete formulation of~\eqref{eq:cont_dar} satisfies standard properties of a large class of dG formulations; namely,  we use consistent dG formulations and norms with optimal or quasi-optimal convergence properties (cf.~\cite{ di2011mathematical}),  which produce a well-posed discrete formulations for~\eqref{eq:cont_dar} of the form~\eqref{eq:dg_load},  and satisfy the following four assumptions (cf.~Lemmas~\ref{lmm:dar_consistency}-\ref{lmm:dar_boundedness} in Section~\ref{sec:model_dg}).}
\begin{ssmptn}[Strong consistency with regularity] \label{as:reg-consi_load} The exact solution $u$ of problem~\eqref{eq:cont_dar} belongs to a subspace $U_{\#} \subset U$ such that the discrete bilinear form $b_h$ supports evaluations in the extended space $V_{h,\#} \times V_{h}$, with $V_{h,\#} : = U_\# + V_h$, and the following holds true:
  \begin{equation}\label{eq:consistency_load} 
    b_h(u,v_h) = l_h (v_h), \quad \forall \, v_h \in V_h.
  \end{equation}
\end{ssmptn}
\begin{ssmptn}[Inf-sup stability]\label{as:inf-sup} The space $V_h$ can be equipped with a norm \tred{$\|\cdot\|_{V_h}=\sqrt{(\cdot,\cdot)_{V_h}}$}, and there exists a mesh independent constant $C_{\textrm{sta}}>0$, such that:
  \begin{equation}\label{eq:infsup_h}
    \sup_{0\neq v_h \in V_h} \dfrac{b_h(w_h,v_h)}{\| v_h \|_{V_h}}  \geq C_{\textrm{sta}} \, \| w_h \|_{V_h}, \quad \forall \, w_h \in V_h. 
  \end{equation}
\end{ssmptn}
\begin{ssmptn}[Boundedness]\label{as:bound_load} \tred{The norm $\|\cdot\|_{V_h}$ of Assumption~\ref{as:inf-sup}} can extend to the space $V_{h,\#}$ defined in Assumption~\ref{as:reg-consi_load}. Moreover, there is a second norm $\|\cdot\|_{V_h,\#}$ on $V_{h,\#}$ satisfying the following two properties:
 \begin{enumerate}[(i)]
 \item $\|v\|_{V_h} \leq \|v\|_{V_h,\#}$, for all $v \in V_{h,\#}$,
 \item there exists a mesh independent constant $C_{\textrm{bnd}}<\infty$  such that:
    \begin{equation}\label{eq:continuity_load}
      b_h(w,v_h) \leq C_{\textrm{bnd}} \, \|w\|_{V_h,\#} \| v_h \|_{V_h}, \quad \forall \, (w, v_h) \in V_{h,\#} \times V_{h}.
    \end{equation}
  \end{enumerate}
\end{ssmptn}
These assumptions are standard for consistent dG formulations. Assumption~\ref{as:inf-sup} is sufficient to guarantee well-posedness for the discrete problem~\eqref{eq:dg_load}, while assumptions~\ref{as:reg-consi_load} and~\ref{as:bound_load} guarantee the following a priori error estimate (see~\cite{ di2011mathematical}):
\begin{equation}\label{eq:apriori_dg_load}
\inf_{v_h \in V_h}\left\|u-v_h \right\|_{V_{h}} \leq  \left\|u-\udg \right\|_{V_h} \leq \Big(1 + \dfrac{C_{\textrm{bnd}}}{C_{\textrm{sta}}}\Big) \inf_{v_h \in V_h}\left\|u-v_h \right\|_{V_{h,\#}}\,\,.
\end{equation}
\tred{Additionally,  we require that the error estimate in~\eqref{eq:apriori_dg_load} is at least quasi-optimal,  i.e. (cf.~\cite{ di2011mathematical}): 
\begin{ssmptn}[Optimality and quasi-optimality]\label{as:optimality}
For a sufficiently smooth analytical solution $u \in U$, the quantities $\displaystyle \inf_{v_h \in V_h}\left\|u-v_h \right\|_{V_{h}}$ and $\displaystyle  \inf_{v_h \in V_h}\left\|u-v_h \right\|_{V_{h,\#}}$ decay with the same convergence rate as $h \rightarrow 0^+$.
\end{ssmptn}
}\tpur{If assumption~\ref{as:optimality} is satisfied, and the norms $\| \cdot \|_{V_{h,\#}}$ and $\|\cdot \|_{V_{h}}$ are equal, the error estimate~\eqref{eq:apriori_dg_load} is optimal; otherwise, the error estimate is quasi-optimal. This assumption is satisfied for both model formulations in Section~\ref{sec:model_dg}, and the estimate is optimal when the velocity field ${\bf b}$ vanishes in $\Omega$ (cf.~\cite{ di2011mathematical}).}

\subsection{Residual minimization problem}
\tred{Rather than solving the discrete problem~\eqref{eq:dg_load}, we solve:}
\begin{equation}\label{eq:armin_direct}
  \left\{
    \begin{array}{l}
      \text{Find } u_h \in U_h \subset V_h,  \text{ such that:} \\
      \displaystyle u_h = \argmin_{w_h \in U_h} \dfrac{1}{2}\| l_h(\cdot) - b_h(w_h, \cdot) \|_{V_h^\ast} ^2 = \argmin_{w_h \in U_h} \dfrac{1}{2}\| R^{-1}_{V_h} B_h(\udg-w_h) \|_{V_h} ^2.
    \end{array}
  \right.
\end{equation}
In the above, %\tred{$\| \cdot \|_{V_h} = \sqrt{(\cdot,\cdot)_{V_h}}$ denotes the norm in which the bilinear form $b_h(\cdot,\cdot)$ in~\eqref{eq:dg_load} is inf-sup stable}. 
$B_h$ corresponds to the operator:
\begin{equation}\label{eq:B_h} 
 \left\{ \begin{array}{rcl}
    B_h&:& V_h \mapsto V_h^\ast \\
        && w_h \mapsto b_h(w_h, \cdot),
  \end{array}
  \right.
\end{equation}
$R^{-1}_{V_h}$ denotes the inverse of the Riesz map:
\begin{equation}\label{eq:riesz}
  \left\{\begin{array}{rcl}
    R_{V_h} &:& V_h \mapsto V_h^\ast \smallskip\\
            && \left< R_{V_h} y_h, v_h\right>_{V_h^\ast \times V_h} := (y_h,v_h)_{V_h},
  \end{array} \right.
\end{equation}
and the dual norm $\|\cdot\|_{V_h^\ast}$ is defined as:
\begin{equation}\label{eq:adjoint_norm}
  \| \phi \|_{V_h^\ast} := \sup_{0 \neq v_h \in V_h} \dfrac{\left< \phi \, , \, v_h \right>_{V_h^\ast \times V_h}}{\| v_h\|_{V_h}}, \quad \forall \, \phi \in V_h^\ast,
\end{equation}
\tred{where $\| \cdot\|_{V_h}$ is the norm in which the dG bilinear form $b_h(\cdot,\cdot)$ is inf-sup stable. The second equality in~\eqref{eq:armin_direct} holds since the Riesz operator is an isometric isomorphism %satisfying:
%
%\begin{eqnarray}\label{eq:riesz_identity}
%  \| \phi \|_{V_h^\ast} = \| R_{V_h}^{-1} \phi \|_{V_h}, \quad \forall \, \phi \in V_h^\ast,
%\end{eqnarray}
%
and the dG solution of the discrete problem~\eqref{eq:dg_load} satisfies $l_h(\cdot) = B_h \udg$ in $V_h^\ast$. Problem~\eqref{eq:armin_direct} is equivalent to the following saddle-point problem (see~\cite{ CohDahWelM2AN2012}):}
\begin{equation}\label{eq:mixed_direct}
  \left\{ 
    \begin{array}{l}
      \text{Find } (\varepsilon_h, u_h) \in V_h \times U_h, \text{ such that:}\\ 
      \begin{array}{lll}
        (\varepsilon_h \, , \, v_h)_{V_h} + b_h(u_h \, , \, v_h) \hspace{-0.2cm} & = \  l_h(v_h),   &\quad \forall\, v_h \in V_h, \\
        b_h(w_h \, , \, \varepsilon_h) & = \  0,  &\quad \forall\, w_h \in U_h,
      \end{array}
    \end{array}
  \right.
\end{equation} 
where $\varepsilon_h \in V_h$ denotes a residual representative in terms of $\udg \in V_h$ and $u_h \in U_h$. Indeed, using the Riesz representation, the first condition in~\eqref{eq:mixed_direct} is equivalent to:
\begin{equation}\label{eq:eh}
  \varepsilon_h = R_{V_h}^{-1}(l_h(\cdot) - b_h(u_h, \cdot)) = R_{V_h}^{-1}B_h(\udg -  u_h).
\end{equation}
\tpur{Solving the saddle-point problem~\eqref{eq:mixed_direct} has several benefits. First, the problem inherits the discrete stability of the dG formulation. Second, the approximation belongs to a conforming subspace of $U$; thus, no further postprocessing is necessary as in the case of dG and dPG formulations to obtain a conforming representation. Third, the residual representative $\varepsilon_h$ is an efficient and reliable (robust) error estimate. We now formalize these claims.}

\subsubsection{A priori error estimates and approximation capacity }
\tred{The main result in~\cite{ rojas2019adaptive} is the following theorem for the saddle point problem~\eqref{eq:mixed_direct}:
\begin{thrm}[A priori bounds and error estimates]\label{th:FEMwdG} If assumptions~\ref{as:reg-consi_load}-\ref{as:bound_load} are satisfied, then
the mixed problem~\eqref{eq:mixed_direct} has a unique solution $(\varepsilon_h,u_h)\in V_h\times U_h$. Moreover, such solution satisfies the following a priori bounds:
\begin{equation}\label{eq:bounds}
  \|\varepsilon_h\|\leq \|l_h\|_{V_h^\ast} \qquad \hbox{ and }\qquad \|u_h\|_{V_h} \leq \dfrac{1}{C_{\textrm{sta}}}\|l_h\|_{V_h^\ast}\,\,,
\end{equation}
and the following a priori error estimate holds true:
\begin{equation}\label{eq:apriori_load}
  \|u-u_h\|_{V_h} \leq \Big(1 + \dfrac{C_{\textrm{bnd}}}{C_{\textrm{sta}}}\Big) \inf_{w_h \in U_h}\|u-w_h\|_{V_{h,\#}}\,\,,
\end{equation}
where $u$ denotes the exact solution to the continuous primal problem~\eqref{eq:cont_dar}.
\end{thrm}
}
\tred{The a priori bound~\eqref{eq:apriori_load} states that the AS-FEM delivers a solution with the same quality of the dG formulation (i.e., with the same convergence rate) if the chosen subspace $U_h\subset V_h$ has the same approximation capacity of $V_h$. That is, the following inequality should be satisfied:
\begin{equation}\label{eq:approx}
\inf_{w_h \in U_h}\|u-w_h\|_{V_{h,\#}} \leq C_{\textrm{app}} \, \inf_{v_h \in V_h}\left\|u-v_h \right\|_{V_{h,\#}}.
\end{equation}}
\tpur{The above inequality is typically satisfied when $U_h$ is a subspace of $V_h$ consisting of conforming functions in $U$. In the model problem, this corresponds to $p=p_t$. Indeed, this order correspondance is the basis for the solution reconstruction via postprocessing in dG methods (see~\cite{ achdou2003priori, alaoui2004residual, karakashian2003posteriori}).  Even though this is not true in general when $p < p_t$,  an augmented test space could improve the error estimates, while keeping the optimal convergence rates in terms of the trial space,  as we show in the numerical section.}
\subsubsection{Saturation assumptions and a posteriori error estimates}\label{sec:aposteriori_primal}
\tred{The residual representative $\varepsilon_h \in V_h$ is an efficient error estimate of $u-u_h$ in the energy norm; that is (up to a constant) $\|\varepsilon_h\|_{V_h}$ is a lower bound.  However,  for $\|\varepsilon_h\|_{V_h}$ to be (up to a constant) also a reliable error estimate (upper bound),  one of the following assumptions must be satisfied:}
\begin{ssmptn}[Saturation]\label{as:saturation_load}
Let $u_h\in U_h$ be the second component of the pair $(\varepsilon_h,u_h)\in V_h\times U_h$ solving~\eqref{eq:mixed_direct} and let $\udg\in V_h$ be the unique solution to~\eqref{eq:dg_load}. There exists a real number $\delta_s \in [0,1)$, uniform with respect to the mesh size, such that $$\|u-\udg\|_{V_h} \le \delta_s \|u-u_h\|_{V_h}.$$
\end{ssmptn}
\begin{ssmptn}[Weak saturation]\label{as:saturation_load_weak}
Let $u_h\in U_h$ be the second component of the pair $(\varepsilon_h,u_h)\in V_h\times U_h$ solving~\eqref{eq:mixed_direct} and let $\udg\in V_h$ be the unique solution to~\eqref{eq:dg_load}. There exists a real number $\delta_w > 0$, uniform with respect to the mesh size, such that $$\|u-\udg\|_{V_h} \le \delta_w \|\udg-u_h\|_{V_h}.$$
\end{ssmptn}
\tred{These two assumptions are not standard in the dG theory as they involve the discrete solution $u_h$ of the saddle-point problem~\eqref{eq:mixed_direct}. Roughly speaking, Assumption~\ref{as:saturation_load} states that the discrete approximation $\udg$ is closer than $u_h$ to the analytical solution $u$ with respect to the norm in $V_h$. This is a meaningful assumption since $U_h \subset V_h$. However, even if the optimality Assumption~\ref{as:optimality} is satisfied, it does not necessarily holds in the pre-asymptotic regime (see~\cite{ rojas2019adaptive}).  The Assumption~\ref{as:saturation_load_weak} is a weaker than Assumption~\ref{as:saturation_load} since the latter implies Assumption~\ref{as:saturation_load_weak}, while the reciprocal is only granted when $\delta_w < 1/2$. We summarize this result in the following proposition.}
\begin{prpstn}[\tred{Robustness} of the residual representative]\label{prop:efficiency} Under the same hypotheses of Theorem~\ref{th:FEMwdG}, it follows:
  \begin{equation}\label{eq:efficiency}
    C_{\emph{sta}}\,  \|\udg-u_h\|_{V_h} \leq \|\varepsilon_h\|_{V_h} \leq C_{\emph{bnd}} \, \|u-u_h\|_{V_h,\#}.
  \end{equation}
  Additionally, when the solution satisfies either of the saturation Assumptions~\ref{as:saturation_load} or~\ref{as:saturation_load_weak}, then the following a~posteriori error estimate holds:
  \begin{equation} \label{eq:a_posteriori}
    \|u-u_h\|_{V_h} \le \dfrac{C_{\emph{sat}}}{C_{\emph{sta}}} \|\varepsilon_h\|_{V_h},
  \end{equation} 
  with $C_{\emph{sat}} = \dfrac{1}{1-\delta_s}$ if Assumption~\ref{as:saturation_load} is satisfied, and $C_{\emph{sat}}  = \delta_w$ if only the weaker Assumption~\ref{as:saturation_load_weak} is satisfied. \tpur{Finally, if Assumption~\ref{as:optimality} is satisfied, then $\|\udg-u_h\|_{V_h}$ and $\|\varepsilon\|_{V_h}$ have the same convergence rate when $h \rightarrow 0^+$.}
\end{prpstn}

\section{Adjoint formulations}\label{sec:adjoint_problems}

In this section, we introduce the adjoint (dual) problems. Section~\ref{sec:adjoint_dg} discusses the adjoint continuous formulation and its dG formulation. Section~\ref{sec:mixed_goal} introduces the adjoint problem to the saddle-point formulation~\eqref{eq:mixed_direct}.

\subsection{Continuous and dG adjoint formulations}\label{sec:adjoint_dg}

We want to accurately approximate a quantity of interest $q(u)$, where $q:U\rightarrow\R$ is a bounded linear form, and $u \in U$ is the analytical solution of the continuous, primal problem~\eqref{eq:cont_dar}. \tpur{The goal-oriented-adaptive (GoA)} strategy considers a second continuous problem, known as the continuous adjoint problem:
\begin{equation}\label{eq:cont_goal}
  \left\{
    \begin{array}{l}
      \text{Find } v^\ast \in V,  \text{ such that:} \\
      b(w, v^\ast) = q(w), \quad \forall \, w \in U.
    \end{array}
  \right.
\end{equation}
\tpur{Assuming that an extended linear form $q$ can be exactly evaluated for any function in $V_h$, } the dG adjoint formulation associated with ~\eqref{eq:cont_dar} reads:
\begin{equation}\label{eq:dg_goal}
  \left\{
    \begin{array}{l}
      \text{Find } \vdg \in V_h,  \text{ such that:} \\
      b_h(v_h, \vdg) = q(v_h), \quad \forall \, v_h \in V_h,
    \end{array}
  \right.
\end{equation} 
where $b(\cdot,\cdot)$ and $b_h(\cdot,\cdot)$ denote the same bilinear forms \tpur{of} the continuous~\eqref{eq:cont_dar} and discrete~\eqref{eq:dg_load} primal problems, respectively. \tpur{When~\eqref{eq:dg_goal} is consistent, that is, adjoint consistency (cf. Assumption~\ref{as:reg-consi_load}), } the following identity holds:
\begin{equation}\label{eq:dg_goal_identity}
  q(u-\udg) = b_h(u-\udg \, , \, v^\ast) = b_h(u-\udg, v^\ast-\vdg),
\end{equation} 
where the last equality follows from the Galerkin orthogonality.  Standard discontinous Galerkin GoA algorithms employ a posteriori error estimates of the quantity of interest~\eqref{eq:dg_goal_identity}. These estimates provide an upper bound for~\eqref{eq:dg_goal_identity} in terms of locally computable variables that guide the adaptivity and controls the error in the quantity $q(u-\udg)$. The main limitation with the standard dG a~posteriori procedure is that the adjoint consistency is not always satisfied for all consistent dG formulations (cf.~\cite{ hartmann2007adjoint, arnold2002unified}), \tpur{which complexifies the error estimation when the selected formulation is not adjoint consistent. An advantage of our GoA strategy is that the adjoint consistency of the reference dG formulation is unnecessary since we seek to reduce the error $q(\udg-u_h)$ as Section~\ref{sec:goal_estimator} explains, which allows us to explore GoA formulations with fewer assumptions on the reference dG formulation.}
\begin{rmrk}[Well-posedness of the dG problem] 
The well-posedness for the dG adjoint problem~\eqref{eq:dg_goal} is guaranteed since the discrete inf-sup condition~\eqref{eq:infsup_h} is equivalent to (see~\cite{ boffi2013mixed}):
\begin{equation}\label{eq:infsup_h_goal}
  \sup_{0\neq w_h \in V_h} \dfrac{b_h(w_h,v_h)}{\| w_h \|_{V_h}}  \geq C_{\textrm{sta}} \, \| v_h \|_{V_h}, \quad \forall \, v_h \in V_h.
\end{equation}
\end{rmrk}
\subsection{Adjoint saddle-point formulation}
\label{sec:mixed_goal}
%We recall the saddle-point formulation considered in the Step~\ref{en:step2} of the proposed GoA algorithm in Section~\ref{sec:GoA_strategy}:
\tred{Following~\cite{ keith2019goal}, we consider the following saddle-point problem as the adjoint formulation of problem~\eqref{eq:mixed_direct}:}
\begin{equation}\label{eq:mix_form_goal}
  \left\{ 
    \begin{array}{l}
      \text{Find } (v_h^\ast, w_h^\ast) \in V_h \times U_h, \text{ such that}:\\ 
      \begin{array}{lll}
        (v_h^\ast \, , \, v_h)_{V_h} + b_h(w_h^\ast \, , \, v_h) \hspace{-0.2cm} & = \ 0,   &\quad \forall\, v_h \in V_h, \\
        b_h(w_h \, , \, v_h^\ast) & = \ q(w_h),  &\quad \forall\, w_h \in U_h,
      \end{array}
    \end{array}
  \right.
\end{equation}
\tpur{where $v_h^\ast \in V_h$ approximates the adjoint dG solution $\vdg$ of~\eqref{eq:dg_goal}, which is the adjoint counterpart of the discrete solution $u_h \in U_h$ of~\eqref{eq:mixed_direct}, while $w_h^\ast \in U_h$ is an auxiliary variable that constrains the dimension of the dG solution.} The saddle-point formulation for the primal problem~\eqref{eq:mixed_direct} is equivalent to a residual minimization. Similarly, the adjoint formulation minimizes the Riesz representation subject to constraints. If the direct saddle-point problem~\eqref{eq:mixed_direct} is well-posed, then the mixed-adjoint problem~\eqref{eq:mix_form_goal} is also well-posed, as both discrete problems share the same left-hand side square matrix. \tpur{Problem~\eqref{eq:mix_form_goal} has, at least, two equivalent representations (see discussion in Remarks~\ref{rk:Riesz.min} and~\ref{rk:PG}).}
\begin{rmrk}[Constrained Riesz representative minimization] \label{rk:Riesz.min}
We seek the optimal Riesz representative $v_h^\ast \in V_h$ subject to the constraint of the second equation in~\eqref{eq:mix_form_goal}. To analyze this problem as an unconstrained optimization problem, we introduce a Lagrangian and the respective set of multipliers $w_h \in U_h$. Let $(v_h^\ast, w_h^\ast) \in V_h \times U_h$ be a stationary point of
\begin{align*}
  \mathcal{L}(v_h, w_h) = \dfrac{1}{2} (v_h\, , \,v_h)_{V_h}+ b_h(w_h \, , \, v_h)-q(w_h).
\end{align*}
The stationarity conditions of $\mathcal{L}(v_h, w_h)$ (i.e., $\tfrac{\partial \mathcal{L}(v_h, w_h)}{\partial v_h}=\tfrac{\partial \mathcal{L}(v_h, w_h)}{\partial w_h}=0$) correspond to the mixed adjoint problem~\eqref{eq:mix_form_goal}. Denoting the stationary value of the Lagrangian $\mathcal{L^\ast}(q)=\mathcal{L}(v_h^\ast, w_h^\ast)=\tfrac{1}{2} (v_h^\ast\, , \,v_h^\ast)_{V_h}$ as an implicit function of quantity of interest, then $-w_h^\ast$ determines the (marginal) effect of each constraint on the attainable value of  the Riesz representative.
\end{rmrk}
\begin{rmrk}[Petrov-Galerkin method with optimal trial functions] \label{rk:PG}
The first equation in~\eqref{eq:mix_form_goal} introduces a Lagrange multiplier $w^\ast_h \in U_h$ (see Remark~\ref{rk:Riesz.min}). This allows us to equivalently express~\eqref{eq:mix_form_goal} in terms of the adjoint discrete solution $v_h^\ast \in V_h$ via the following Petrov-Galerkin problem:
\begin{equation}\label{eq:PG_goal}
  \left\{
    \begin{array}{l}
      \text{Find } v_h^\ast \in \Theta_h,  \text{ such that:} \\
      b_h(w_h, v_h^\ast) = q(w_h), \quad \forall w_h \in U_h,
    \end{array}
  \right.
\end{equation}
where the discrete space $\Theta_h \subset V_h$ is defined as: 
\begin{equation}\label{eq:Theta_h}
  \Theta_h := \left\{ \varphi_h \in V_h \, \text{ s.t. } \, \exists \, w_h \in U_h \, : \, (\varphi_h, v_h)_{V_h} + b_h(w_h, v_h) = 0, \forall v_h \in V_h \right\}. 
\end{equation}
This subspace is the Riesz representation of the action of $b_h$ on each basis of $U_h$. Problem~\eqref{eq:mix_form_goal} inherits the well-posedness  from the direct saddle-point problem~\eqref{eq:mixed_direct}, which is a direct consequence of the inf-sup Assumption~\ref{as:inf-sup}. Moreover, the existence of a unique representative $w_h \in U_h$ in the definition~\eqref{eq:Theta_h} is a consequence of the bijectivity of the Riesz isomorphism and the injectivity of the operator $B_h$ (see Equation~\eqref{eq:B_h}), implying that $w_h^\ast \in U_h$ is the unique representative of $v_h^\ast$.
\end{rmrk}

\subsection{Residual based error representative and error estimates for the adjoint problem}
\label{sec:error_goal}

Unlike the saddle-point formulation of the primal problem~\eqref{eq:mixed_direct}, the adjoint saddle-point formulation does not deliver an on-the-fly error estimate, as $w_h^\ast \in U_h$ is a Lagrange multiplier. Nevertheless, we can estimate the error by solving the following discrete problem:
\begin{equation}\label{eq:error_est_goal}
\left\{\begin{array}{l}
\text{Find } \varepsilon_h^\ast \in V_h,  \text{ such that:} \\
(\varepsilon_h^\ast, v_h)_{V_h} = q(v_h)-b_h(v_h,v_h^\ast), \quad \forall \, v_h \in V_h.
\end{array}
\right.
\end{equation}
which estimates the distance of the residual in $V_h$ to its orthogonal projection onto the optimal trial function space.  This problem has a unique solution and is well posed as we discuss below.
\begin{rmrk}
  If $U_h$ is strictly contained in $V_h$ then, in general, $\varepsilon_h^\ast \neq 0$ in $V_h^\ast$. Indeed, the orthogonality constraint $(\varepsilon_h^\ast, v_h)_{V_h} = 0$ holds only if $v_h \in U_h$. Therefore, equation~\eqref{eq:error_est_goal}, together with the second identity in~\eqref{eq:mix_form_goal}, imply that $\varepsilon_h^\ast = R_{V_h}^{-1}B_h^\ast(\vdg-v_h^\ast)$, with $B_h^\ast$ defined as:
\begin{equation}\label{eq:B_h_star} 
\left\{\begin{array}{ll}
B_h^\ast:& V_h \rightarrow V_h^\ast \\
& v_h \rightarrow b_h(\cdot, v_h).
\end{array} \right.
\end{equation}
\end{rmrk}
\tpur{Unlike the residual representative for the primal problem $\varepsilon_h$ (see~\eqref{eq:eh}), the adjoint error estimate $\varepsilon_h^\ast$ (see~\eqref{eq:B_h_star}) may be an inefficient as the adjoint solution $v_h^\ast$  belongs to a subspace $\Theta_h \subset V_h$ (see Remark~\ref{rk:PG}), which does not necessarily have the same approximation capacity than $V_h$ (cf.~\eqref{eq:approx})}. Nevertheless, the following results hold:
\tred{\begin{prpstn}[Upper bound for the discrete adjoint error]\label{prop:min_dual}
Under the same hypothesis of Proposition~\ref{th:err_rep_goal}. For all $c \in \R$, it holds:
\begin{equation}
\left\| \vdg-v_h^\ast \right\|_{V_h} \leq \left\|\vdg- c \, v_h^\ast \right\|_{V_h} \leq \dfrac{1}{C_{\textrm{sta}}} \, \left\|q \right\|_{V_h^\ast}
\end{equation}
\end{prpstn}
\begin{proof}
The quantity $\left\|\vdg- c \, v_h^\ast \right\|_{V_h}$ is minimized for $c_{\textrm{min}} = \dfrac{\big(v_h^\ast\, , \, \vdg\big)_{V_h}}{\big(v_h^\ast\, , \, v_h^\ast \big)_{V_h}}$. Using the first equation in~\eqref{eq:mix_form_goal}, and the fact that $b_h(w_h^\ast\, , v_h^\ast) = b_h(w_h^\ast\, , \vdg)$, we conclude that $c_{\textrm{min}} = 1$. The second inequality is a direct consequence of the adjoint inf-sup condition~\eqref{eq:infsup_h_goal} and the dual norm definition~\eqref{eq:adjoint_norm}. 
\end{proof}}
\begin{prpstn}[\tred{Robustness} of the adjoint residual representative] \label{th:err_rep_goal}
Let $v_h^\ast \in V_h$ be the first component of the solution pair $(v_h^\ast,w_h^\ast)\in V_h\times U_h$ of~\eqref{eq:mix_form_goal}. Let $\vdg \in V_h$ be the unique solution to~\eqref{eq:dg_goal}. Then:
\begin{equation} \label{eq:bnd_uh-thetah_goal}
C_{\textrm{sta}} \, \left\|\vdg-v_h^\ast \right\|_{V_h} \le \left\|\varepsilon_h^\ast \right\|_{V_h} \tred{\le C_{\textrm{bnd}} \, \left\|\vdg-v_h^\ast \right\|_{V_h, \#} }.
\end{equation} 
\end{prpstn}
\begin{proof}
We have
\begin{alignat*}{2}
\tag{by~\eqref{eq:error_est_goal} and~\eqref{eq:dg_goal}} \left\|\varepsilon_h^\ast \right\|_{V_h} = & \left\|R^{-1}_{V_h}B_h^\ast \big(\vdg - v_h^\ast\big) \right\|_{V_h^\ast}\\
 = & \sup_{0 \neq v_h \in V_h} \dfrac{b_h(v_h,\vdg - v_h^\ast)}{\|v_h\|_{V_h}}\\
 \tag{by~\eqref{eq:infsup_h_goal}} \geq &  C_{\textrm{sta}} \left\|\vdg - v_h^\ast \right\|_{V_h},
\end{alignat*}
proving \tred{the first inequality in}~\eqref{eq:bnd_uh-thetah_goal}; \tred{the second inequality follows from~\eqref{eq:error_est_goal} and the boundedness Assumption~\ref{as:bound_load}.}
\end{proof}

\section{Goal oriented a posteriori error estimation}
\label{sec:goal_estimator}
\tpur{Here, we describe the localized error estimates that guide our GoA. We estimate them in terms of dG solutions associated with~\eqref{eq:dg_load} and~\eqref{eq:dg_goal}.  Therefore, they rely on the discrete stability of these problems.  First, we introduce a proposition that encapsulates useful relations in terms of the involved discrete terms; we then prove the main Theorem of this section.}
\begin{prpstn}[Discrete orthogonality]\label{prop:disc_ortho} Let $\udg$, $(\varepsilon_h, u_h)$, $\vdg$, and $(v_h^\ast, w_h)$ be the unique solution of problems~\eqref{eq:dg_load}, ~\eqref{eq:mixed_direct}, ~\eqref{eq:dg_goal}, and~\eqref{eq:mix_form_goal}, respectively. Then, the following equalities hold:
\begin{enumerate}
\item[a)] $b_h\big(u_h \, , \, \vdg-v_h^\ast\big) = 0$,
\item[b)] $\big(v_h^\ast\, , \, \varepsilon_h\big)_{V_h}=0$,
\item[c)] $q(u_h) - l_h(v_h^\ast)=0$,
\item[d)] $b_h\big(\udg-u_h \, , \, v_h^\ast\big)=0$.
\end{enumerate}
\end{prpstn}
\begin{proof}
Equality $a)$ is a direct consequence of the identities~\eqref{eq:dg_goal} and the second equation in~\eqref{eq:mix_form_goal}. Since $v_h^\ast \in V_h$, using the first equation in~\eqref{eq:mix_form_goal}, and the second equation in~\eqref{eq:mixed_direct}, respectively, we obtain:
\begin{equation*}
\big(v_h^\ast\, , \, \varepsilon_h\big)_{V_h} = - b_h(w_h\, , \, \varepsilon_h) = 0,
\end{equation*}
proving $b)$. To prove $c)$, we consider the second equation in~\eqref{eq:mix_form_goal}, the first equation in~\eqref{eq:mixed_direct} and the result in $b)$. Finally, $d)$ is a consequence of the first equation in~\eqref{eq:mixed_direct}, the identity~\eqref{eq:dg_load}, and the symmetry of the discrete inner product in $V_h$.
\end{proof}
\begin{thrm}[Error in the quantity of interest] \label{prop:goal_error} Under the same hypotheses of Proposition~\ref{prop:disc_ortho}, the following identity holds:
\begin{align}\label{eq:qoi_dg}
 q\big(\udg - u_h\big) = b_h\big(\udg-u_h \, , \, \vdg - v_h^\ast\big) = \big(\varepsilon_h\, , \, \vdg - v_h^\ast\big)_{V_h} = l_h\big(\vdg-v_h^\ast\big).
\end{align}
If in addition the dG problem~\eqref{eq:dg_load} is adjoint consistent, then it holds:
\begin{align}
 q(u - u_h) = b_h\big(u-u_h \, , \, v^\ast - v_h^\ast\big) = l_h\big(v^\ast-v_h^\ast\big),
 \label{eq:qoi_u}
\end{align}
\tpur{where $u \in U$, $v^\ast \in V$ are the analytical solutions of problems~\eqref{eq:cont_dar} and ~\eqref{eq:cont_goal}, respectively.}
\end{thrm}
\begin{proof}
  Using~\eqref{eq:dg_goal} and the bilinearity of the form $b_h$, we obtain:
  \begin{equation*}
    q\big(\udg - u_h\big) = b_h\big(\udg-u_h \, , \, \vdg - v_h^\ast \big) + b_h\big(\udg-u_h \, , \, v_h^\ast\big).
  \end{equation*}
The first equality in~\eqref{eq:qoi_dg} is a consequence of $d)$ in Proposition~\ref{prop:disc_ortho}. \tred{The second equality is a consequence of identity~\eqref{eq:eh}.} The third equality is obtained by employing the identity $a)$ of Proposition~\ref{prop:disc_ortho}.  Finally,~\eqref{eq:qoi_u} is a direct consequence of adjoint consistency and the identity
%\begin{equation*}
$b_h\big(u_h \, , \, v^\ast-v_h^\ast \big) = b_h\big(u_h \, , \, \vdg- v_h^\ast \big)=0.$
%\end{equation*}
%
\end{proof}
\tpur{Proposition~\ref{prop:disc_ortho} implies that identity~\eqref{eq:qoi_dg} also holds if we scale $v_h^\ast$ by $c$ into $c v_h^\ast$, with $c \in \R$.  However, $v_h^\ast$ is the closest to the dG adjoint solution $\vdg$ in the $V_h$ norm (see Proposition~\ref{prop:min_dual}).}
 
An immediate consequence of the efficiency of the residual representative $\varepsilon_h$ (see Proposition~\ref{prop:efficiency}) is:
\begin{prpstn}[First a posteriori GoA error estimator]\label{prop:pre_GO} Let $\udg$ and $(\varepsilon_h, u_h)$ be the unique solution of problems~\eqref{eq:dg_load} and~\eqref{eq:mixed_direct}, respectively. It holds:
  \begin{equation}\label{eq:g_pre_estimator}
    \big|q(\udg-u_h)\big|  \leq \dfrac{1}{C_{\textrm{sta}}}\|q\|_{V_h^\ast} \|\varepsilon_h\|_{V_h}. 
  \end{equation}
  Moreover, if the saturation Assumption~\ref{as:saturation_load} is satisfied, then:
  \begin{equation}\label{eq:g_estimator1}
    \big|q(u-u_h)\big|  \leq \dfrac{1}{C_{\textrm{sta}}(1-\delta_s)}\|q\|_{V_{h, \#}^\ast} \|\varepsilon_h\|_{V_h},
  \end{equation}
  while if Assumption~\ref{as:saturation_load} is not satisfied, but the weaker condition of Assumption~\ref{as:saturation_load_weak} is satisfied, it holds:
  \begin{equation}\label{eq:g_estimator2}
\big|q(u-u_h)\big|  \leq \dfrac{1+\delta_w}{C_{\textrm{sta}}}\|q\|_{V_{h, \#}^\ast} \|\varepsilon_h\|_{V_h},
\end{equation}
where $u$ denotes the analytical solution of problem~\eqref{eq:cont_dar}, $V_{h,\#}$ is the space defined in Assumption~\ref{as:reg-consi_load}, and $\|\cdot\|_{V_{h,\#}^\ast}$ denotes the extension of the adjoint norm $\| \cdot\|_{V_h^\ast}$, defined in~\eqref{eq:adjoint_norm}, to the adjoint space of $V_{h,\#}$.
\end{prpstn}
Even if inequality~\eqref{eq:g_pre_estimator} implies that the error in the estimation of $q(\udg-u_h)$ is controlled by $\|\varepsilon_h\|_{V_h}$, this estimate ignores the contribution of the adjoint saddle-point problem. \tred{In the following, we do not assume that the dG formulation is adjoint consistent to derive the GoA error estimates as we express our a posteriori error bounds using~\eqref{eq:qoi_dg}. Instead, we consider the following additional assumptions motivated by Assumptions~\ref{as:saturation_load} and~\ref{as:saturation_load_weak}:}
\begin{ssmptn}[adjoint saturation condition]\label{as:saturation_GO}
Let $\udg$ and $(\varepsilon_h, u_h)$ be the unique solution of problems~\eqref{eq:dg_load} and~\eqref{eq:mixed_direct}, respectively. There exists a mesh independent constant $\mu_s \in [0,1)$, such that $\big|q(u-\udg) \big| \le \mu_s \big|q(u-u_h) \big|$.
\end{ssmptn}
\begin{ssmptn}[adjoint weak condition]\label{as:saturation_GO_weak}
Let $\udg$ and $(\varepsilon_h, u_h)$ be the unique solution of problems~\eqref{eq:dg_load} and~\eqref{eq:mixed_direct}, respectively. There exists a mesh independent constant $\mu_w > 0$, such that $\big|q(u-\udg) \big| \le \mu_w \big|q(\udg-u_h) \big|$.
\end{ssmptn}  
Then, the following result follows:
\begin{prpstn}[Second a posteriori GoA error estimator]\label{prop:GO}
Let $\udg$, $(\varepsilon_h, u_h)$, $\vdg$, $(v_h^\ast, w_h^\ast)$ be the unique solutions of problems~\eqref{eq:dg_load},~\eqref{eq:mixed_direct},~\eqref{eq:dg_goal}, and~\eqref{eq:mix_form_goal}, respectively, and let $\varepsilon_h^\ast$ be the a posteriori residual estimator of~\eqref{eq:error_est_goal}. The following holds true:
\begin{align}\label{eq:GO_pre_estimator}
\big| q(\udg-u_h) \big|  = \big|(\varepsilon_h \, , \, \vdg-v_h^\ast )_{V_h} \big| \leq  \dfrac{1}{C_{\textrm{sta}}}\|\varepsilon_h\|_{V_h} \|\varepsilon_h^\ast\|_{V_h}.
\end{align}
Moreover, if the saturation Assumption~\ref{as:saturation_GO} is satisfied, it holds:
\begin{align}\label{eq:GO_estimator}
 \big| q(u-u_h) \big|  \leq \dfrac{1}{1-\mu_s}  \big|(\varepsilon_h \, , \, \vdg-v_h^\ast )_{V_h} \big|  \leq  \dfrac{1}{C_{\textrm{sta}}(1-\mu_s)}\|\varepsilon_h\|_{V_h} \|\varepsilon_h^\ast \|_{V_h}.
\end{align}
while if only the weaker Assumption~\ref{as:saturation_GO_weak} is satisfied, then:
\begin{align}\label{eq:GO_estimator_weak}
 \big| q(u-u_h) \big|  \leq (1+\mu_w)  \big|(\varepsilon_h \, , \, \vdg-v_h^\ast )_{V_h} \big|  \leq  \dfrac{1+\mu_w}{C_{\textrm{sta}}}\|\varepsilon_h\|_{V_h} \|\varepsilon_h^\ast \|_{V_h},
\end{align}
 with $u$ being the analytical solution of problem~\eqref{eq:cont_dar}.  
\end{prpstn}
\begin{proof}
\eqref{eq:GO_pre_estimator} is a direct consequence of the identity~\eqref{eq:qoi_dg} and the bound for the adjoint problem~\eqref{eq:bnd_uh-thetah_goal}. Finally, Equations~\eqref{eq:GO_estimator} and~\eqref{eq:GO_estimator_weak} are consequence of the triangular inequality $\big|q(u-u_h)\big| \leq \big|q(u-\udg)\big| + \big|q(\udg-u_h)\big|$. 
\end{proof}
If the adjoint saturation Assumption~\ref{as:saturation_GO} is satisfied, the following result also holds:
\begin{crllr}\label{coro:aposteriori} Under the same hypotheses of Proposition~\ref{prop:GO}, if the adjoint saturation Assumption~\ref{as:saturation_GO} is satisfied, then:
\begin{align}\label{eq:go_efficiency}
\big|(\varepsilon_h \, , \, \vdg-v_h^\ast )_{V_h} \big| \leq \dfrac{1}{1+\mu_s}\big| q(u-u_h) \big|.
\end{align}
\end{crllr}
\begin{proof}
Direct consequence of the triangular inequality $\big|q(\udg-u_h)\big| \leq \big|q(u-\udg)\big| + \big|q(u-u_h)\big|$ and identity~\eqref{eq:GO_pre_estimator}.
\end{proof}

\subsection{Localizable upper-bound estimates}\label{sec:upper_estimates}

\tpur{Inequalities~\eqref{eq:GO_estimator_weak} and~\eqref{eq:go_efficiency} imply that the estimate $\big| (\varepsilon_h, \vdg-v_h^\ast )_{V_h}\big|$ is robust if Assumption~\ref{as:saturation_GO} is satisfied,  while it is efficient if only Assumption~\ref{as:saturation_GO_weak} is satisfied. In general, neither is true for the estimate $\|\varepsilon_h\|_{V_h} \|\varepsilon_h^\ast \|_{V_h}$, as $\|\varepsilon_h^\ast \|_{V_h}$ may be sub-optimal (see Section~\ref{sec:error_goal})}. \tred{ We first decompose the inner product into local contributions $(\cdot\, , \cdot)_T$ to derive local estimates of the error in the quantity of interest: }  %
\tpur{
  \begin{enumerate}
  \item $(\cdot\, , \cdot)_T$ is an inner product for $V_h|_{T}$,
  \item $\displaystyle (w_h , v_h)_{V_h} = \sum_{T\in \Par_h} (v_h \, , w_h)_T, \, \forall \, w_h,v_h \in V_h$,
  \end{enumerate}
  where $\Par_h$ denotes a conforming partition of the domain $\Omega$ (cf.~Section~\ref{sec:disc_sett}).} \tred{Defining $\|\cdot\|_{T}^2:= (\cdot,\cdot)_{T}$, we propose two strategies to obtain local estimations of the quantity $(\varepsilon_h\, , \, \vdg-v_h^\ast)_{V_h}$}., \tpur{which solve a third discrete problem. The first strategy (Estimator $E(\varepsilon_h\, , \, \vdg-v_h^\ast)$) solves problem~\eqref{eq:dg_goal} to obtain explicitly $\vdg$, and marks each element with the upper bound $\|\varepsilon_h\|_{T} \, \| \vdg-v_h^\ast\|_{T}$;}
\tpur{the second strategy (Estimator $E(\varepsilon_h\, , \, \varepsilon_h^\ast)$) solves the adjoint residual representative $\varepsilon_h^\ast$ from problem~\eqref{eq:error_est_goal}, and marks each element with the upper bound $\|\varepsilon_h\|_{T} \, \| \varepsilon_h^\ast\|_{T}$. Thus, we require the following assumption:}
\begin{ssmptn}[Local a posteriori adjoint residual estimation]\label{as:local_bound} There exists a mesh independent positive constant $C_{\textrm{sta}}^\ast$, such that:
\begin{equation}
\left\|\vdg-v_h \right\|_T \leq \dfrac{1}{C_{\textrm{sta}}^\ast} \left\|\varepsilon_h^\ast \right\|_T, \text{ for all } T \in \Par_h.
\end{equation}
\end{ssmptn}
\tred{This assumption ensures the existence of a localizable a posteriori error estimate in line with~\eqref{eq:bnd_uh-thetah_goal}.} \tpur{When considering direct solvers, the third problem's solution cost is the same for both estimators, as they require to invert a matrix of the same size. When considering iterative solvers, solving~\eqref{eq:dg_goal} requires to form and invert an equation system that is different from the saddle-point formulation~\eqref{eq:mixed_direct} if $U_h$ is a proper subspace of $V_h$. Obtaining $\varepsilon_h^\ast$ also solves an additional problem (see~\eqref{eq:error_est_goal}); however, it inverts the Gram matrix coming from the inner product $(\cdot\, , \, \cdot)_{V_h}$, which is always symmetric and positive definite. Moreover, when using an iterative solver for the saddle-point problem, a preconditioner for the Gram matrix is already available (cf.~\cite{ rojas2019adaptive}), which reduces the computational cost of $\varepsilon_h^\ast$ significantly when compared to computing $\varepsilon_h$. Effectively, computing $\varepsilon_h^\ast$ is equivalent to computing an extra outer loop of the iteration to compute the adjoint problem.}

\begin{rmrk}[Upper bound for the energy norm residual representative]\label{rem:equal_inner_rem}
$(\varepsilon_h\, , \, \varepsilon_h)_T$ is the square of the local estimation of the energy norm error estimator considered in~\cite{ rojas2019adaptive}. To compare our results with those of~\cite{ rojas2019adaptive}, we square the proportion of the total residual estimates employed in the D{\"o}rfler bulk-chasing criterion.
\end{rmrk}
\begin{rmrk}[Localization of the inner product for the model problem]
We localize the inner product~\eqref{eq:vh_norm} for the model problem as:
\begin{equation*}%\label{eq:vh_norm_split}
\displaystyle (w_h\, v_h)_{V_h} = (w_h\, , \, v_h)_{T} := (w_h\, , \, v_h)_{\textrm{loc}, T} + \dfrac{1}{2}\sum_{e \in \Sk_h^0 \cap \partial T} S_e(w_h \, , \, v_h), 
\end{equation*}
with
\begin{align*}
(w_h \, , \, v_h)_{\textrm{loc}, T} := & \displaystyle \, \kappa \int_T \nabla w_h \cdot \nabla v_h + \int_T ({\gamma} + \beta L^{-1} ) w_h \, v_h
 + \displaystyle \beta_l \, h_T \int_T (\Vb \cdot \nabla w_h )( \Vb \cdot \nabla v_h)
  \\
                                       & \displaystyle \, + \sum_{e \in \Sk_h^\partial \cap \partial T}\int_e \big(\kappa \, \eta_e  +  \dfrac{1}{2}  |\Vb \cdot \Vn_e| \big) w_h \,  v_h, \\ %
S_e(w_h \, , \, v_h) := &  \displaystyle \int_e \big(\kappa \, \eta_e  +  \dfrac{1}{2}  |\Vb \cdot \Vn_e| \big) \llbracket w_h \rrbracket \, \llbracket v_h \rrbracket.
\end{align*}
\end{rmrk}

\section{Goal-oriented-adaptivity algorithm}\label{sec:GoA_strategy}
\tred{We now summarize the GoA algorithm. We consider a dG space $V_h$, a subspace $U_h \subset V_h$ conforming in $U$ (e.g.,  standard FEM space of continuous piece-wise polynomial functions), and perform an iterative loop consisting of the following four steps:}
\begin{enumerate}
\item \label{en:step1}{\bf Primal problem:} We solve the primal saddle-point  problem:
\begin{equation*}%\label{eq:mixed_direct}
\left\{ 
\begin{array}{l}
\text{Find } (\varepsilon_h, u_h) \in V_h \times U_h, \text{ such that}: \\
\begin{array}{lll}
(\varepsilon_h \, , \, v_h)_{V_h} + b_h(u_h \, , \, v_h) \hspace{-0.2cm} & = \  l_h(v_h),   &\quad \forall\, v_h \in V_h, \\
b_h(w_h \, , \, \varepsilon_h) & = \  0,  &\quad \forall\, w_h \in U_h,
\end{array}
\end{array}
\right.
\end{equation*}  
\item \label{en:step2} {\bf Adjoint problem:} We solve the adjoint saddle-point problem:
\begin{equation*}
\left\{ 
\begin{array}{l}
\text{Find } (v_h^\ast, w_h^\ast) \in V_h \times U_h, \text{ such that}:\\ 
\begin{array}{lll}
(v_h^\ast \, , \, v_h)_{V_h} + b_h(w_h^\ast \, , \, v_h) \hspace{-0.2cm} & = \ 0,   &\quad \forall\, v_h \in V_h, \\
b_h(w_h \, , \, v_h^\ast) & = \ q(w_h),  &\quad \forall\, w_h \in U_h,
\end{array}
\end{array}
\right.
\end{equation*}
\item \label{en:step3} {\bf Residual estimation:} We consider one of the following alternatives: 
\begin{itemize}
\item[A)] {\bf Adjoint dG based estimator}. We solve the adjoint dG problem:
\begin{equation*}
\left\{\begin{array}{l}
\text{Find } \vdg \in V_h,  \text{ such that:} \\
b_h(v_h\, , \, \vdg) = q(v_h), \quad \forall \, v_h \in V_h,
\end{array}
\right.
\end{equation*} 
\tpur{we estimate the error in the QoI as $E\big(\varepsilon_h \, , \,  \vdg-v_h^\ast\big) := \big|\big(\varepsilon_h \, , \,  \vdg-v_h^\ast\big)_{V_h}\big|$ and its local upper-bounds $\|\varepsilon_h\|_{T} \, \| \vdg-v_h^\ast\|_{T}$.}
\item[B)] {\bf Adjoint residual based estimator}. We solve the residual representative problem:
\begin{equation*}
\left\{\begin{array}{l}
\text{Find } \varepsilon_h^\ast \in V_h,  \text{ such that:} \\
(\varepsilon_h^\ast, v_h)_{V_h} = q(v_h)-b_h(v_h,v_h^\ast), \quad \forall \, v_h \in V_h,
\end{array}
\right.
\end{equation*}
\tpur{we estimate the error in the QoI as $E\big(\varepsilon_h\, , \, \varepsilon_h^\ast\big) := \big|\big(\varepsilon_h\, , \, \varepsilon_h^\ast\big)_{V_h}\big|$ and its local upper-bounds  $\|\varepsilon_h\|_{T} \, \| \varepsilon_h^\ast\|_{T}$.}
\end{itemize}
\item \label{en:step4} {\bf Marking criteria:} We use the local estimations to guide the goal oriented adaptivity by using the D{\"o}rfler bulk-chasing marking criterion (see~\cite{ dorfler1996convergent}). The strategy first orders in a decreasing order the local error estimates. Then, it marks for refinement the elements for which the cumulative sum remains smaller than a given percentage of the total error estimate.
\end{enumerate}
This procedure requires no further a posteriori error estimation. Additionally, steps~\ref{en:step1} and~\ref{en:step2} require the solution of the same saddle-point system. Thus, the problem becomes a single system with multiple right-hand sides. Finally, this process does not require the adjoint consistency assumption as in the standard dG goA strategy, since we obtain upper bounds in terms of the dG discrete solutions. This subtle insight significantly enlarges the range of available formulations that we can apply to this class of problems. 

\section{\tred{Numerical examples}}\label{sec:numerical_examples}
We consider several test cases focusing on advection-diffusion-reaction problems. These examples demonstrate the performance of the GoA strategy described in Section~\ref{sec:GoA_strategy}. We use FEniCS~\cite{  alnaes2015fenics} to perform the simulations. We consider a QoI of the form:
\begin{equation}
\displaystyle q(u)= {1\over|\Omega_0|}\int_{\Omega_0} u,
\end{equation}
where $u$ denotes the analytic solution of the corresponding problem, and $\Omega_0$ is a subdomain of the physical domain $\Omega$. For a given polynomial degree $p\geq1$ and $\Delta_p \in \{0,1\}$, we consider the test space $V_h$ to be a standard discontinuous piece-wise polynomial space of degree $p+\Delta_p$, and the trial space $U_h$ as a standard FEM subspace of continuous piece-wise polynomial functions of degree $p$. The initial mesh is $\Omega_0$-conforming. We then perform a loop following the standard modules in adaptive procedures:
$$
\text{ SOLVE } \rightarrow \text{ ESTIMATE } \rightarrow \text{ MARK }
\rightarrow  \text{ REFINE. }
$$
The estimation procedure considers independent adaptive mesh refinements based on the two GoA estimators defined in Step~\ref{en:step3} of the GoA algorithm, Section~\ref{sec:GoA_strategy}, and we compare their performance with respect to the energy norm-based error estimate $E\big(\varepsilon_h \, , \,  \varepsilon_h \big):=\|\varepsilon_h\|^2_{V_h}$ (see Remark~\ref{rem:equal_inner_rem}).\\
The marking procedure follows the D{\"o}rfler bulk-chasing criterion with the corresponding fraction to be 20\% (see Remark~\ref{rem:equal_inner_rem}). Finally, we employ a bisection-type refinement criterion~\cite{ bank1983some}. In the first two examples, we use $LU$ direct solver. In the last example, we use an iterative scheme (cf.~\cite{ rojas2019adaptive}) on the resulting multiple right hand sides.\\
\tred{In the numerical examples,  we compare the error plots against the expected optimal convergence for the error in the QoI (see~\cite{ feischl2016abstract}):
\begin{eqnarray}\label{eq:optimal_rate}
\textrm{NDOFs}^{-\frac{2 \, (p + r)}{d}},
\end{eqnarray}
with $d=2,3$, being the dimentionality of the physical domain $\Omega$,  $p\geq 1$ the polynomial degree of the trial space,  $\textrm{NDOFs}$ the total number of degrees of freedom of the saddle-point problem~\eqref{eq:mixed_direct},  and $r=1,1/2,0$ for reaction-dominated,  advection-dominanted, and diffusion-dominated problems,  respectively.}
\subsection{Diffusion problem}

\begin{figure}[ht!]
    \centering
    %\hspace{0.3cm}
   % \begin{subfigure}[b]{0.4\textwidth}
    \includegraphics[width=0.4\textwidth]{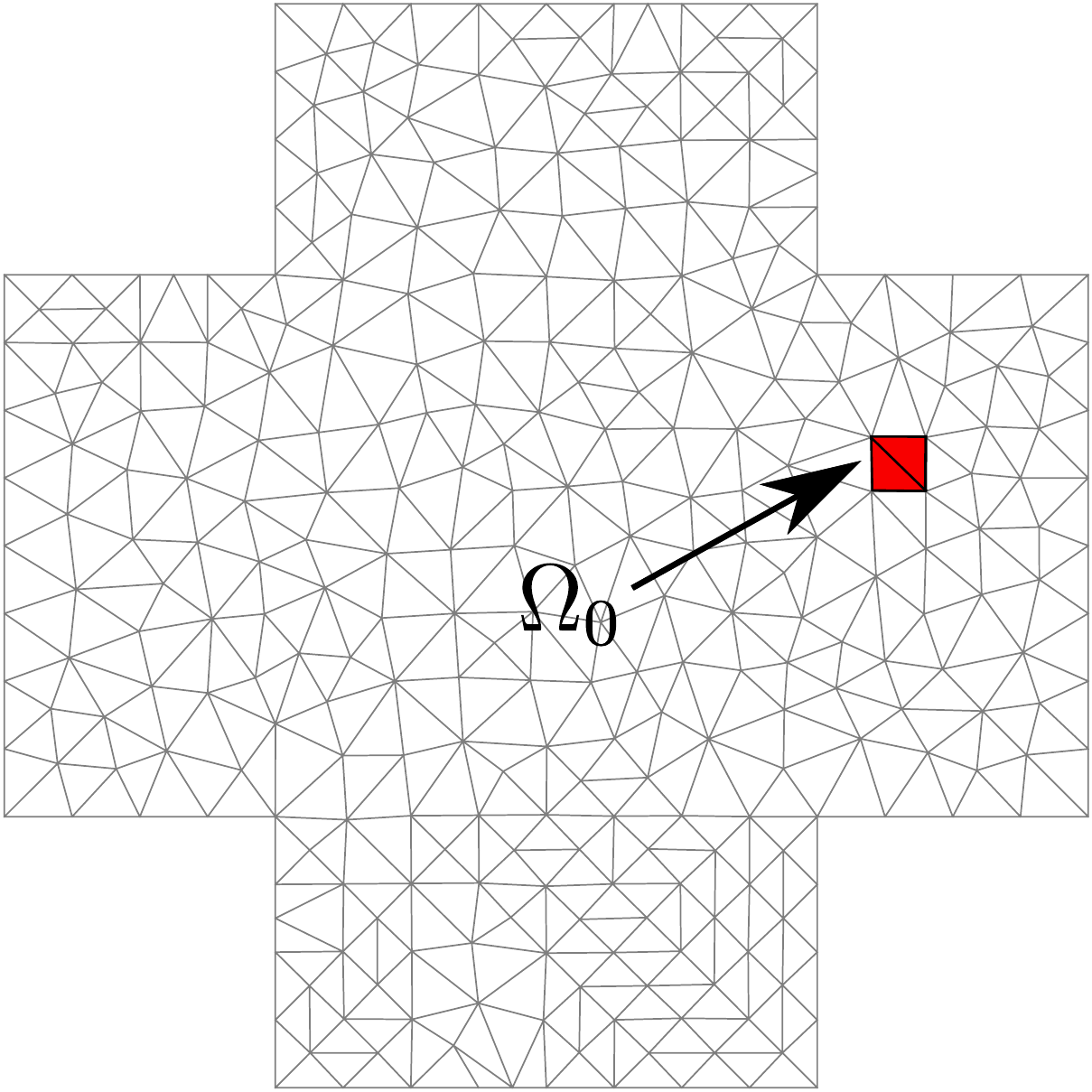}
    %    \caption{Reference solution}
     %   \label{fig:cross_shape}
    %%\end{subfigure} 
    %
    \caption{Cross-shaped domain initial mesh}
    \label{fig:cross_shape_data_set}
\end{figure}
\begin{figure}[ht!]
    \centering
    %\hspace{0.3cm}
    \begin{subfigure}[b]{0.45\textwidth}
    \includegraphics[width=\textwidth]{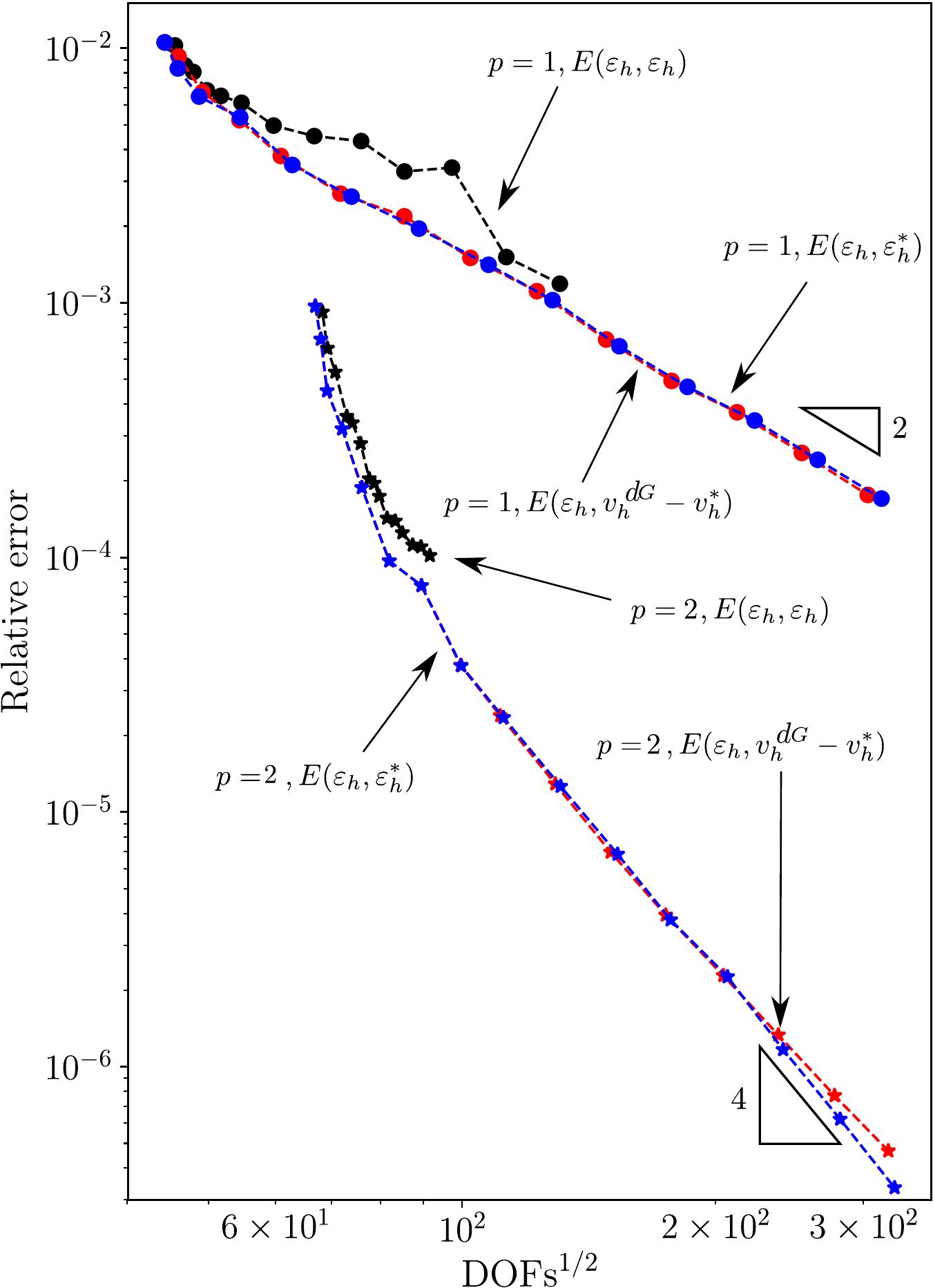}
        \caption{Symmetric Interior Penalty (SIP)}
        \label{fig:cross_goal_sip}
    \end{subfigure} 
    \hspace{0.1cm}
    \begin{subfigure}[b]{0.45\textwidth}
        \includegraphics[width=\textwidth]{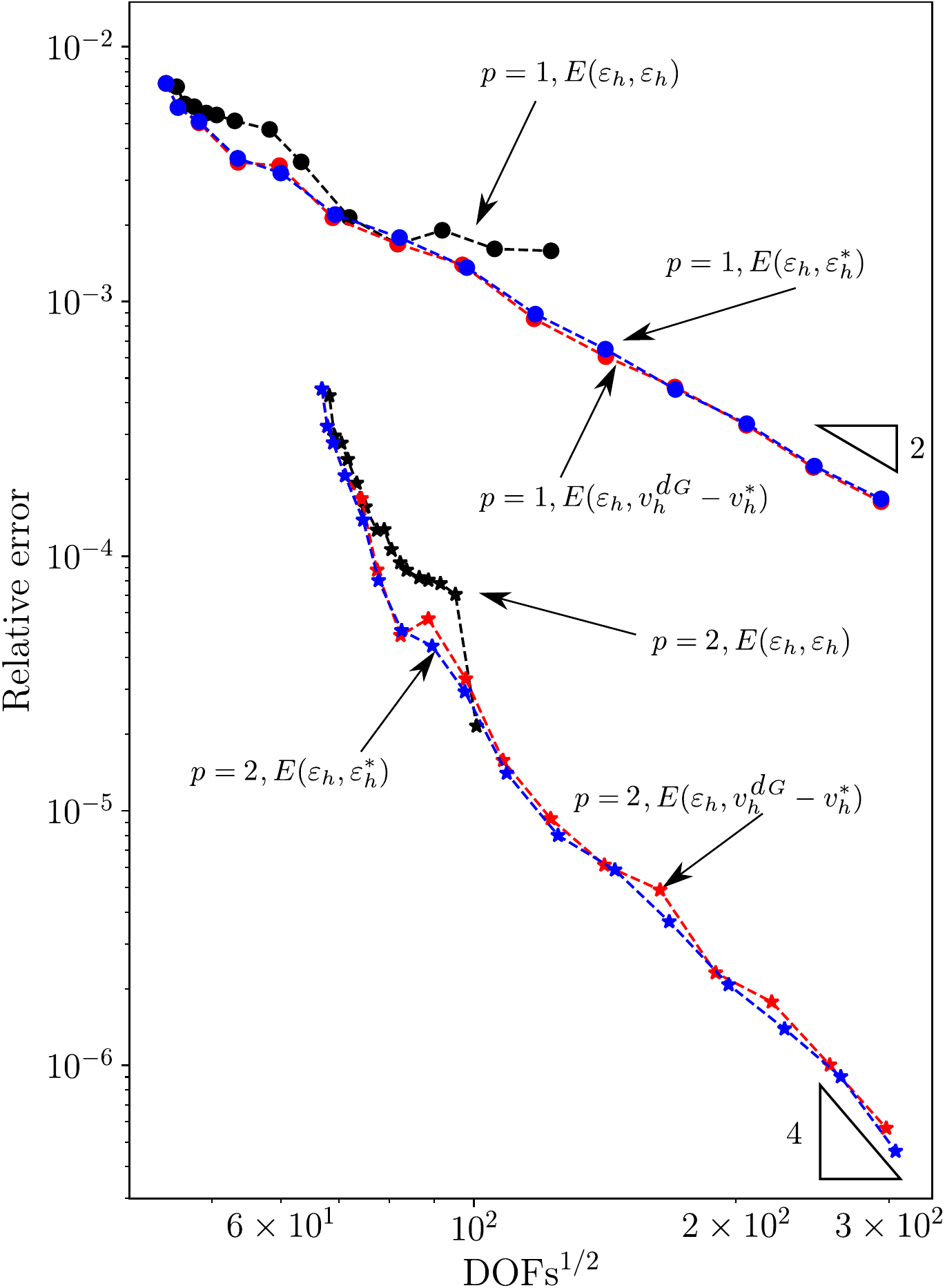}
        \caption{Nonsymmetric Interior Penalty (NIP)}
        \label{fig:cross_goal_nip}
    \end{subfigure}
    \caption{Relative error comparison in the quantity of interest (QoI) using the Symmetric Interior Penalty (SIP) and Nonsymmetric Interior Penalty (NIP) schemes for $p=1,2$ and $\Delta_p = 0$.}
    \label{fig:cross_goal}
\end{figure}

\begin{figure}[ht!]
    \centering
    %\hspace{0.3cm}
    \begin{subfigure}[b]{0.45\textwidth}
    \includegraphics[width=\textwidth]{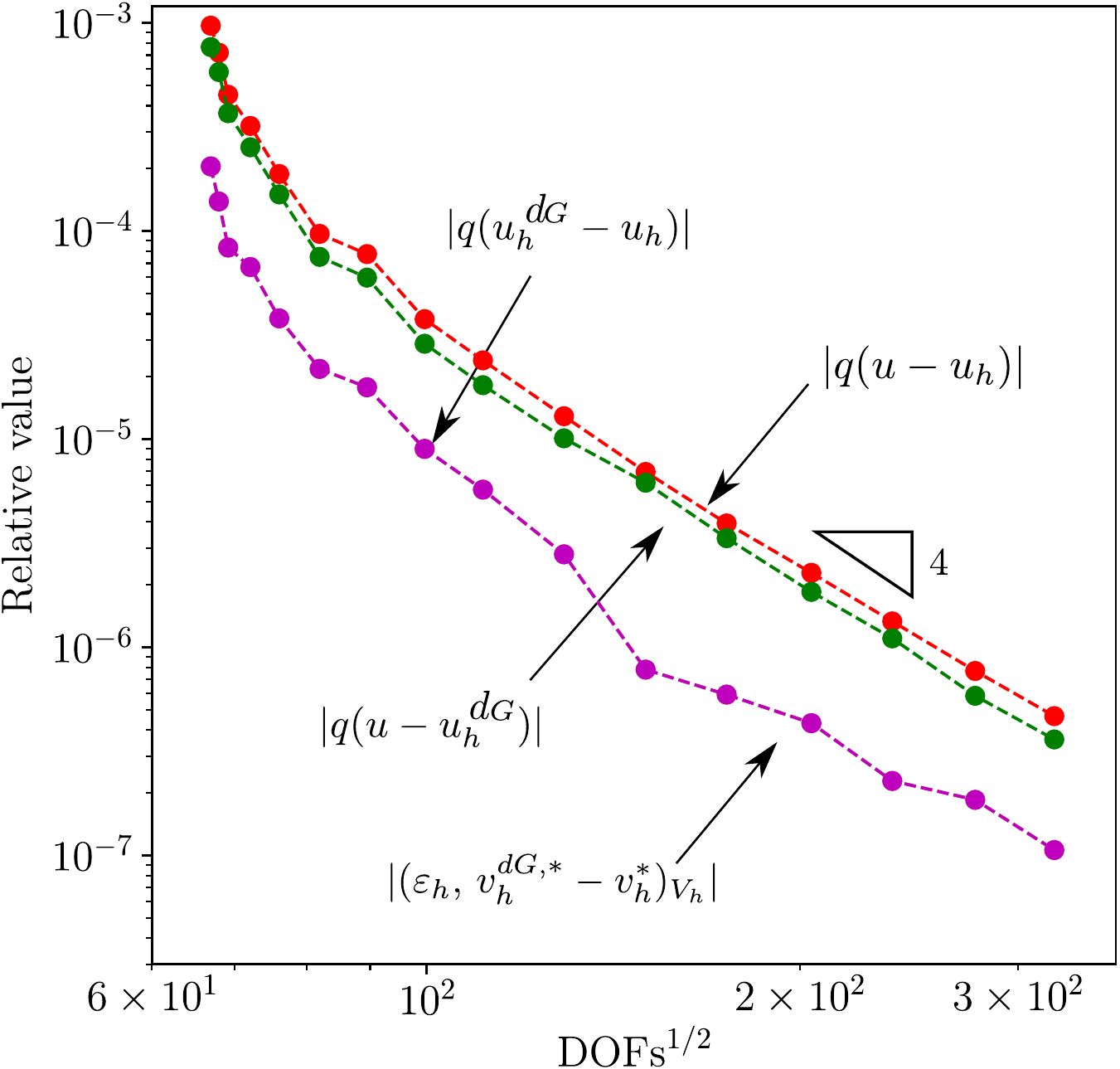}
        \caption{SIP with estimator $E\big(\varepsilon_h\, , \, \vdg-v_h^\ast\big)$}
        \label{fig:cross_A_sip_22}
    \end{subfigure} 
    \hspace{0.2cm}
    \begin{subfigure}[b]{0.45\textwidth}
        \includegraphics[width=\textwidth]{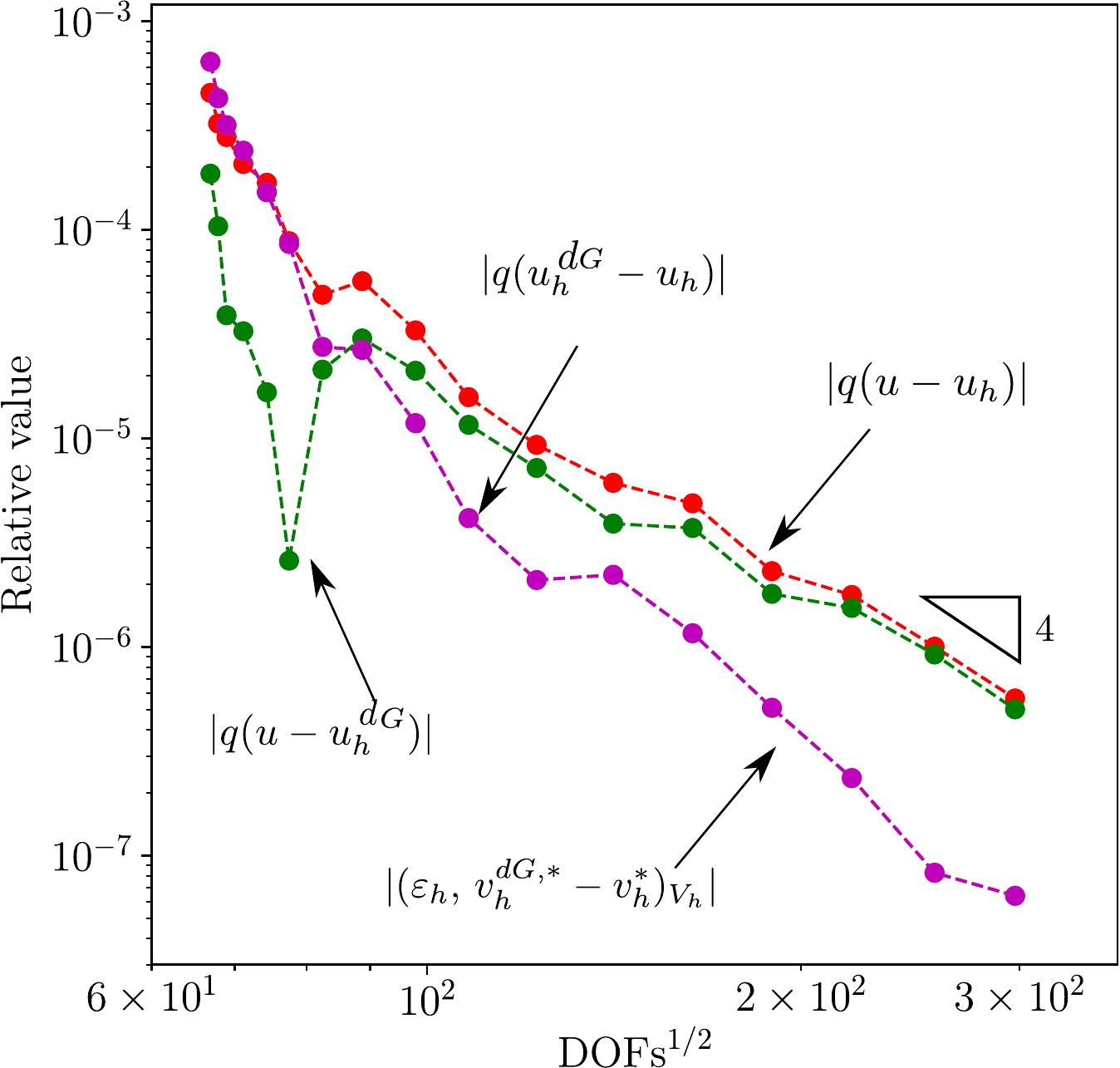}
        \caption{NIP with estimator $E\big(\varepsilon_h\, , \, \vdg-v_h^\ast\big)$}
        \label{fig:cross_A_nip_22}
    \end{subfigure}
    \begin{subfigure}[b]{0.45\textwidth}
    \includegraphics[width=\textwidth]{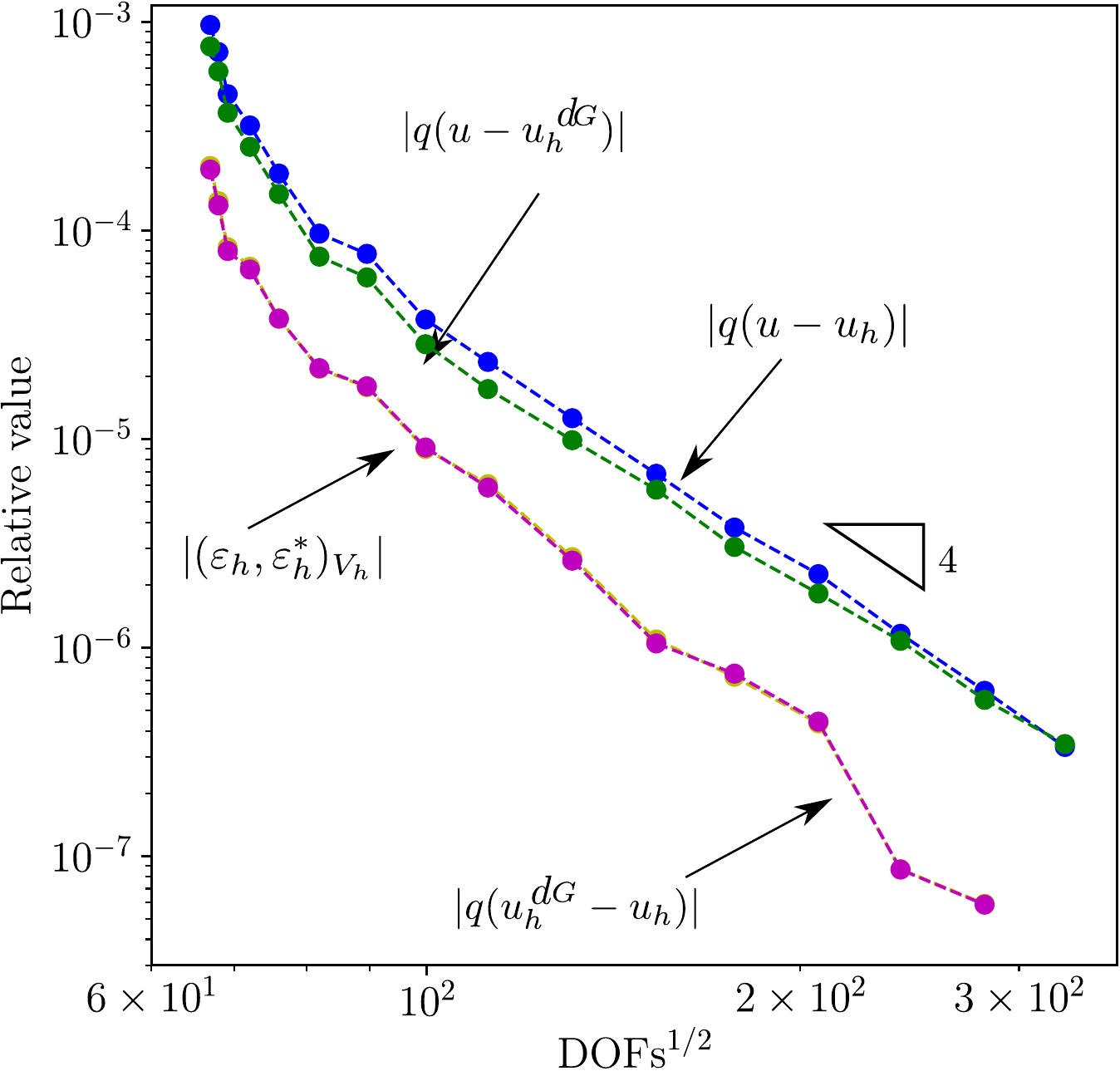}
        \caption{SIP with estimator $E\big(\varepsilon_h\, , \, \varepsilon_h^\ast\big)$}
        \label{fig:cross_B_sip_22}
    \end{subfigure} 
    \hspace{0.2cm}
    \begin{subfigure}[b]{0.45\textwidth}
        \includegraphics[width=\textwidth]{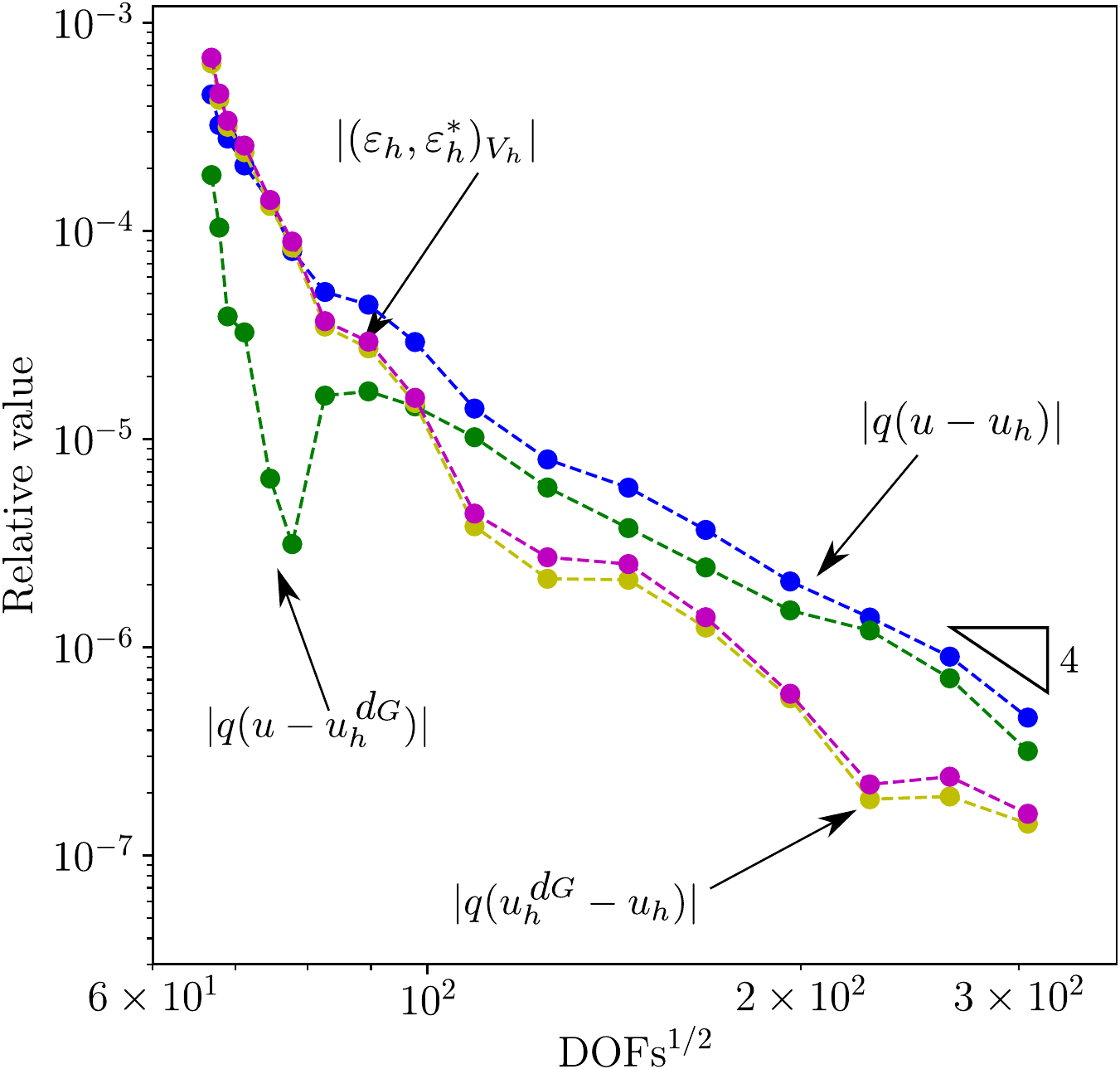}
        \caption{NIP with estimator $E\big(\varepsilon_h\, , \, \varepsilon_h^\ast\big)$}
        \label{fig:cross_B_nip_22}
    \end{subfigure}

    \caption{Several relative value comparisons using the Symmetric Interior Penalty (SIP) and Nonsymmetric Interior Penalty (NIP) schemes for $p=2$ and $\Delta_p = 0$.}
    \label{fig:cross_comp_22}
\end{figure}

\begin{figure}[ht!]
    \centering
    %\hspace{0.3cm}
    \begin{subfigure}[b]{0.45\textwidth}
    \includegraphics[width=\textwidth]{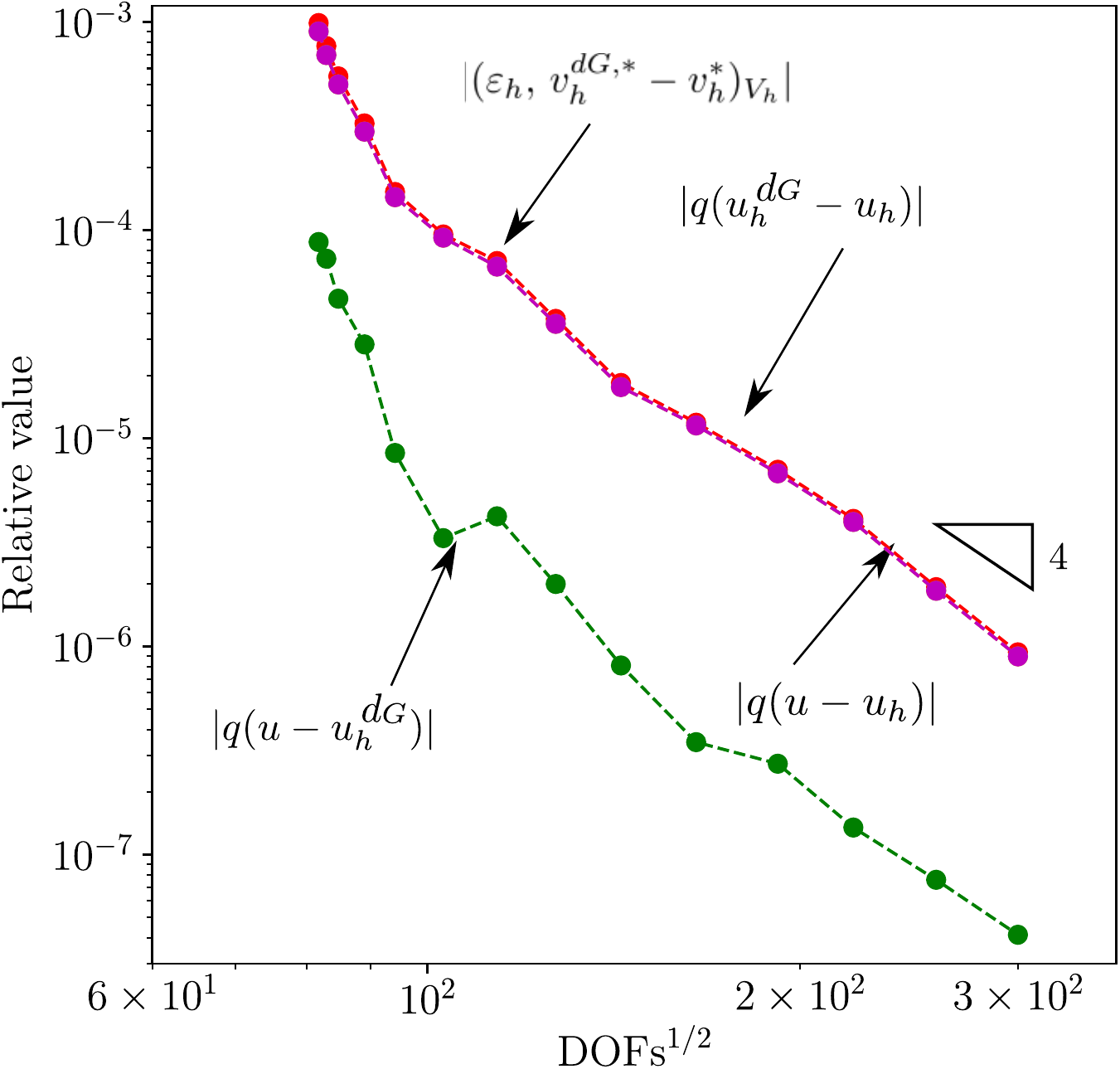}
        \caption{SIP with estimator $E\big(\varepsilon_h\, , \, \vdg-v_h^\ast\big)$}
        \label{fig:cross_A_sip_23}
    \end{subfigure} 
    \hspace{0.2cm}
    \begin{subfigure}[b]{0.45\textwidth}
        \includegraphics[width=\textwidth]{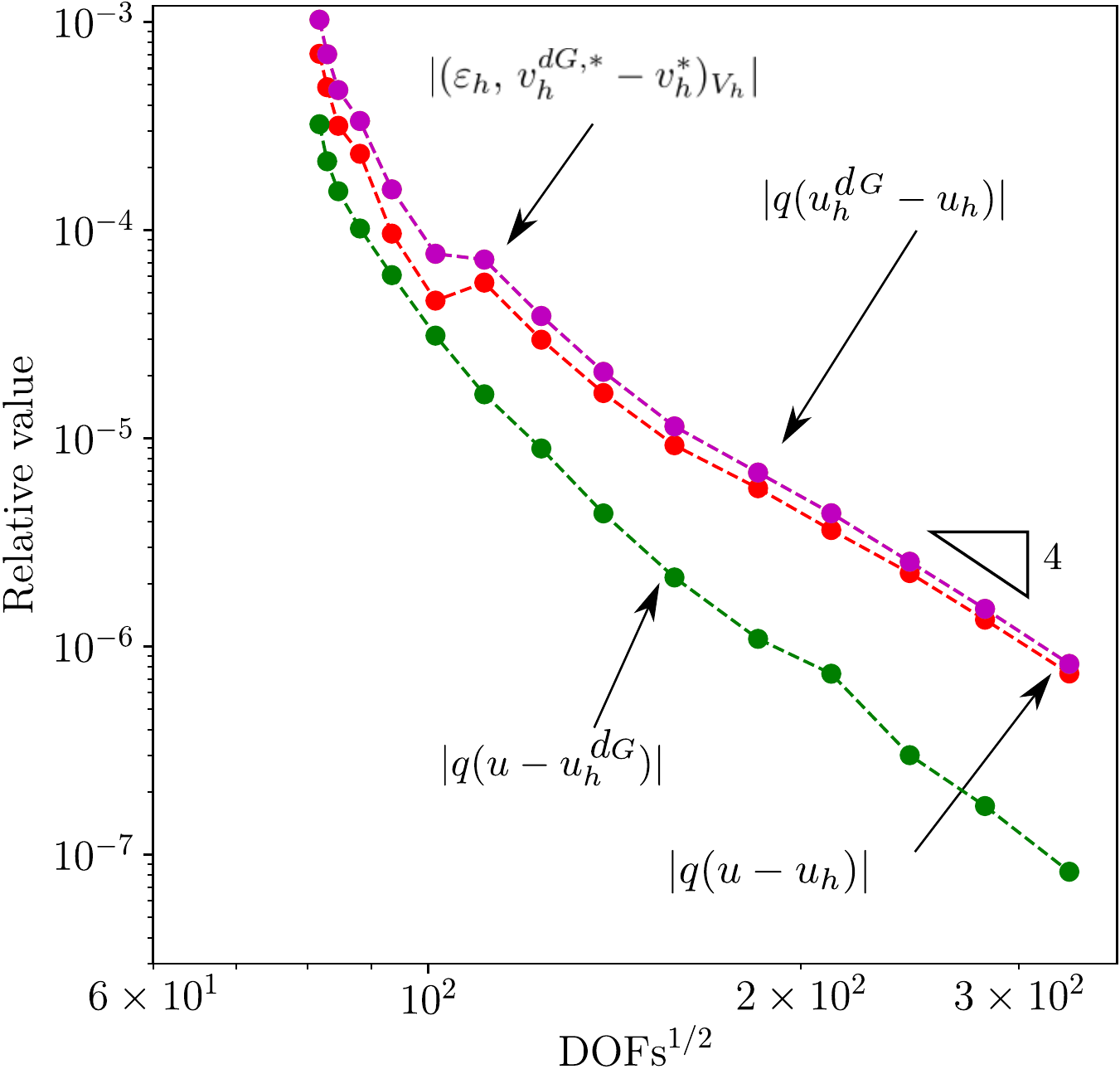}
        \caption{NIP with estimator $E\big(\varepsilon_h\, , \, \vdg-v_h^\ast\big)$}
        \label{fig:cross_A_nip_23}
    \end{subfigure}
    \begin{subfigure}[b]{0.45\textwidth}
    \includegraphics[width=\textwidth]{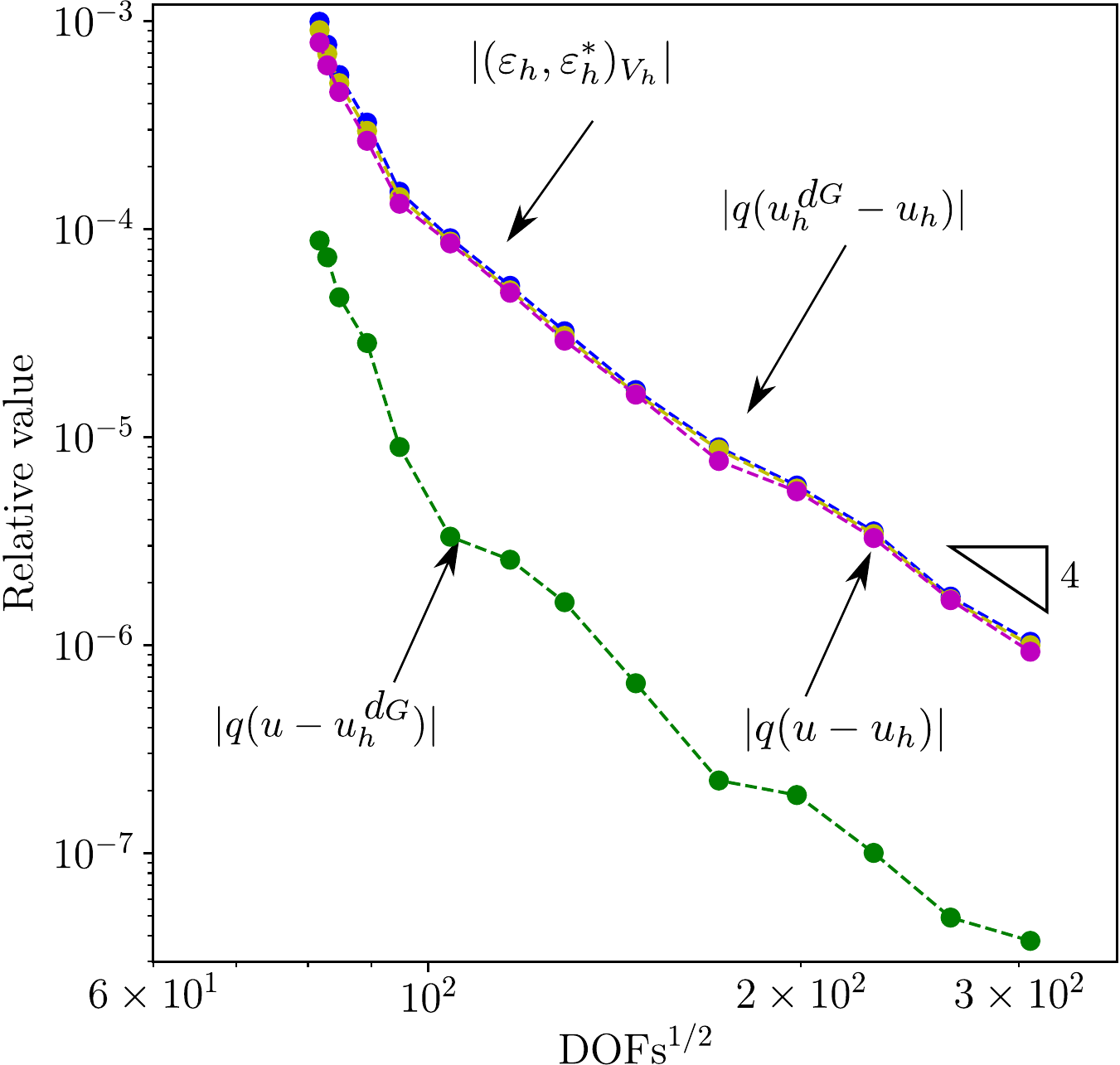}
        \caption{SIP with estimator $E\big(\varepsilon_h\, , \, \varepsilon_h^\ast\big)$}
        \label{fig:cross_B_sip_23}
    \end{subfigure} 
    \hspace{0.2cm}
    \begin{subfigure}[b]{0.45\textwidth}
        \includegraphics[width=\textwidth]{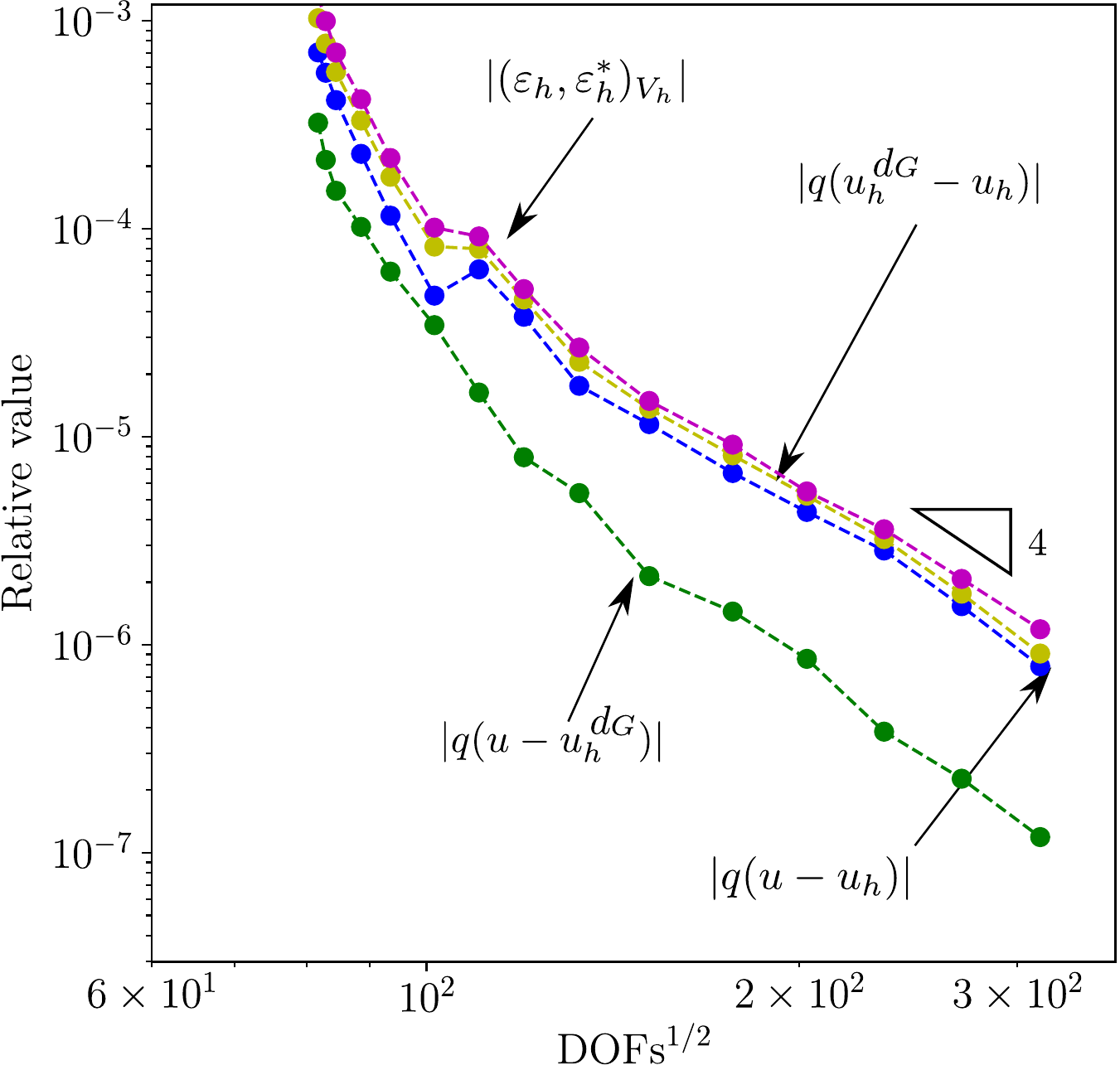}
        \caption{NIP with estimator $E\big(\varepsilon_h\, , \, \varepsilon_h^\ast\big)$}
        \label{fig:cross_B_nip_23}
    \end{subfigure}

    \caption{Several relative value comparisons using the Symmetric Interior Penalty (SIP) and Nonsymmetric Interior Penalty (NIP) schemes for $p=2$ and $\Delta_p = 1$.}
    \label{fig:cross_comp_23}
\end{figure}
We start by comparing our strategy with the state-of-the-art GoA strategies. For this, we consider the diffusion problem:  
\begin{align}\label{eq:num_diff}
\left\{\begin{array}{l}
\text{Find } u \text{ such that:} \smallskip \\
\begin{array}{rll}
-\Delta u  \hspace{-0.2cm} &= \, f, & \text{ in } \Omega, \smallskip\\
u \hspace{-0.2cm} &= \, 0, & \text{ on } \partial \Omega,
\end{array}
\end{array}\right.
\end{align}
defined in the cross-shaped domain $\Omega = (-2,2) \times (-1,1) \cup (-1,1) \times (-2,2)$, with $f=1$ as forcing term. Equations~\eqref{eq:linear_dar} and~\eqref{eq:bilinear_dar} express the corresponding dG bilinear and linear forms, respectively, while~\eqref{eq:vh_norm} defines the discrete inner product with $\kappa = 1$, ${\bf b} = {\bf 0}$, and $\gamma = 0$. Our description allows for two possible formulations: the Symmetric Interior Penalty (SIP) formulation ($\epsilon = -1$) being adjoint consistent (see Section~\ref{sec:adjoint_dg}), and the Nonsymmetric Interior Penalty (NIP) formulation ($\epsilon = 1$), which is not adjoint consistent. We compare our results with~\cite{ dolejvsi2017goal}, considering the QoI defined over the subdomain $\Omega_0 = (1.2,1.4) \times (0.2,0.4)$ (see Fig.~\ref{fig:cross_shape_data_set}), and the same reference value $q(u) = 0.407617863684$ computed previously in~\cite{  ainsworth2012guaranteed}.
Figure~\ref{fig:cross_goal_sip} shows the relative error for three estimators considering the SIP formulation. We consider two polynomial orders for the trial space (namely, $p=1,2$), with a test space of the same polynomial order (i.e., $\Delta_p = 0$). We plot the evolution of the relative error $|q(u-u_h)|/|q(u)|$ versus the square root of the total number of degrees of freedom in the system (DOFs) we use to solve the saddle-point problem~\eqref{eq:mixed_direct} (i.e.~$\dim(V_h) + \dim(U_h)$). In Figure~\ref{fig:cross_goal_nip}, we repeat these plots considering the NIP formulation.  These figures show up to fourteen levels of refinement. As our theoretical analysis predicts, the convergence rates of the GoA error estimates $E\big(\varepsilon_h\, , \, \vdg-v_h^\ast\big)$ and $E\big(\varepsilon_h\, , \, \varepsilon_h^\ast\big)$ are significantly better than those the energy norm delivers. Moreover, both GoA error estimates we propose deliver optimal convergence rates for both reference dG formulations when $p=1$ (see Equation~\eqref{eq:optimal_rate}).  Instead, for $p=2$, we observe that the most efficient estimation, which is also optimal and in line with the results in~\cite{ dolejvsi2017goal}, is obtained with the SIP formulation together with the GoA estimator $E\big(\varepsilon_h\, , \, \varepsilon_h^\ast\big)$.
To validate the analysis we present in Section~\ref{sec:goal_estimator}, in Figure~\ref{fig:cross_comp_22} we compare several of the involved discrete quantities (scaled by $|q(u)|$).  Figures~\ref{fig:cross_A_sip_22} and~\ref{fig:cross_A_nip_22} show a comparison for the SIP and NIP formulations, respectively, considering the estimator $E\big(\varepsilon_h\, , \, \vdg-v_h^\ast\big)$. Figures~\ref{fig:cross_B_sip_22} and~\ref{fig:cross_B_nip_22} show a comparison for the SIP and NIP formulations, respectively,  considering the estimator $E\big(\varepsilon_h\, , \, \varepsilon_h^\ast\big)$. In the figures, the quantity $|q(u-\udg)|$ remains below the curve $|q(u-u_h)|$, implying that the Assumption~\ref{as:saturation_GO} is satisfied.  Moreover, the quantities $|q(u-u_h)|$ and $\big|q(\udg-u_h)\big|=\big|\big(\varepsilon_h\, , \, \vdg-v_h^\ast\big)_{V_h} \big|$, or $\big|\big(\varepsilon_h\, , \, \varepsilon_h^\ast\big)_{V_h}\big|$ when it corresponds, have the same rate of convergence. Figures also show that the respective estimation is sharper if the value $\big|q(u-\udg)\big|$ becomes smaller, since this reduces the constant in Assumption~\ref{as:saturation_GO}.
There exist several alternatives to improve the GoA estimation if required, for instance, by increasing the polynomial order of the test space. In Figure~\ref{fig:cross_comp_23} we repeat the plots of Figure~\ref{fig:cross_comp_22}, considering a test space of degree $p_t=3$. Even if this implies to solve a larger system, from figures we can appreciate a clear improvement in the estimation, while conserving the expected convergence rate in terms of the DOFs. Moreover, the robustness of the error estimation with respect to the discrete formulation shows that our GoA algorithm is independent of the adjoint consistency assumption.
\subsection{Advection-reaction problem}

\begin{figure}[ht!]
    \centering
    \begin{subfigure}[b]{0.45\textwidth}
    \includegraphics[width=\textwidth]{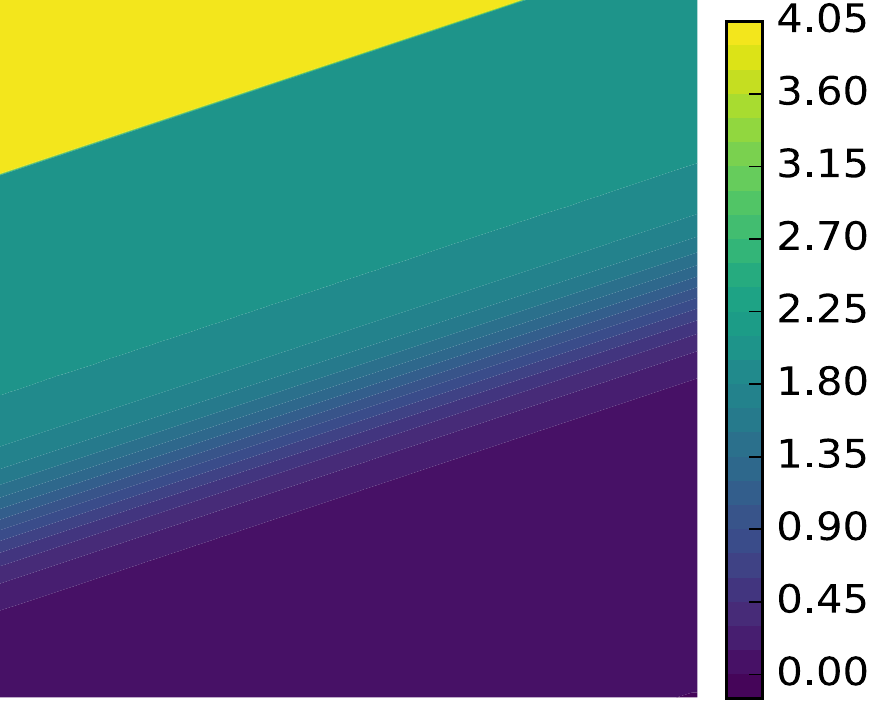}
     \caption{Reference solution}
        \label{fig:ad_re_ref_sol}
    \end{subfigure}
    \hspace{0.2cm}
    \begin{subfigure}[b]{0.36\textwidth}
        \includegraphics[width=\textwidth]{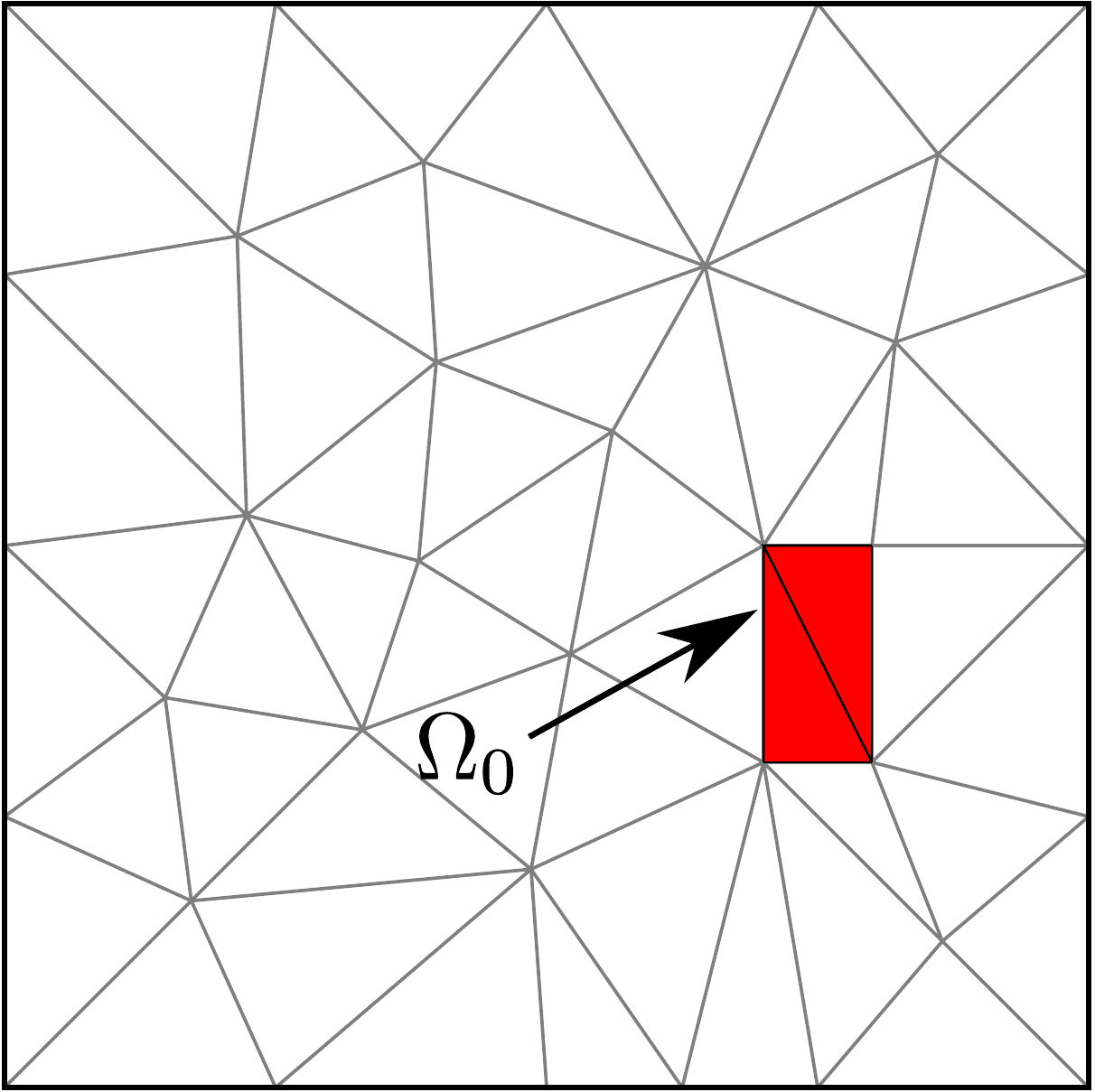}
         \caption{Initial mesh}
        \label{fig:ad_re_ini_mesh}
    \end{subfigure}
    \caption{Reference solution and initial mesh}
    \label{fig:ad_re_data_set}
\end{figure}

As a second example, we consider the advection-reaction problem~\eqref{eq:advection_reaction} in the unit square $\Omega = (0 \, , \, 1)^2 \subset \R^2$, with a constant velocity field $\Vb = (3 \, , \, 1)^T$. For a given $\gamma\geq0$, the source term is $f = \gamma \, u$ in $\Omega$, and an inflow boundary datum $g^-=u|_{\Gamma^-}$, where $\Gamma^-=\{(0,y),y\in(0,1)\}\cup\{(x,0),x\in(0,1)\}$, and the exact solution $u$ is (see Figure~\ref{fig:ad_re_ref_sol}):
\begin{equation}\label{eq:adv_anal}
 u(x_1,x_2) = 2+\tanh\big(10 \big(x_2-\dfrac{x_1}{3} - \dfrac{1}{4}\big) \big)+\tanh\big(1000 \big(x_2-\dfrac{x_1}{3} - \dfrac{3}{4}\big) \big).
\end{equation}
Since $\kappa = 0$, the dG bilinear and linear forms correspond to equations~\eqref{eq:linear_dar} and~\eqref{eq:bilinear_dar}, respectively, while the discrete inner product corresponds to equation~\eqref{eq:vh_norm}.

The nature of the analytical solution $u$ implies that an adaptive algorithm based on the energy norm refines in a neighbourhood of the characteristic line starting from the inflow boundary at $y=3/4$ (cf.~\cite{ rojas2019adaptive}). We analyze a pure advection case ($\gamma = 0$),  and a reaction-dominant case ($\gamma = 1000$). We set $\Omega_0 = (0.7,0.8)\times(0.3,0.5)$ as the subdomain that defines the QoI, and we consider the $\Omega_0$-conforming mesh of Figure~\ref{fig:ad_re_ini_mesh} as our starting point for the adaptive procedure.

\begin{figure}[h]
	\centering
	\begin{subfigure}[b]{0.48\textwidth}
		\includegraphics[width=\textwidth]{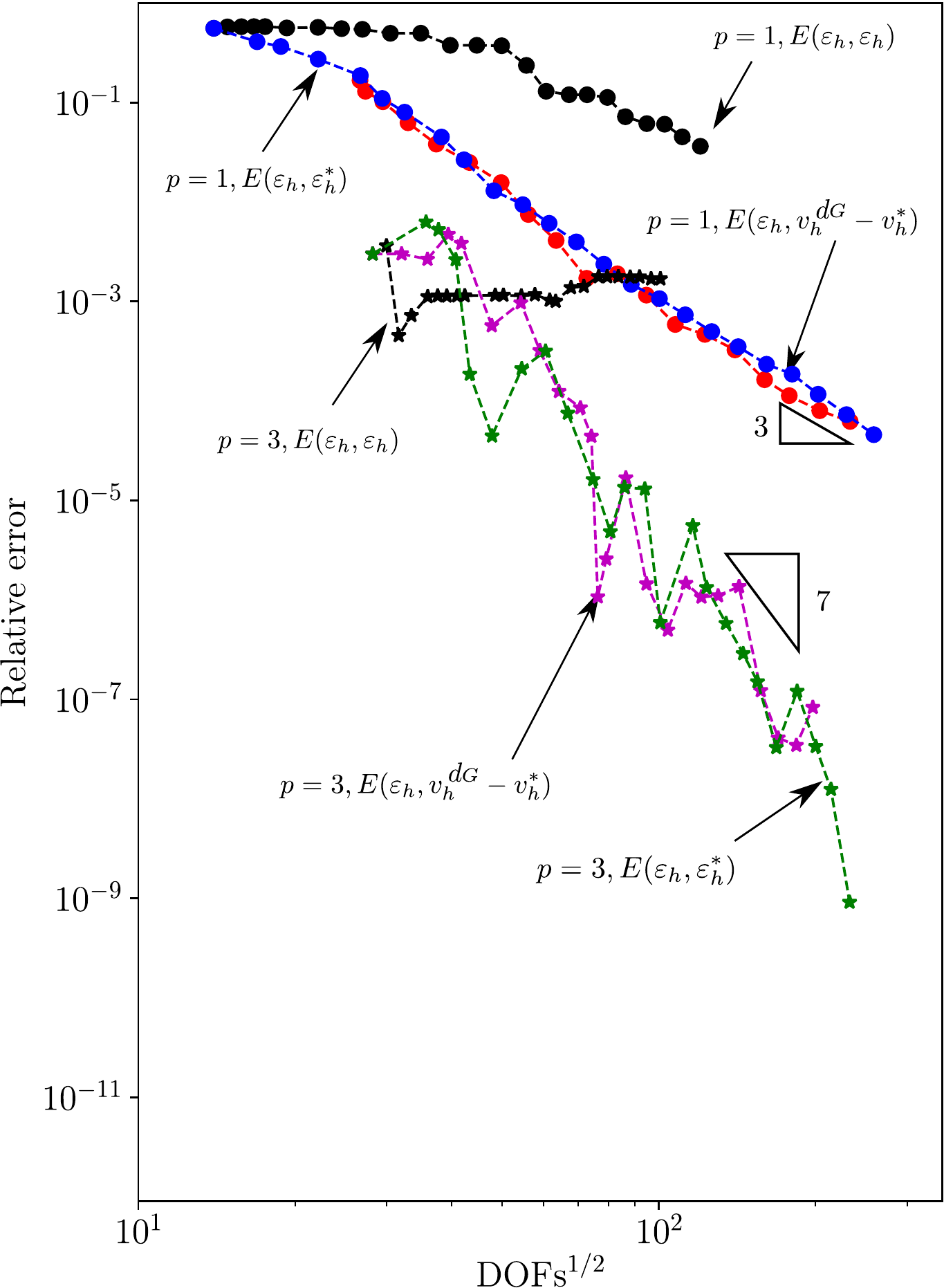}

		\caption{Pure advection ($\gamma =0$)}
		\label{fig:ad_re_all_0}
	\end{subfigure}
	\begin{subfigure}[b]{0.48\textwidth}
		\includegraphics[width=\textwidth]{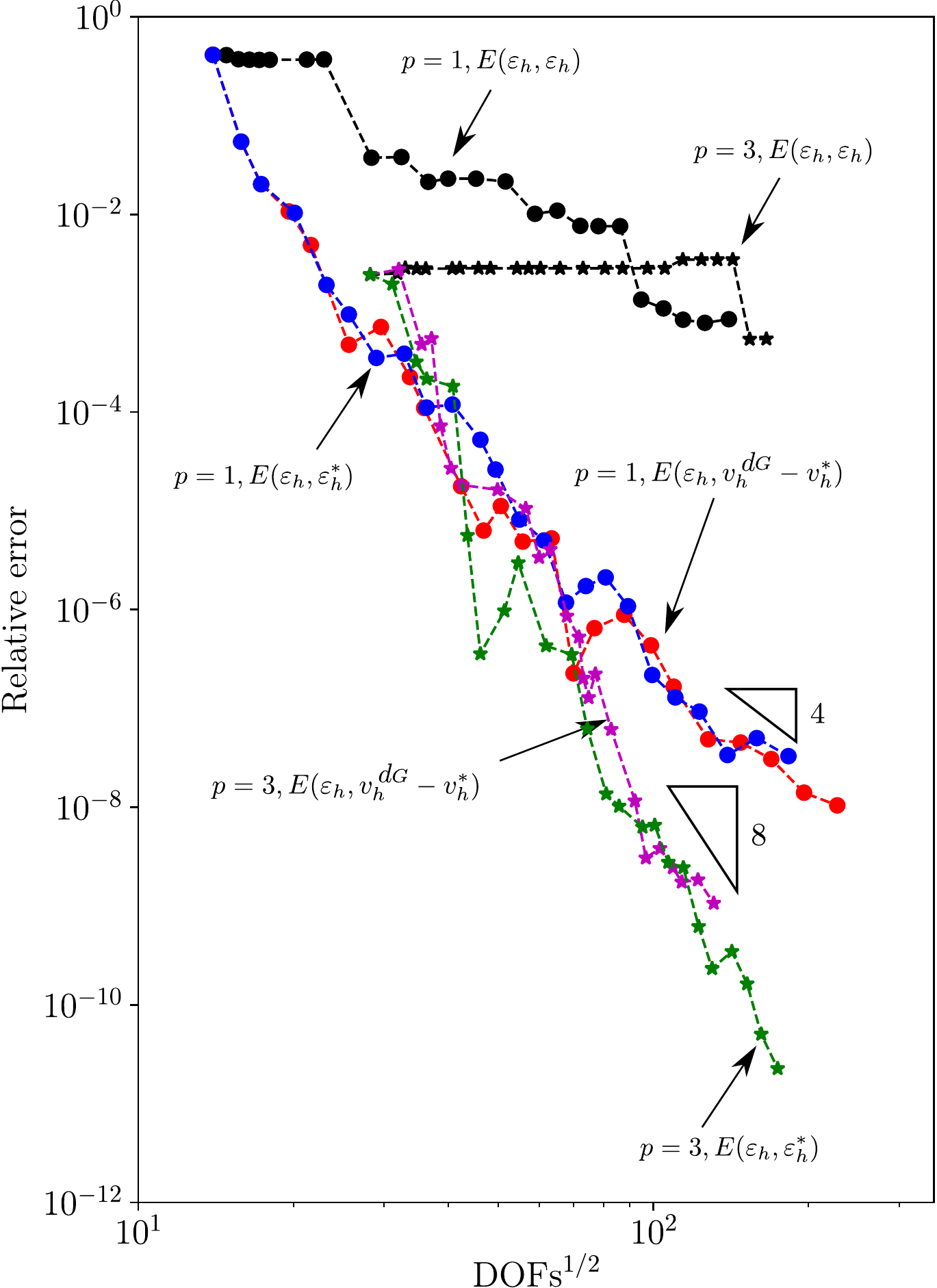}

		\caption{Reaction-dominant ($\gamma = 1000$) }
		\label{fig:ad_re_all_100}
	\end{subfigure}

  \caption{Relative error ($|q(u-u_h)|/|q(u)|$) in quantity of
    interest (QoI) using upwinded (UP) scheme for (a) pure advection
    and (b) reaction-dominant problems.  (See
    Section~\ref{sec:upper_estimates} for a detailed discussion.)}
  \label{fig:ad_re_relative_all}
\end{figure}

We consider two polynomial orders ($p=1,3$) with test space of the same polynomial order (i.e., $\Delta_p = 0$), and we repeat the plots of previous section. Figure~\ref{fig:ad_re_all_0} shows the evolution of the relative error $|q(u-u_h)|/|q(u)|$ versus the square root of the total number of degrees of freedom in the system (DOFs). In Figure~\ref{fig:ad_re_all_100}, we repeat these plots for the reaction dominant case. These figures show up to eighteen levels of refinement. Again, both GoA error estimates $\big(\varepsilon_h\, , \, \varepsilon_h^\ast\big)_{V_h}$ and $\big(\varepsilon_h\, , \, \vdg-v_h^\ast\big)_{V_h}$ deliver similar results, which are significantly better than those the energy norm delivers.
\begin{figure}[h]
    \centering
    \begin{subfigure}[b]{0.48\textwidth}
        \includegraphics[width=\textwidth]{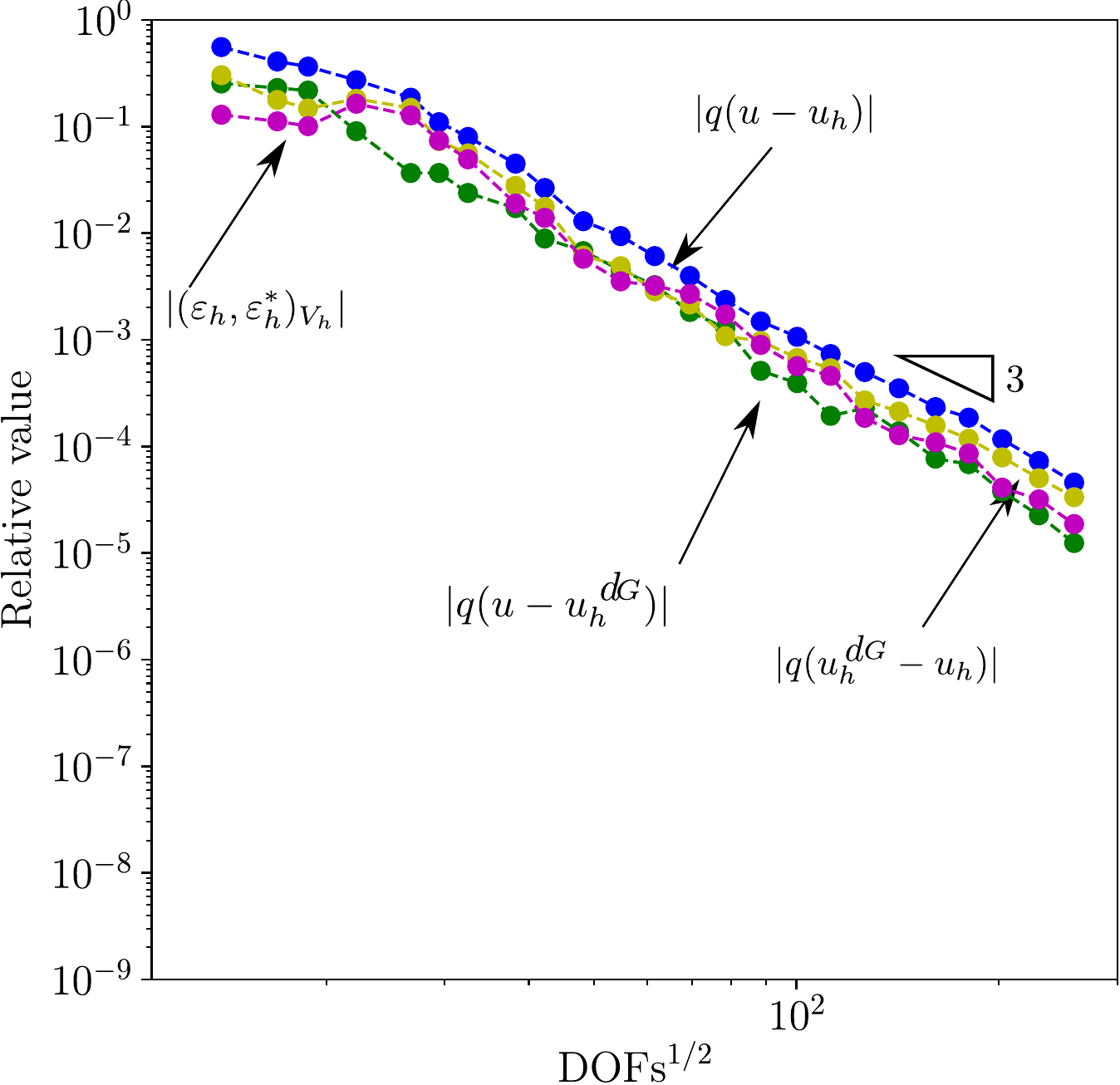}
        \caption{Pure advection ($\gamma = 0$)}
        \label{fig:ad_re_type_1_p1_0}
    \end{subfigure}
   \begin{subfigure}[b]{0.48\textwidth}
        \includegraphics[width=\textwidth]{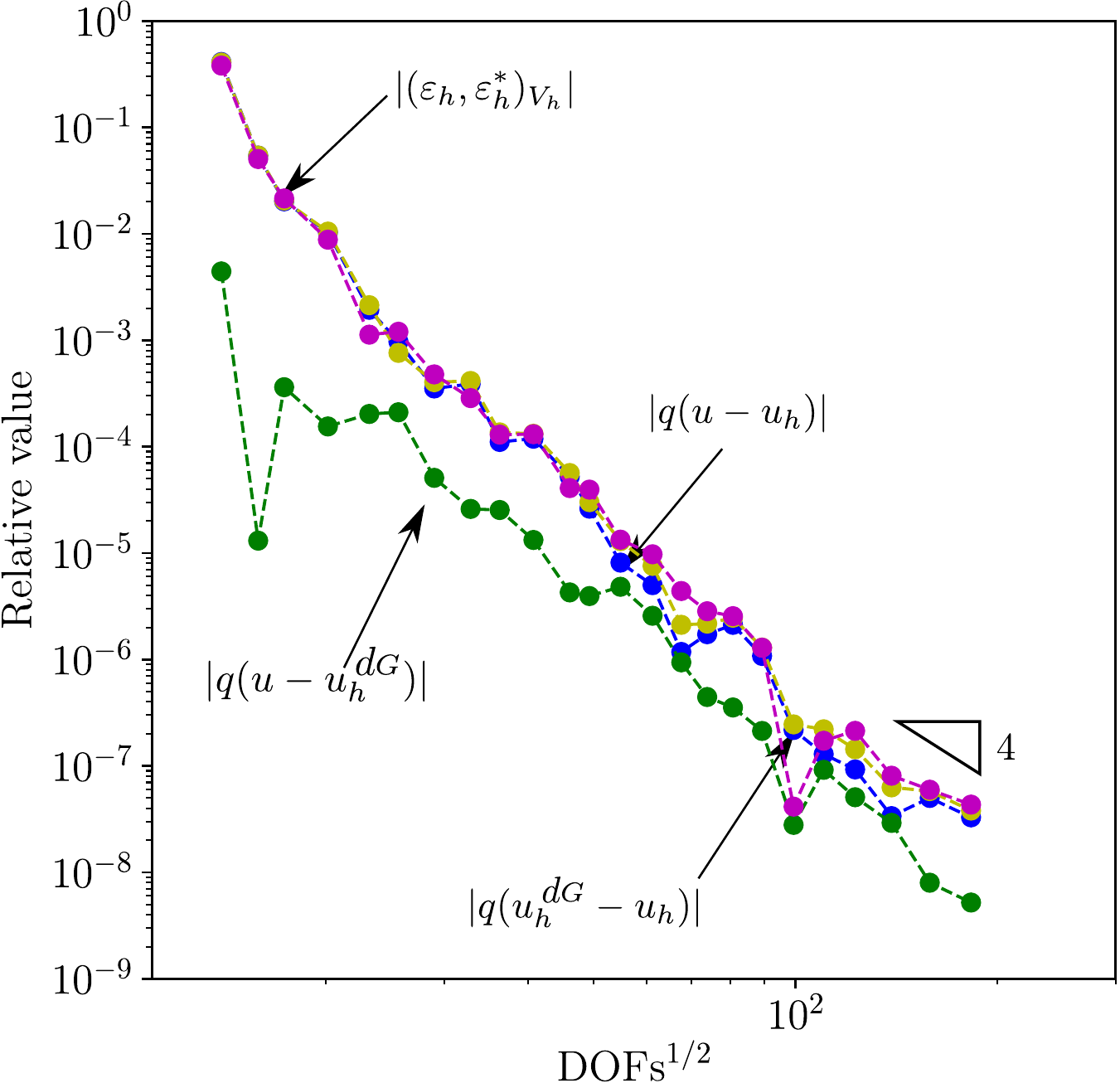}
        \caption{Reaction-dominant ($\gamma = 1000$)}
        \label{fig:ad_re_type_1_p1_100}
    \end{subfigure}
    \caption{Discrete relative error for upwinded (UP) scheme for
      pure advection with $p=1$ and $\Delta_p = 0$: goal-oriented error estimate $\big(\varepsilon_h\,,\, \varepsilon_h^\ast\big)_{V_h}$ (see
      Section~\ref{sec:upper_estimates} for definitions).}
    \label{fig:ad_re_type_1_0_all}
\end{figure}

\begin{figure}[h]
    \centering 
\begin{subfigure}[b]{0.48\textwidth}
        \includegraphics[width=\textwidth]{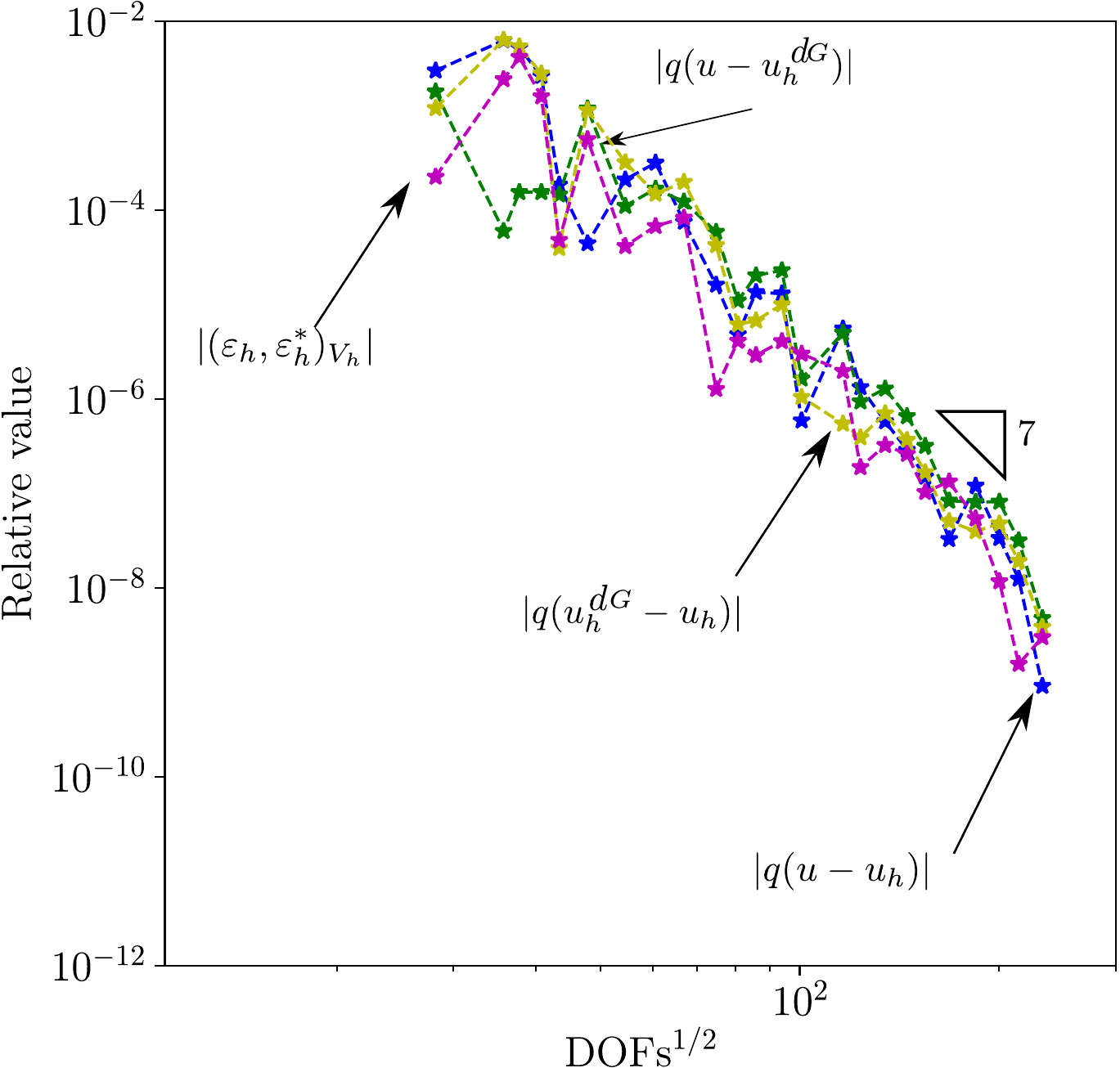}
        \caption{Pure advection ($\gamma = 0$)}
        \label{fig:ad_re_type_1_p3_0}
    \end{subfigure}
    \begin{subfigure}[b]{0.48\textwidth}
        \includegraphics[width=\textwidth]{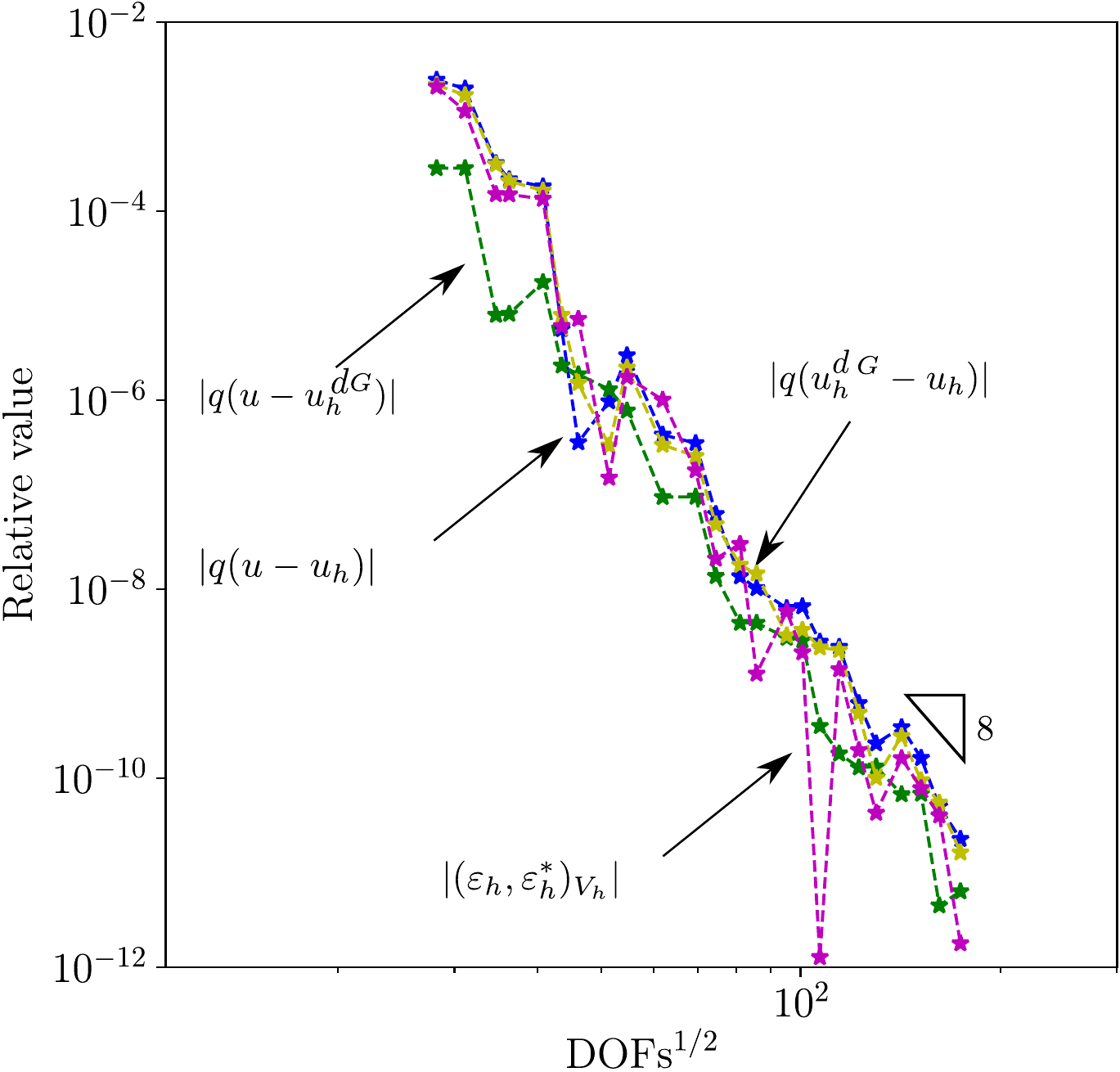}
        \caption{Reaction-dominant ($\gamma = 1000$)}
        \label{fig:ad_re_type_1_p3_100}
    \end{subfigure}
    \caption{Discrete relative error for upwinded (UP) scheme for
      pure advection with $p=3$ and $\Delta_p = 0$: goal-oriented error estimate $\big(\varepsilon_h\,,\, \varepsilon_h^\ast\big)_{V_h}$(see
      Section~\ref{sec:upper_estimates} for definitions).}
    \label{fig:ad_re_type_1_100_all}
  \end{figure}
  
Figure~\ref{fig:ad_re_type_1_0_all}  displays, for $p=1$, the evolution of the quantities $|q(u-u_h)|$, $|q(u-\udg)|$, $|q(\udg-u_h)|$, $|(\varepsilon_h\, , \, \varepsilon_h^\ast)_{V_h}|$ (scaled by $|q(u)|$) against the square root of the total number of DOFs. Figure~\ref{fig:ad_re_type_1_p1_0} corresponds to the pure advection case ($\gamma=0$) and Figure~\ref{fig:ad_re_type_1_p1_100} to the reaction dominant case ($\gamma=100$). Similarly, Figure~\ref{fig:ad_re_type_1_100_all} displays the same ratios for $p=3$. These figures are of particular relevance as they show that the GoA Assumptions~\ref{as:saturation_GO} and~\ref{as:saturation_GO_weak} are meaningful. Indeed, Figure~\ref{fig:ad_re_type_1_0_all} shows that the curve $|q(u-\udg)|$ remains below the curve $|q(u-u_h)|$ for both cases of $\gamma$. This bound implies that the goal-saturation Assumption~\ref{as:saturation_GO} is satisfied as in the pure diffusive example. While Figure~\ref{fig:ad_re_type_1_100_all} shows that this assumption is violated. Nevertheless, the weaker Assumption~\ref{as:saturation_GO_weak} is satisfied instead. Moreover, all the involved quantities share the same upper bound.\\ 
Finally, Figure~\ref{fig:ad_re_mesh_0_all} displays the resulting meshes at the eighteenth level of refinement using the energy error estimate (i.e., $\big(\varepsilon_h\, , \, \varepsilon_h\big)_{V_h}$) for pure advection~(a),  and reaction-dominant~(b) cases. Similarly, Figure~\ref{fig:ad_re_mesh_all} displays the meshes when using the estimate $\big(\varepsilon_h\, , \, \varepsilon_h^\ast\big)_{V_h}$. A comparison of these figures shows that the GoA estimates adjust the mesh refinement process according to the physical nature of the problem. In both cases, the energy error estimates attenuate the characteristic line that starts at $y=3/4$ and induces an interior layer.

\begin{figure}[h]
    \centering
    \begin{subfigure}[b]{0.4\textwidth}
        \includegraphics[width=\textwidth]{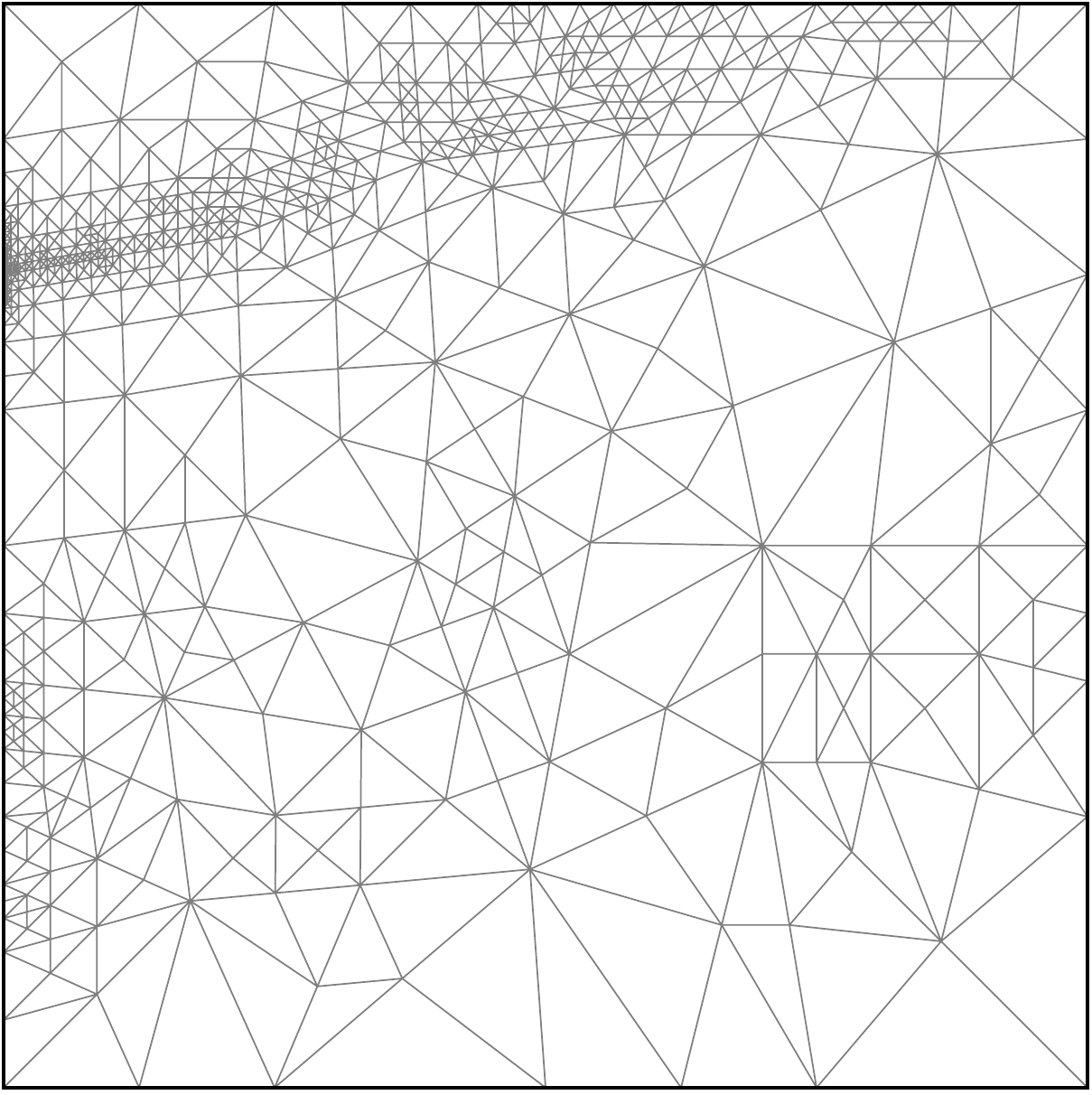}
        \caption{Pure advection $\gamma = 0$}
        \label{fig:ad_re_mesh_0_gamma_0}
    \end{subfigure}
   \hspace{0.5cm}
    \begin{subfigure}[b]{0.4\textwidth}
        \includegraphics[width=\textwidth]{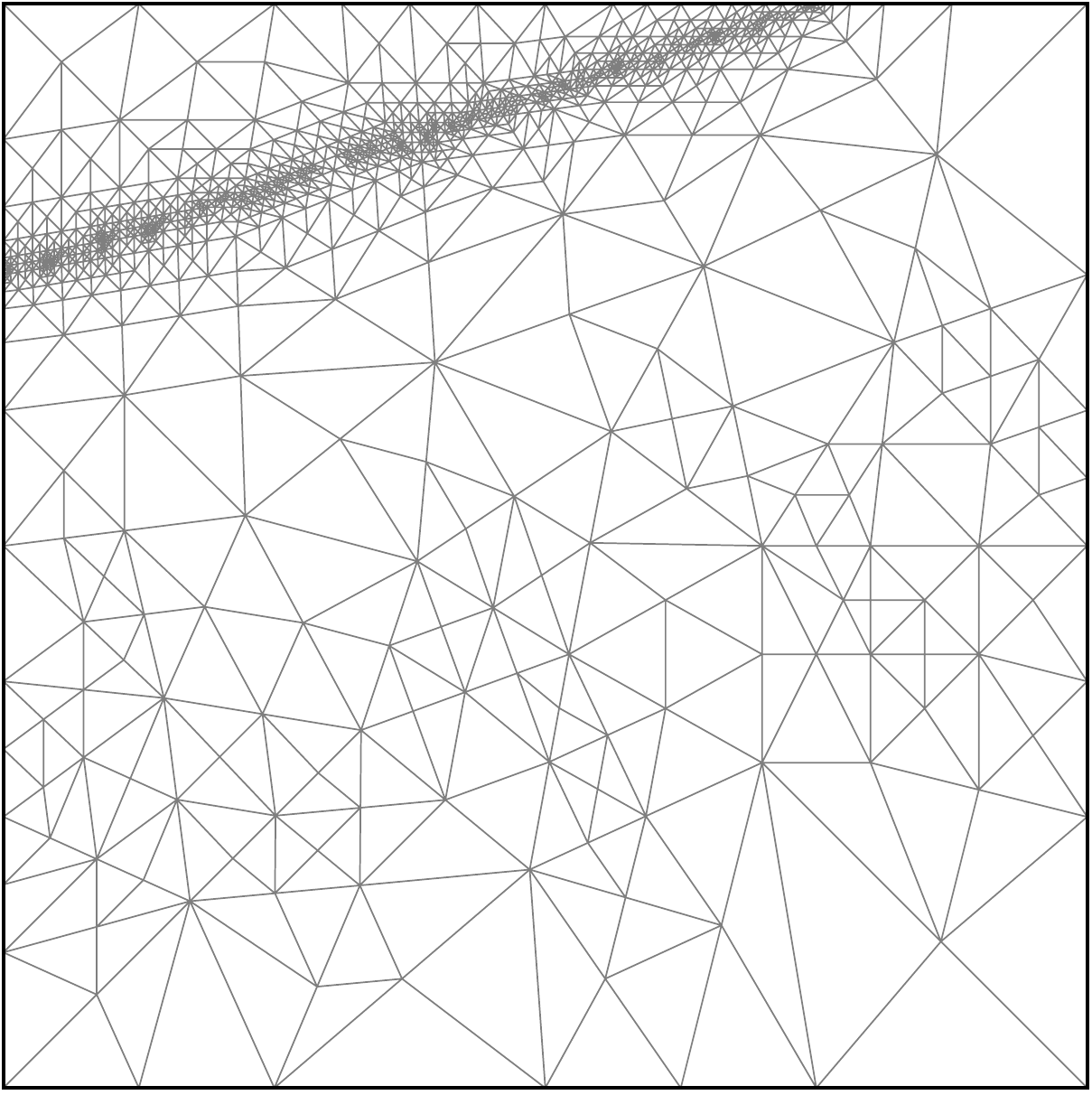}
        \caption{Reaction-dominant $\gamma = 1000$}
        \label{fig:ad_re_mesh_0_gamma_100}
    \end{subfigure}
    \caption{Resulting meshes after eighteenth levels of 
      refinements using the UP scheme with $p=1$ and the energy-based error
      estimate $\big(\varepsilon_h\,,\, \varepsilon_h\big)_{V_h}$ (see
      Section~\ref{sec:upper_estimates} for definitions). }
    \label{fig:ad_re_mesh_0_all}
\end{figure}
\begin{figure}[h]
    \centering
    \begin{subfigure}[b]{0.4\textwidth}
        \includegraphics[width=\textwidth]{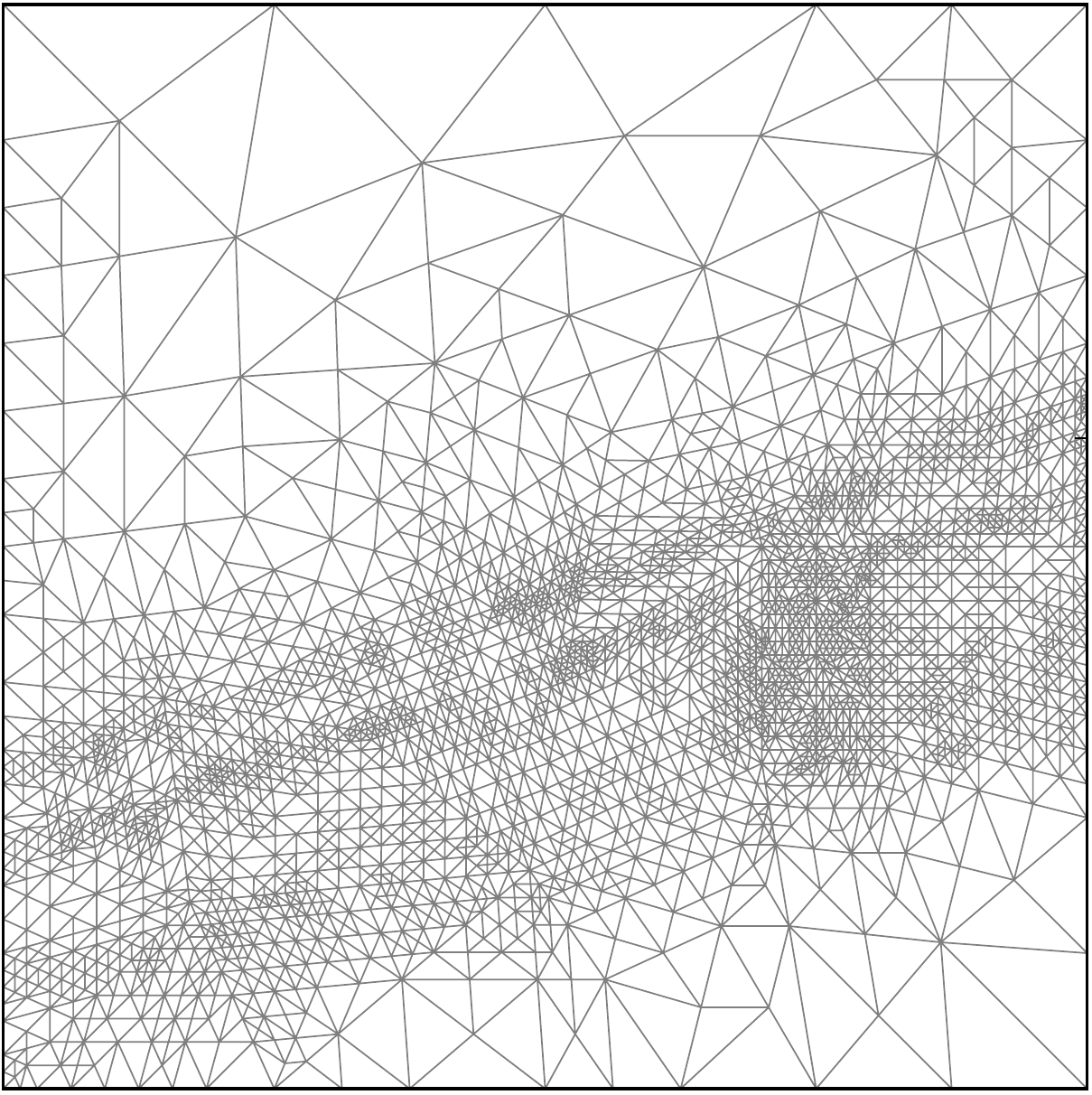}
        \caption{Pure advection ($\gamma = 0$)}
        \label{fig:ad_re_mesh_gamma_0}
    \end{subfigure}
\hspace{0.5cm}   
    \begin{subfigure}[b]{0.4\textwidth}
        \includegraphics[width=\textwidth]{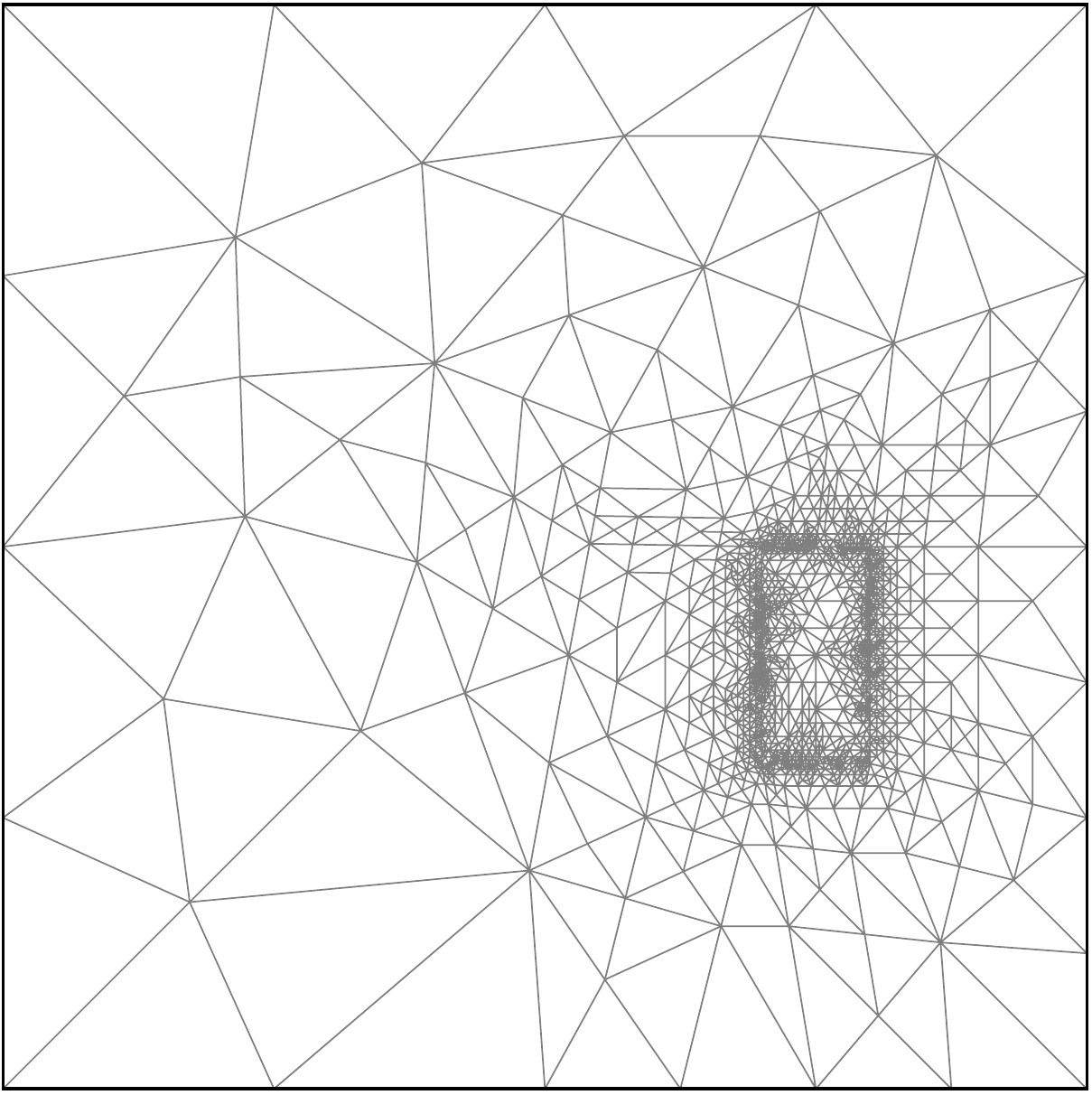}
        \caption{Reaction-dominant ($\gamma = 1000$)}
        \label{fig:ad_re_mesh_gamma_100}
    \end{subfigure}
    \caption{Resulting meshes after eighteenth levels of 
      refinements using the UP scheme with $p=1$ and the goal-oriented error estimate
      $E\big(\varepsilon_h\,,\, \varepsilon_h^\ast\big)$ (see
      Section~\ref{sec:upper_estimates} for definitions).}
    \label{fig:ad_re_mesh_all}
\end{figure}

\subsection{advection-diffusion-reaction problem with dominant advection}

As a final example, we consider a 3D advection-dominated advection-diffusion-reaction problem. Here we explore the performance of the GoA strategies by considering iterative solvers. We employ the iterative solver proposed in~\cite{ bank1989class} (cf.~\cite[Section 5.2.2]{rojas2019adaptive}) to solve the saddle-point problems of steps 1 and 2 in the GoA algorithm (see Section~\ref{sec:GoA_strategy}). We consider a sparse Cholesky (see~\cite{ chen2008algorithm}) as preconditioner for the Gramm matrix associated with the inner product~\eqref{eq:vh_norm}, and the LGMRES algorithm (see~\cite{ baker2005technique}) as preconditioner for the reduced Schur complement. To obtain the GoA estimator 3.A, we obtain $\vdg$ by solving problem~\eqref{eq:dg_goal} employing the LGMRES solver preconditioned with a sparse iLU factorization of the dG matrix. To obtain the GoA estimator 3.B, we obtain $\varepsilon_h^\ast$ by solving~\eqref{eq:error_est_goal} using the same sparse Cholesky preconditioner for the Gramm matrix. 

We particularize the definition of problem~\eqref{eq:scalar}. We denote by $\Vx = (x_1,x_2,x_3)$ the space variable, and set  the physical domain to the unit cube $\Omega = (0,1)^3$. We set $f=0$ as the source term, $\kappa = 10^{-6}$ as the diffusion coefficient, $\Vb = \big(-0.6 \sin(6\pi x_3)\, , \, 0.6 \cos(6\pi x_3) \, , \, 1\big)$ as the velocity field, and $\gamma = 10^{-6}$ as the reaction coefficient. We also define the Dirichlet datum as
\begin{equation*}
u_D = \left\{
\begin{array}{rl}
2+\tanh\big(1000\, c(0.1,(0.35,0.35)) \big)
+\tanh\big( 1000 \, c(0.1,(0.65,0.65)) \big), & \text{ if } x_3=0, \smallskip \\
0, & \text{ elsewhere },
\end{array}
\right.
\end{equation*}
where $c(r,(y_1,y_2)) := r^2-(x_1-y_1)^2-(x_2-y_2)^2$. Here, the boundary datum $u_D|_{x_3=0}$ corresponds to a smooth extension of a boundary source term which is different from zero in the interior of two circumferences, as shown in Figure~\ref{fig:3d_ini_mesh}. 
The solution of this problem is close to the solution of a homogeneous pure advection problem~\eqref{eq:advection_reaction} (i.e, with $\gamma = 0$ and $f=0$) considering the inflow datum $g^- = u_D|_{\Gamma^-}$. Such solution corresponds to two smoothed spirals starting at $x_3 = 0$, and arriving at $x_3 = 1$ (at the same starting position in the XY plane) after three periods of rotation (cf.~\cite{ rojas2019adaptive}). However, the homogeneous outflow boundary condition at $x_3=1$ induces strong boundary layers in the solution of the diffusion problem. Thus, this double spiral solution has strong interior and boundary layers. 

As in the previous examples, the energy-based adaptivity refines the solution to minimize its global error. In this particular problem, the energy estimate first refines around the inflow region, $x_3= 0$ and then follows the velocity field $\Vb$ towards the outflow region $x_3 = 1$. Figure~\ref{fig:3d_type0} shows a solution contour for the eighth level of refinement for a polynomial degree 1 and the initial mesh of Figure~\ref{fig:3d_ini_mesh}. The error estimate in this case reads $\big(\varepsilon_h\, , \, \varepsilon_h\big)_{V_h}$. 

We measure the quantity of interest (QoI) in the cube $\Omega_0 = (0.375\, , \, 0.5)\times (0.625\, , \, 0.75) \times (0.75\, , 0.875)$ as the domain for the QoI. The domain $\Omega_0$ only intersects the trayectory of the spiral startin at the circumference with center $(0.65, 0.65)$ and radius $0.1$. Figure~\ref{fig:3d_typeA} shows the solution contours at its sixth level of refinement guided by the GoA error estimate $\big(\varepsilon_h\, , \vdg-v_h^\ast\big)_{V_h}$. Figure~\ref{fig:3d_typeA} shows the solution contours at its fifth level of refinement guided by the GoA error estimate $\big(\varepsilon_h\, , \varepsilon_h^\ast\big)_{V_h}$. The problem setup is identical to the one described above. Both goal-oriented strategies guide similar mesh refinements of the spiral starting at the circumference of center $(0.6,0.6)$.
Figure~\ref{fig:3d_error} shows the relative error against the cubic root of the total number of DOFs. We show up to nine levels of refinements for polynomial orders $p=1,2$ and $\Delta_p = 0$ for the estimator $E\big(\varepsilon_h\, , \, \vdg-v_h^\ast \big) $ (left), and the estimator $E\big(\varepsilon_h\, , \, \varepsilon_h^\ast\big) $ (right). We compare these values with the respective estimator, and the relative error obtained considering the energy-norm based estimator $E\big(\varepsilon_h\, , \, \varepsilon_h\big)$.  Since the problem has no known analytical solution, we use an overkill simulation as a reference QoI. That is, we use the GoA strategy with polynomial degree $p=3$ and $\Delta_p = 0$ for the discrete spaces. We use the value $q\big(u_{ref}\big) = 1.36701319$ as the reference QoI obtained after ten levels of GoA refinements, requiring a total of $4,662,310$ DOFs to solve the final saddle-point problem. These figures show a significant error reduction in the QoI when comparing the results that the GoA estimates deliver against the energy-based ones. In this case, both GoA strategies deliver optimal convergence rates (see Equation~\eqref{eq:optimal_rate}). The estimator $E\big(\varepsilon_h\, , \, \varepsilon_h^\ast\big) $ has the advantage of requiring a reduced computational effort, since it is obtained by emplying the precomputed preconditioner. Indeed, we observe a reduction up to two orders of magnitude in the computational time required for solving the third problem, when compared with the computational time for obtaining $\vdg$.
\begin{figure}[h]
    \centering
    \begin{subfigure}[b]{0.48\textwidth}
    \includegraphics[width=\textwidth]{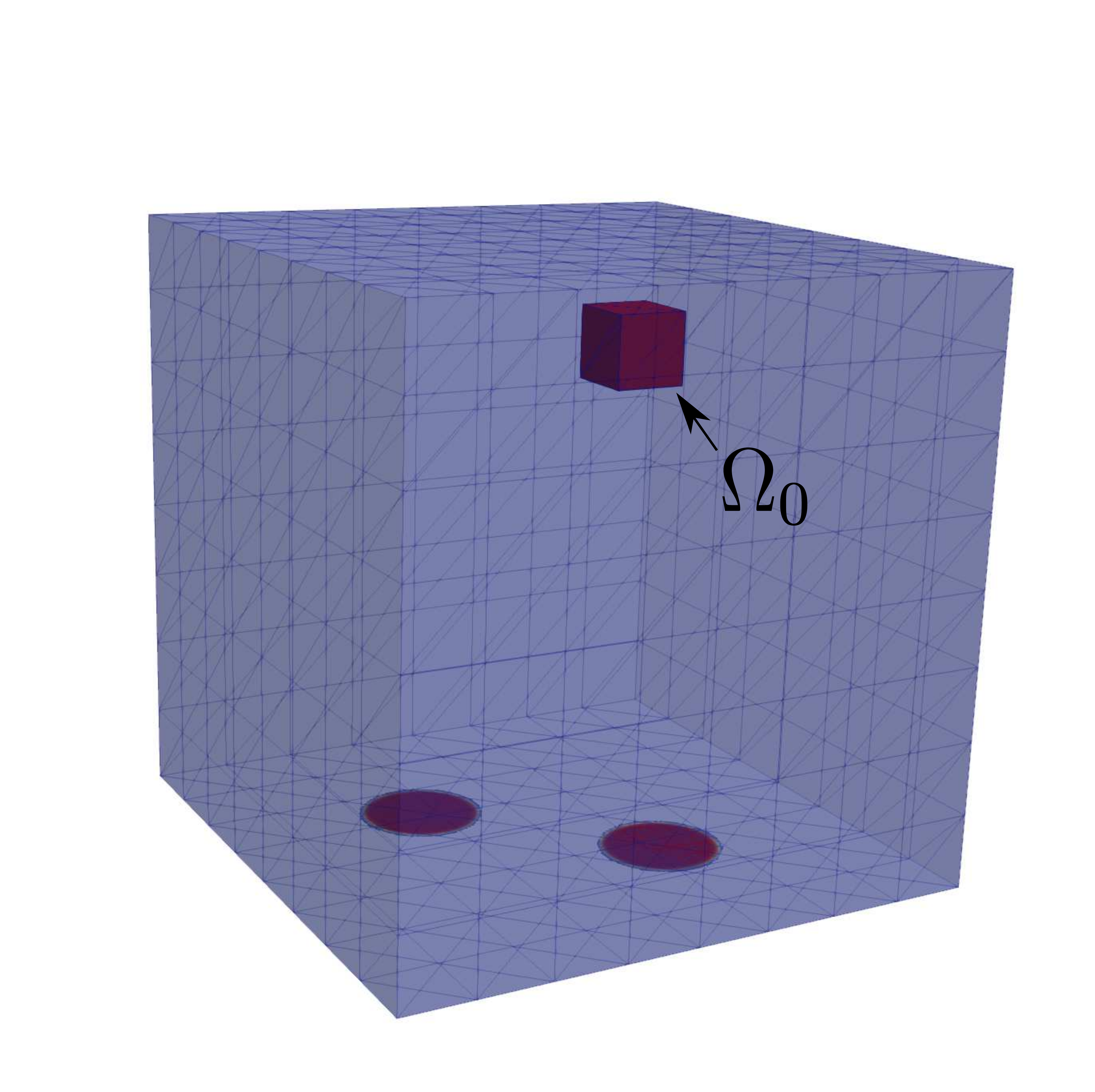}
        \caption{Initial mesh and Dirichlet boundary condition}
        \label{fig:3d_ini_mesh}
    \end{subfigure}
    \begin{subfigure}[b]{0.48\textwidth}
        \includegraphics[width=\textwidth]{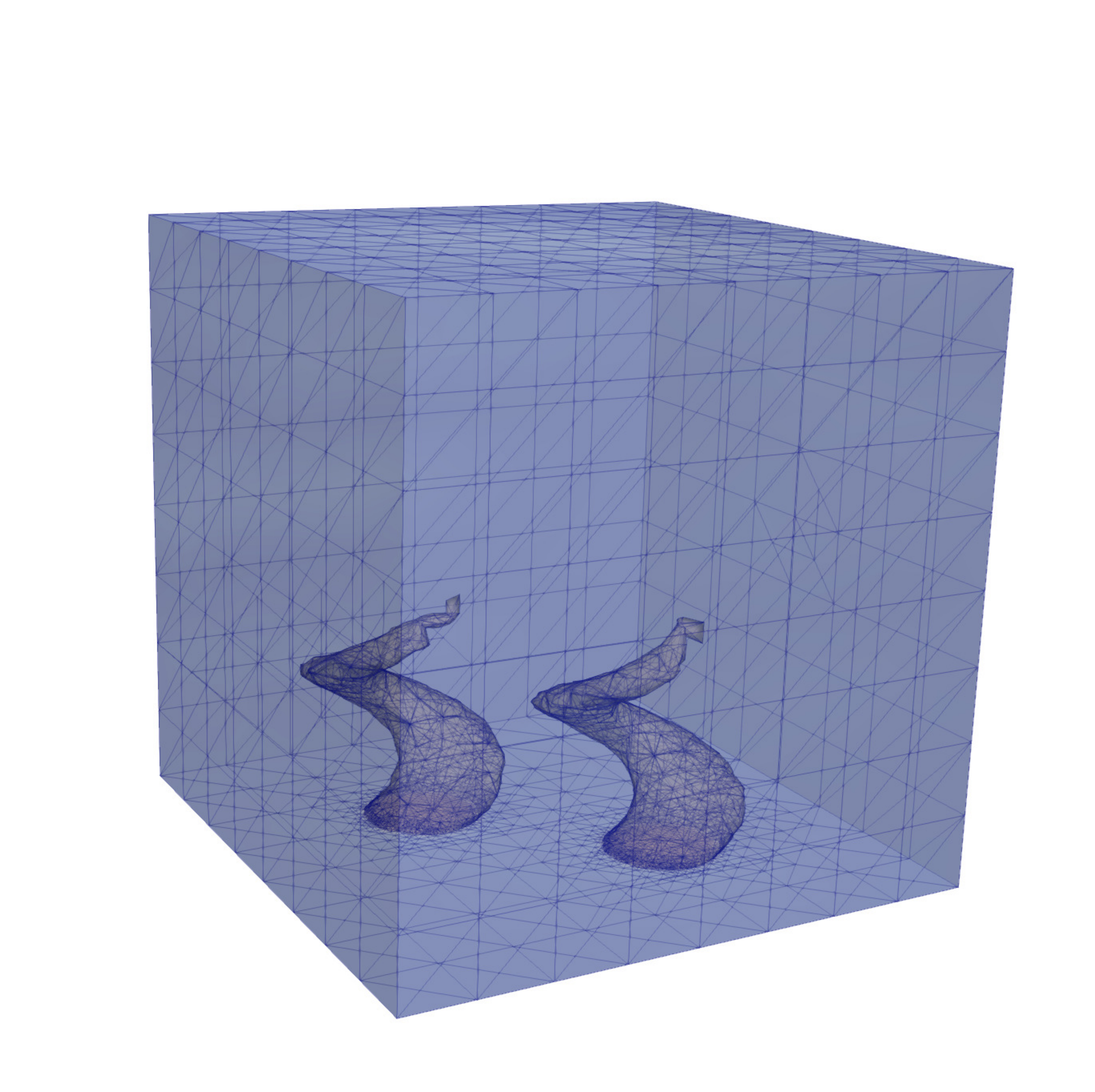}
        \caption{Solution contour for Estimator $E\big(\varepsilon_h,\varepsilon_h\big)$}
       \label{fig:3d_type0}
    \end{subfigure}
\begin{subfigure}[b]{0.48\textwidth}
    \includegraphics[width=\textwidth]{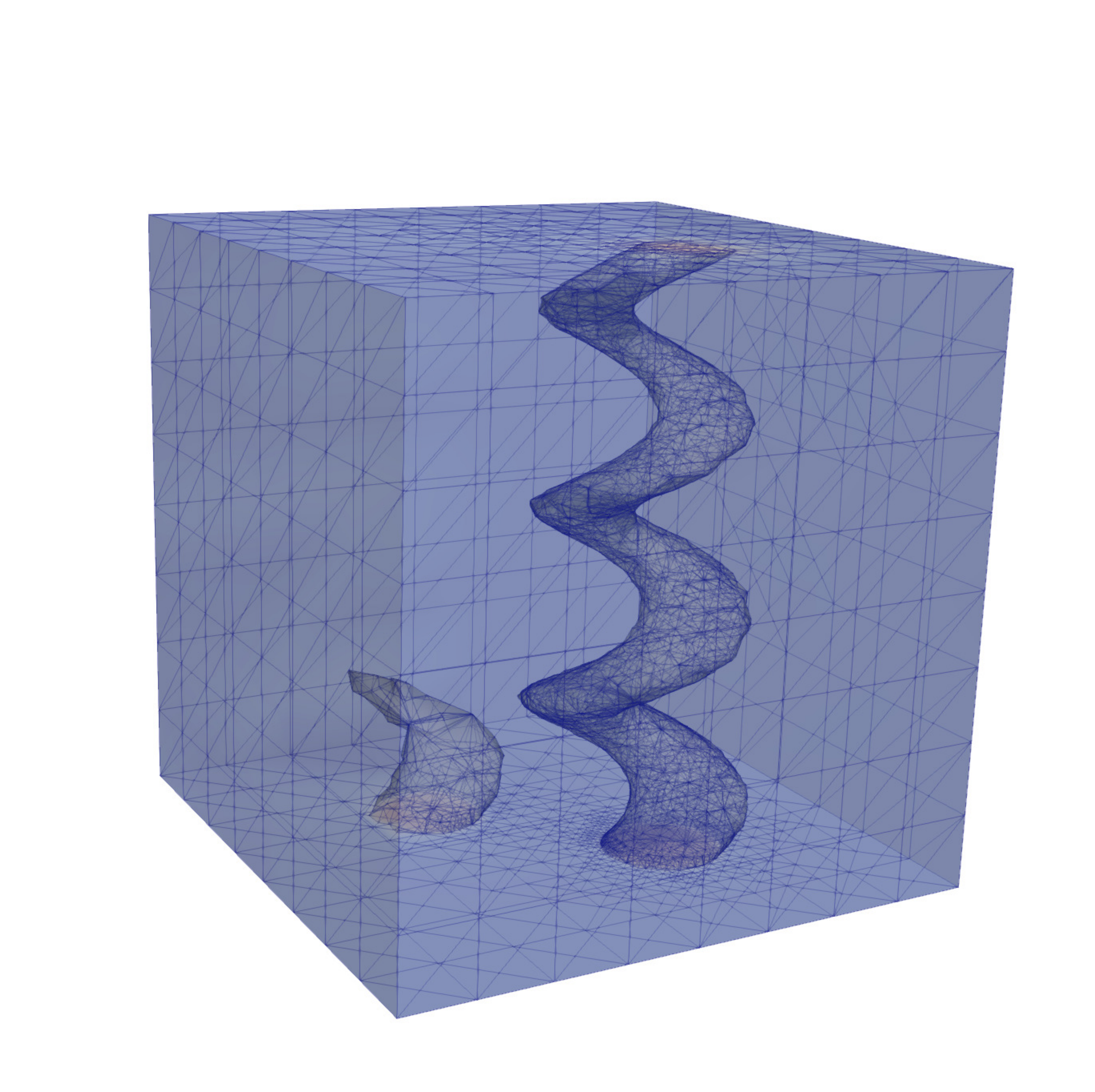}
        \caption{Solution contour for Estimator $E\big(\varepsilon_h,\vdg-v_h^\ast\big)$}
        \label{fig:3d_typeA}
    \end{subfigure}
    \begin{subfigure}[b]{0.48\textwidth}
        \includegraphics[width=\textwidth]{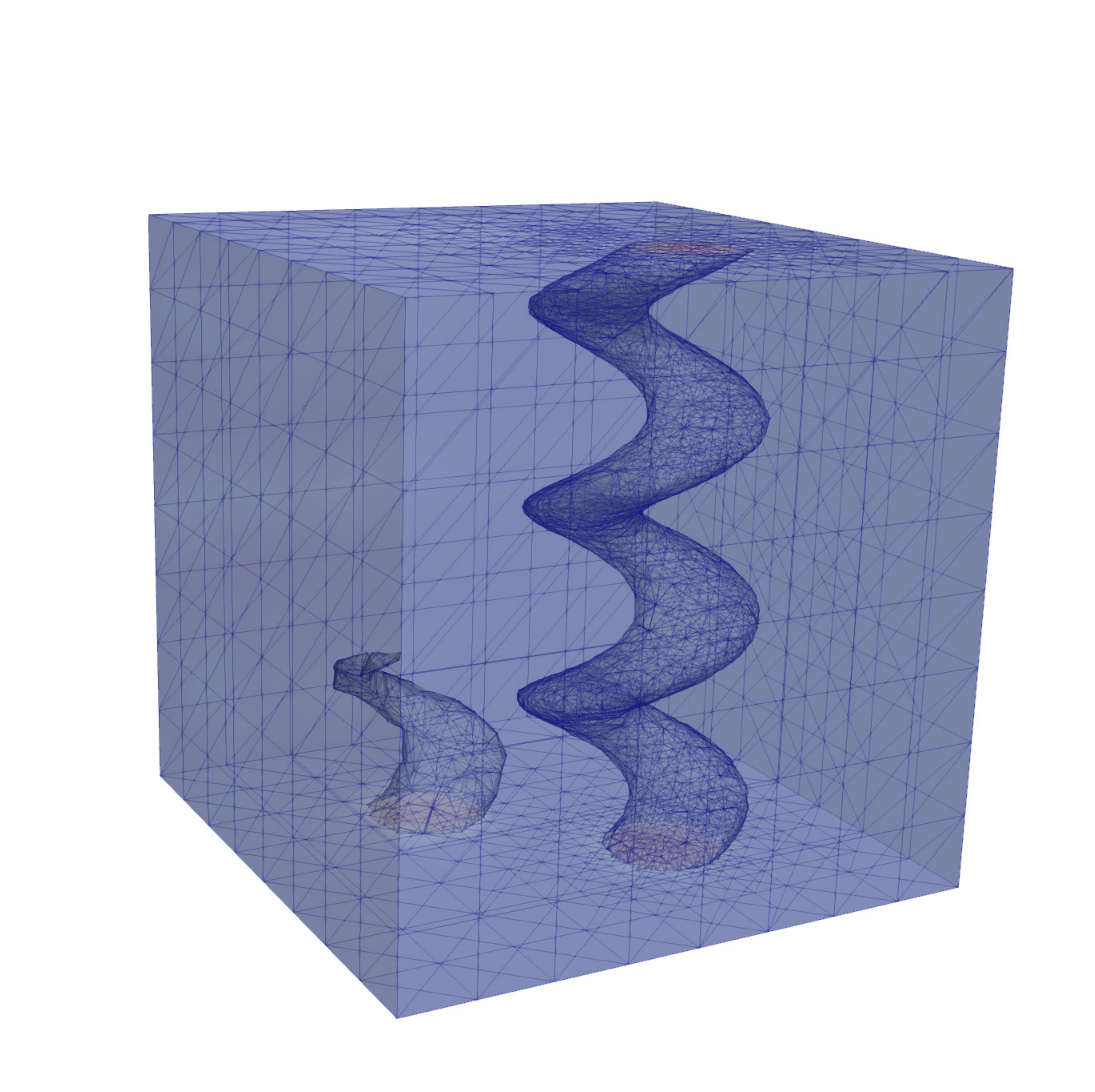}
        \caption{Solution contour for Estimator $E\big(\varepsilon_h,\varepsilon_h^\ast\big)$}   
       \label{fig:3d_typeB}
    \end{subfigure} 
  \caption{Resulting mesh contour comparison:~\eqref{fig:3d_ini_mesh} $\Omega_0$-conforming Initial mesh.~\eqref{fig:3d_type0} Seventh mesh of the adaptive alrgorithm driven by the energy-based estimator $E\big(\varepsilon_h, \varepsilon_h\big)$. ~\eqref{fig:3d_typeA}  Sixth mesh of the adaptive algorithm driven  by the goal-oriented estimator $E\big(\varepsilon_h, \vdg-v_h^\ast\big)$.~\eqref{fig:3d_typeB} Fifth mesh of the adaptive algorithm driven by the goal-oriented estimator $\big(\varepsilon_h, \varepsilon_h^\ast\big)$.}
\end{figure}

%%
%\begin{figure}[h]
%  \centering
%  \begin{subfigure}[b]{0.48\textwidth}
%    \includegraphics[width=\textwidth]{FENICS/3d/big_velocity/level_10_energy.png}
%    \caption{Energy-based error estimate \\
%      \centering{$E\big(\varepsilon_h\,,\, \varepsilon_h\big)$, DOFs = 449,841}}
%    \label{fig:3d_energy_fixed_level}
%  \end{subfigure}
%  % 
%  \begin{subfigure}[b]{0.48\textwidth}
%    \includegraphics[width=\textwidth]{FENICS/3d/big_velocity/level_6_goa.png}
%    \caption{Goal-oriented error estimate\\
%      \centering{$E\big(\varepsilon_h\,,\, \varepsilon_h^\ast\big)$, DOFs = 313,704 DOFs}}
%    \label{fig:3d_goa_fixed_level}
%  \end{subfigure}
%  \caption{Resulting mesh contour comparison: Adaptivity driven by the energy-based $E\big(\varepsilon_h, \varepsilon_h\big)$ (left) versus goal-oriented $E\big(\varepsilon_h, \varepsilon_h^\ast\big)$ (right) error estimates.}
%  \label{fig:3d_fixed_level}
%\end{figure}

\begin{figure}[h]
    \centering
    \begin{subfigure}[b]{0.48\textwidth}
    \includegraphics[width=\textwidth]{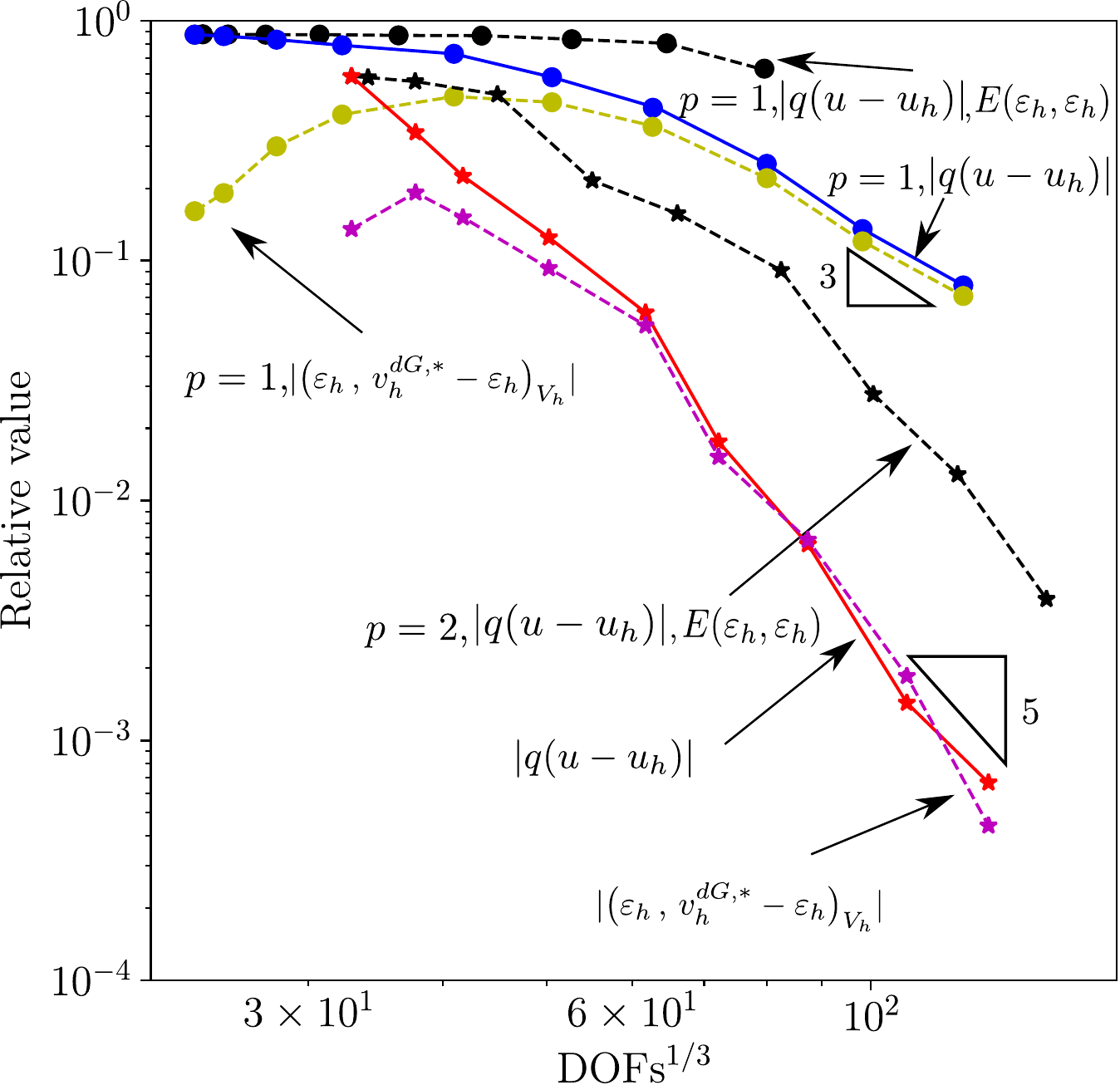}
    \caption{Estimator $E\big(\varepsilon_h\,,\, \vdg-\varepsilon_h\big)$}
        \label{fig:3d_estimatorA}
    \end{subfigure}
    \begin{subfigure}[b]{0.48\textwidth}
        \includegraphics[width=\textwidth]{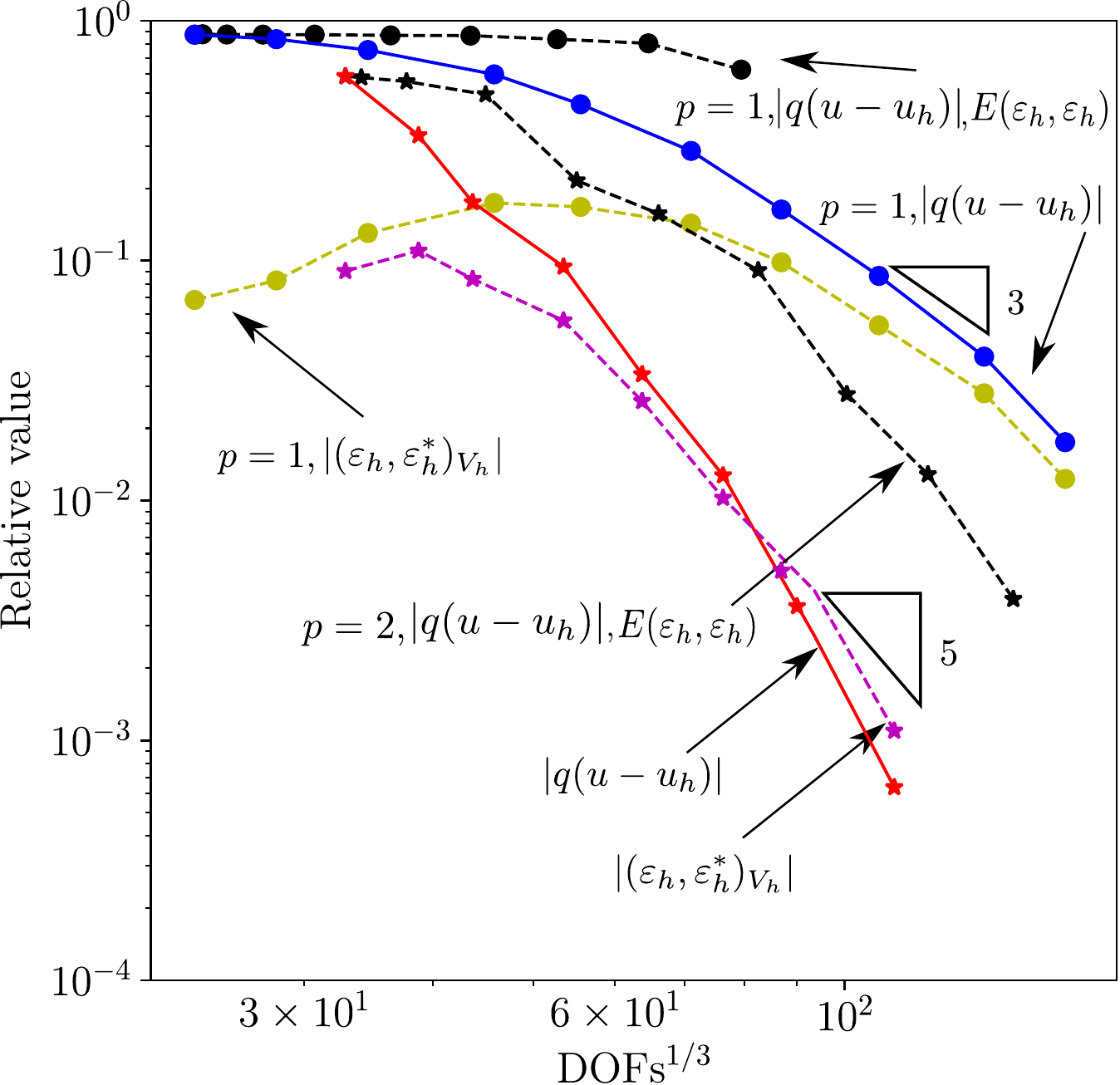}
    \caption{Estimator $E\big(\varepsilon_h\,,\, \varepsilon_h^\ast\big)$}
       \label{fig:3d_estimatorB}
    \end{subfigure}
    \caption{Relative error versus cubic root of the total number of degrees of freedom in the system (DOFs): Adaptivity driven by the energy-based $E\big(\varepsilon_h, \varepsilon_h\big)$ (left) versus goal-oriented $E\big(\varepsilon_h, \varepsilon_h^\ast\big)$ (right) error estimates. Data for nine levels of refinement, polynomial order $p=1,2$.}
   \label{fig:3d_error}
\end{figure}
%
%\begin{figure}[h]
%    \centering
%    
%    \includegraphics[width=0.5\textwidth]{figures/time2_test2-eps-converted-to.pdf}
%        
%    \label{fig:comp_time}
%    
%    \caption{Computational time required to solve the third GoA problem for $p=2$.}
%\end{figure}

\section{Contributions and future work}

\label{sec:conclusion}

In this paper, we present a new stabilized conforming goal-oriented adaptive method based on the stabilized finite element method introduced in~\cite{ rojas2019adaptive}. The adaptive framework automatically delivers stable solutions for both the direct and adjoint problems. Our process requires the resolution of a third problem. This allows us to compute an error estimate for the quantity of interest robustly. We present two alternative definitions of the third problem. The first definition solves an adjoint discontinuous Galerkin formulation. The second one solves a discrete Riesz representation problem, which inverts a symmetric positive definite matrix using fast approximations. Under a meaningful assumption to be satisfied by the reference discontinuous Galerkin formulation, we prove that both definitions provide an upper bound for the error in the quantity of interest. We validate the superiority of our goal-oriented strategy against an energy-based error estimate numerically for advection-diffusion-reaction problems, showing that both strategies can deliver optimal convergence rates for the error in the quantity of interest. 

Further studies are on the way to explore the performance of the method when applied to other challenging problems. For example, we will study the performance of metal-air electrochemical cells to improve battery storage capacity. We are also pursuing the extension of the methodology to time-dependent problems by considering space-time formulations, and its extension to non-linear goal functionals. 

\section*{Acknowledgements}

This publication was made possible in part by the CSIRO Professorial Chair in Computational Geoscience at Curtin University and the Deep Earth Imaging Enterprise Future Science Platforms of the Commonwealth Scientific Industrial Research Organisation, CSIRO, of Australia. At Curtin University, The Curtin Corrosion Centre, the Curtin Institute for Computation, and The Institute for Geoscience Research (TIGeR) kindly provide continuing support. Additional support was received from the European Union's Horizon 2020 research and innovation programme under the Marie Sklodowska-Curie grant agreement No 777778 (MATHROCKS). David Pardo has received funding from the European POCTEFA 2014-2020 Project PIXIL (EFA362/19) by the European Regional Development Fund (ERDF) through the Interreg V-A Spain-France-Andorra programme, the Project of the Spanish Ministry of Science and Innovation with reference PID2019-108111RB-I00 (FEDER/AEI), the BCAM "Severo Ochoa" accreditation of excellence (SEV-2017-0718), and the Basque Government through the BERC 2018-2021 program, the two Elkartek projects 3KIA (KK-2020/00049) and MATHEO (KK-2019-00085), the grant "Artificial Intelligence in BCAM number EXP. 2019/00432", and the Consolidated Research Group MATHMODE (IT1294-19) given by the Department of Education. Part of this work was carried over while the first author was invited by INRIA SERENA team in Paris. 
%The authors want to thank the anonymous referees for their comments that helped to improve the results of this paper.

\appendix

\bibliography{mybibfile}

\end{document}